%% file: r4.tex
\documentclass[11pt]{amsart}
\usepackage{latexsym, amssymb, amsmath, amsthm, color, mathabx}
\usepackage[shortlabels]{enumitem}
\usepackage{graphicx}

\usepackage{geometry}
\newtheorem{thm}{Theorem}
\newtheorem*{thm*}{Theorem}
\newtheorem*{question*}{Question}
\newtheorem{prop}{Proposition}[section]
\newtheorem{lem}[prop]{Lemma}

\newtheorem{cor}[prop]{Corollary}

\newtheorem{rem}[prop]{Remark}
\newtheorem{exe}{Example}[section]

\numberwithin{equation}{section}

\newcommand{\thistheoremname}{}
\newtheorem*{genericthm*}{\thistheoremname}
\newenvironment{namedthm*}[1]
  {\renewcommand{\thistheoremname}{#1}%
   \begin{genericthm*}}
  {\end{genericthm*}}

\DeclareMathOperator{\bD}{\mathbb{D}}

\DeclareMathOperator{\bM}{\mathbb{M}}
\DeclareMathOperator{\bR}{\mathbb{R}}

\DeclareMathOperator{\bW}{\mathbb{W}}
\DeclareMathOperator{\bZ}{\mathbb{Z}}

\DeclareMathOperator{\bfa}{\mathbf{a}}

\DeclareMathOperator{\bfk}{\mathbf{k}}

\DeclareMathOperator{\bfn}{\mathbf{n}}

\DeclareMathOperator{\bfp}{\mathbf{p}}

\DeclareMathOperator{\bfs}{\mathbf{s}}

\DeclareMathOperator{\bfw}{\mathbf{w}}

\DeclareMathOperator{\cA}{\mathcal{A}}
\DeclareMathOperator{\cB}{\mathcal{B}}
\DeclareMathOperator{\cC}{\mathcal{C}}
\DeclareMathOperator{\cD}{\mathcal{D}}
\DeclareMathOperator{\cE}{\mathcal{E}}

\DeclareMathOperator{\cG}{\mathcal{G}}
\DeclareMathOperator{\cH}{\mathcal{H}}
\DeclareMathOperator{\cI}{\mathcal{I}}
\DeclareMathOperator{\cJ}{\mathcal{J}}
\DeclareMathOperator{\cK}{\mathcal{K}}
\DeclareMathOperator{\cL}{\mathcal{L}}
\DeclareMathOperator{\cM}{\mathcal{M}}

\DeclareMathOperator{\cO}{\mathcal{O}}
\DeclareMathOperator{\cP}{\mathcal{P}}
\DeclareMathOperator{\cQ}{\mathcal{Q}}
\DeclareMathOperator{\cR}{\mathcal{R}}
\DeclareMathOperator{\cS}{\mathcal{S}}
\DeclareMathOperator{\cT}{\mathcal{T}}

\DeclareMathOperator{\cW}{\mathcal{W}}
\DeclareMathOperator{\cX}{\mathcal{X}}
\DeclareMathOperator{\cY}{\mathcal{Y}}
\DeclareMathOperator{\cZ}{\mathcal{Z}}

\title[High-dimensional families and three-dimensional energy surfaces]{High-dimensional families of holomorphic curves and three-dimensional energy surfaces}
\author{Rohil Prasad}
\email{rrprasad@berkeley.edu}

\begin{document}

\begin{abstract}
Let $H: \mathbb{R}^4 \to \bR$ be any smooth function. This article introduces some arguments for extracting dynamical information about the Hamiltonian flow of $H$ from high-dimensional families of closed holomorphic curves. We work in a very general setting, without imposing convexity or contact-type assumptions. 

For any compact regular level set $Y$, we prove that the Hamiltonian flow admits an infinite family of pairwise distinct, proper, compact invariant subsets whose union is dense in $Y$. This is a generalization of the Fish--Hofer theorem, which showed that $Y$ has at least one proper compact invariant subset. We then establish a global Le Calvez--Yoccoz property for almost every compact regular level set $Y$: any compact invariant subset containing all closed orbits is either equal to $Y$ or is not locally maximal. Next, we prove quantitative versions, in four dimensions, of the celebrated almost-existence theorem for Hamiltonian systems; such questions have been open for general Hamiltonians since the late $1980$s. We prove that almost every compact regular level set of $H$ contains at least two closed orbits, a sharp lower bound. Under explicit and $C^\infty$-generic conditions on $H$, we prove almost-existence of infinitely many closed orbits. 
\end{abstract}

\maketitle

\setcounter{tocdepth}{1}
\tableofcontents

\section{Introduction}\label{sec:intro}

\input{intro.tex}

\section{Preliminaries}\label{sec:prelim}

\input{prelim.tex}

\section{Dense existence of compact invariant sets}\label{sec:invariant_sets}

\input{invariant_sets.tex}

\section{Global Le Calvez--Yoccoz property}\label{sec:lcy}

\input{lcy.tex}

\section{Quantitative almost-existence}\label{sec:periodic_orbits}

\input{periodic_orbits.tex}

\appendix

\section{Existence results for closed curves}\label{sec:closed_curves}

\input{closed_curves.tex}

\section{Generic $4$-dimensional Hamiltonians}\label{sec:generic}

\input{generic.tex}

\bibliographystyle{alpha}
\bibliography{r4.bib}

\end{document}

%% file: intro.tex
\subsection{Background and statement of main results}\label{subsec:results}

Hamilton's equation 
\begin{equation}\label{eq:hamilton}\Omega(X_H, -) = -dH\end{equation}
associates to any smooth function $H: \bR^{2n} \to \bR$ a so-called Hamiltonian vector field $X_H$. The flow of $X_H$ preserves the symplectic form $\Omega$. The dynamical behavior of Hamiltonian flows has been profitably studied over many years from many different perspectives. This article studies the invariant sets and closed orbits of Hamiltonians $H: \bR^4 \to \bR$ from the perspective of symplectic geometry. Most results require $H$ to satisfy convexity or ``contact-type'' assumptions. We will not make any such assumptions. This presents several difficulties. 

The most fundamental issue is that arbitrary Hamiltonians appear to be much worse-behaved than contact-type Hamiltonians. For example, Viterbo \cite{Viterbo87} proved that any contact-type level set of a Hamiltonian carries a closed orbit, but non-contact-type examples due to Ginzburg \cite{Ginzburg95}, Ginzburg--G\"urel \cite{GG03}, and Herman \cite{Herman99} have no closed orbits. Another issue is that few tools exist to study the dynamics of arbitrary Hamiltonians. Symplectic field theory (SFT) and its variants are not well-defined for non-contact-type Hamiltonians. Floer theory, e.g. symplectic homology, can be defined, but appears to be weaker without a contact-type assumption\footnote{Let $Y \subset \mathbb{R}^{2n}$ be a level set bounding a compact domain $U$. The Floer--Hofer symplectic homology of $U$ \cite{FH94} is only known to detect dynamical features of $Y$, e.g. closed orbits, when $Y$ is contact-type. Other variants are not well-defined if $Y$ is not contact-type.}. 

One of our results, Theorem~\ref{thm:main}, is a very general existence result for proper, compact, $X_H$-invariant subsets in level sets of $H$. Another result, Theorem~\ref{thm:main3}, gives a sharp quantitative refinement, in four dimensions, of the celebrated almost-existence theorem for closed orbits. To compensate for the absence of SFT and Floer theory, we develop some new arguments to extract dynamical information from simpler invariants: moduli spaces of closed $J$-holomorphic curves. We now state and discuss our results and their background in detail. Afterwards, we will give sketches of the proofs. 

\subsubsection{Dense existence of compact invariant sets} It follows from \eqref{eq:hamilton} that $dH(X_H) \equiv 0$, so the function $H$ is invariant under the flow of $X_H$. Therefore, each level set of $H$ is invariant under the flow of $X_H$. Herman asked at the $1998$ ICM \cite{Herman98} whether the level sets themselves contain compact $X_H$-invariant subsets. Here is a paraphrased version of his question. 

\begin{namedthm*}{Herman's Question}
Fix a smooth function $H: \bR^{2n} \to \bR$ and a compact regular level set $Y$. Does there exist a proper, compact, $X_H$-invariant subset $\Lambda \subset Y$?  
\end{namedthm*}

Herman's question is elegant but very difficult, since it is posed for arbitrary Hamiltonians without any contact-type assumption. A groundbreaking work by Fish--Hofer \cite{FH23} resolved Herman's question in the case $n = 2$.

\begin{namedthm*}{Fish--Hofer Theorem}[{\cite[Theorem $1$]{FH23}}]
Let $H: \bR^4 \to \bR$ be a smooth function and let $Y$ be a compact regular level set of $H$. Then there exists a proper, compact, $X_H$-invariant subset $\Lambda \subset Y$. 
\end{namedthm*}

Our first main theorem is a generalization of the Fish--Hofer theorem. 

\begin{thm}\label{thm:main}
Let $H: \bR^4 \to \bR$ be a smooth function and let $Y$ be a compact, connected, regular level set of $H$. Then $Y$ contains an infinite family of pairwise distinct, proper, compact $X_H$-invariant subsets whose union is dense in $Y$.
\end{thm}

The connectedness assumption can be removed; see Remark~\ref{rem:main}. In the context of symplectic dynamics, Theorem~\ref{thm:main} is a substantial and possibly unexpected generalization. To illustrate this by way of analogy, we review detection results for closed orbits in contact-type level sets in $\bR^4$. For contact-type $Y$, it has been known since the $1980$s, after Weinstein \cite{Weinstein78}, Rabinowitz \cite{Rabinowitz78}, and Viterbo \cite{Viterbo87} that $Y$ contains a closed orbit of $X_H$. This was improved to two closed orbits by Cristofaro-Gardiner--Hutchings \cite{CGH16}\footnote{See \cite{GHHM13} for an alternate approach to this result, for star-shaped $Y$, using contact homology.}. It is now known, after Hofer--Wysocki--Zehnder \cite{HWZ98, HWZ03}, Cristofaro-Gardiner--Hutchings--Pomerleano \cite{CGHP19}, and a recent tour-de-force by Cristofaro-Gardiner--Hryniewicz--Hutchings--Liu \cite{CGHHL23}, that any star-shaped $Y$ has either two or infinitely many closed orbits. The works \cite{CGHP19, CGHHL23} extend to the contact-type case given a torsion assumption on the Chern class; work of Colin--Dehornoy--Rechtman \cite{CDR23} drops the torsion assumption but requires the Hamiltonian flow to be nondegenerate. Irie \cite{Irie15} proved that for $C^\infty$-generic $Y$, the closed orbits are dense. This represents over $30$ years of work, with accelerated progress in the last decade, to go from one closed orbit to infinitely many closed orbits or to dense closed orbits. Theorem~\ref{thm:main} proceeds straight from one proper, compact, invariant subset in an arbitrary level set $Y$ to a simultaneously infinite and dense family of proper, compact, invariant subsets. 

In \S\ref{subsec:proof_outlines}, we outline the new ideas behind Theorem~\ref{thm:main} that are not present in \cite{FH23} or other previous works. Before stating our other main results, we make some additional remarks.

\begin{rem}\label{rem:main}
\normalfont
Theorem~\ref{thm:main} generalizes to disconnected level sets. Fix $H: \bR^4 \to \bR$ and a compact regular level set $Y$. Then, each connected component $Y_* \subset Y$ contains an infinite family of pairwise distinct, compact, proper $X_H$-invariant subsets with dense union in $Y_*$. Here is a proof. Since $Y_*$ is compact, orientable, and null-homologous, there exists a smooth function $H_*: \bR^4 \to \bR$ for which $0$ is a regular value and $Y_* = H_*^{-1}(0)$. The Hamiltonian vector fields $X_H$ and $X_{H_*}$ coincide on $Y_*$ after rescaling the former by a nowhere zero smooth function. Thus, their flows have the same invariant sets. Apply Theorem~\ref{thm:main} to $H_*$. 
\end{rem}

\begin{rem}\label{rem:katok}
\normalfont
A celebrated construction by Katok \cite{Katok73} shows that Hamiltonian flows on $\bR^{2n}$ are remarkably flexible. His results imply that the conclusions of Theorem~\ref{thm:main} fail to hold when invariant subsets are replaced by other natural dynamical objects: closed orbits, minimal subsets\footnote{A compact invariant subset in which every orbit is dense.}, and ergodic measures\footnote{An invariant probability measure that assigns each invariant subset a probability of $0$ or $1$.}. Katok constructed a star-shaped level set in $\bR^4$ whose Hamiltonian flow has exactly two closed orbits and exactly three ergodic invariant measures: the Dirac measures on the closed orbits and the volume measure. The two closed orbits are the only minimal subsets of the flow. Thus, closed orbits, minimal subsets, and ergodic measures can be quite simple. Theorem~\ref{thm:main} shows that the compact invariant subsets are always quite complex and spread out throughout the level set.
\end{rem}

\begin{rem}
\normalfont
The conclusions of Theorem~\ref{thm:main}, the so-called ``dense existence of compact invariant sets'', appear to be an emergent phenomenon in symplectic dynamics, at least in low dimensions. Earlier this year, in joint work with Cristofaro-Gardiner \cite{CGP24}, an analogue of Theorem~\ref{thm:main} for area-preserving surface diffeomorphisms and three-dimensional Reeb flows was proved. Theorem~\ref{thm:main} was announced in that article. The tools and arguments in \cite{CGP24} are quite different from the present work. We defer to \S\ref{subsec:proof_outlines} for a more detailed comparison. 
\end{rem}

\subsubsection{A global Le Calvez--Yoccoz property} A remarkable series of works in the late $1990$s and early $2000$s by Le Calvez--Yoccoz \cite{LCY97}, Franks \cite{Franks99}, and Salazar \cite{Salazar06} established the following result for invariant sets of homeomorphisms of the $2$-sphere. Recall that a compact invariant set $\Lambda$ of a homeomorphism or flow is \emph{locally maximal} if any sufficiently Hausdorff-close invariant set $\Lambda'$ must be contained in $\Lambda$. For any area-preserving homeomorphism $\phi$ of $\mathbb{S}^2$, it was proved that any compact $\phi$-invariant set $\Lambda$ containing all closed orbits is either equal to $\mathbb{S}^2$ or is not locally maximal. We call such a result a ``global Le Calvez--Yoccoz property'', since it produces invariant subsets near any invariant subset containing all closed orbits, which could occupy a significant part of $\mathbb{S}^2$. These works rely on fixed point theory; it is unclear to us how to extend their arguments to flows. 

Our next theorem is a global Le Calvez--Yoccoz property for Hamiltonian flows on $\bR^4$. To state the result, we need to fix some notation. For any smooth function $H: \bR^4 \to \bR$, let $\cR_c(H)$ denote the set of regular values $s \in \bR$ such that $H^{-1}(s)$ is compact. For any $s \in \cR_c(H)$, let $\cP(s) \subseteq H^{-1}(s)$ denote the union of the closed orbits of $X_H$ lying in $H^{-1}(s)$. 

\begin{thm}\label{thm:main2}
Let $H: \bR^4 \to \bR$ be a smooth function. Then there exists a subset $\cQ \subseteq \cR_c(H)$ of full measure such that the following holds for any $s \in \cQ$. Any compact $X_H$-invariant subset $\Lambda \subseteq H^{-1}(s)$ containing $\cP(s)$ is either equal to $H^{-1}(s)$ or is not locally maximal in $H^{-1}(s)$. 
\end{thm}

The global Le Calvez--Yoccoz property is quite powerful. By an elementary topological argument\footnote{See the arguments in \cite[\S$2.4$]{CGP24}.}, the global Le Calvez--Yoccoz property implies the dense existence of compact invariant sets. So, Theorem~\ref{thm:main2} can be regarded as a refinement of Theorem~\ref{thm:main}, but only for almost every compact regular level set. It is unclear to us whether every compact regular level set satisfies the global Le Calvez--Yoccoz property. Theorem~\ref{thm:main2} also has an interesting application towards detecting closed orbits; see Theorem~\ref{thm:main4} below. 

\begin{rem} 
\normalfont
Global Le Calvez--Yoccoz properties for monotone area-preserving surface diffeomorphisms and Reeb flows on torsion contact $3$-manifolds are proved in \cite{CGP24}. Ginzburg--G\"urel \cite{GG18} proved a related result. They showed that, for any Hamiltonian diffeomorphism of $\mathbb{CP}^{n}$ with finitely many periodic points, no periodic point is locally maximal. Cineli--Ginzburg--G\"urel--Mazzucchelli \cite{CGGM23} recently proved a contact analogue of \cite{GG18}. They showed that, for any nondegenerate and dynamically convex star-shaped hypersurface $Y \subset \bR^{2n}$, no closed orbit is locally maximal in $Y$. 
\end{rem}

\subsubsection{Quantitative almost-existence} In the late $1980$s, Hofer--Zehnder \cite{HZ87} discovered the existence of closed orbits near any compact regular level set of a Hamiltonian. After some refinements by Rabinowitz \cite{Rabinowitz87} and Struwe \cite{Struwe90}, this became known as the ``almost-existence'' theorem: 

\begin{namedthm*}{Almost-existence theorem}[\cite{HZ87, Rabinowitz87, Struwe90}] 
Let $H: \bR^{2n} \to \bR$ be a smooth function. Then there exists a subset $\cQ \subseteq \cR_c(H)$ of full measure such that for any $s \in \cQ$, the level set $H^{-1}(s)$ contains a closed orbit of $X_H$. 
\end{namedthm*}

The almost-existence theorem holds for any Hamiltonian $H$ without convexity or contact-type assumptions. Since its introduction, the almost-existence theorem has been generalized to many other symplectic manifolds, sometimes with restrictions on the Hamiltonian. There have been many significant contributions from many authors. Symplectic methods such as Floer homology and symplectic capacities have played a key role. We mention, in no particular order, works of Hofer--Viterbo \cite{HV92}, Cieliebak--Ginzburg--Kerman \cite{CGK04}, Ginzburg--G\"urel \cite{GG04}, Biran--Polterovich--Salamon \cite{BPS03}, McDuff--Slimowitz \cite{MS01}, Macarini \cite{Macarini04}, Schlenk \cite{Schlenk06}, Macarini--Schlenk \cite{MS05}, Frauenfelder--Schlenk \cite{FS07}, Lu \cite{Lu98, Lu98correction}, and Fish--Hofer \cite{FH22}. This list should be regarded only as a sample of interesting works in this area; see \cite{Ginzburg05, GG09} for more thorough surveys. 

Beyond proving that closed orbits exist, establishing the optimal \emph{multiplicity} of closed orbits (one, two, infinitely many, etc.) of a Hamiltonian flow is a central problem in Hamiltonian dynamics. The multiplicity problem has seen an enormous amount of interest and progress in several different directions. First, a long-standing conjecture asserts that there are at least $n$ closed orbits in any convex level set in $\bR^{2n}$. We refer to \cite{DLLQW24} for a comprehensive survey of results on this question. Second, as mentioned above near Theorem~\ref{thm:main}, the multiplicity problem for contact-type level sets in $\bR^4$ has been of great interest since pioneering work of Hofer--Wysocki--Zehnder \cite{HWZ98, HWZ03}. Third, very strong multiplicity results have been proved for level sets near extrema of Hamiltonians. The celebrated Weinstein--Moser theorem \cite{Weinstein73, Moser76} shows that any level set of $H: \bR^{2n} \to \bR$ that is near a nondegenerate minimum of $H$ contains at least $n$ closed orbits. Generalizations to other types of extrema were proved Kerman \cite{Kerman99}. A very general Weinstein--Moser theorem, replacing $\bR^{2n}$ by other symplectic manifolds, was proved by Ginzburg--G\"urel \cite{GG09}; some further contributions were made by Usher \cite{Usher09}. Finally, there are several multiplicity results for kinetic Hamiltonians on magnetic cotangent bundles. See the works of Arnold \cite{Arnold88}, Ginzburg \cite{Ginzburg87, Ginzburg96}, Contreras \cite{Contreras06}, Abbondandolo--Macarini--Paternain \cite{AMP15}, Asselle--Benedetti \cite{AB16}, and Abbondandolo--Macarini--Mazzucchelli--Paternain \cite{AMMP17}.

Thus, there is a significant gap between the generality of the almost-existence theorem and the known multiplicity results for Hamiltonian flows. The almost-existence theorem holds for almost every compact regular level set of any $H: \bR^{2n} \to \bR$; more generally the domain $\bR^{2n}$ can be replaced by any of a large family of symplectic manifolds. On the other hand, all known multiplicity results require that either (i) the level set is convex or contact-type, (ii) the level set is near some kind of extremum, or (iii) the Hamiltonian is of a specific form, e.g. the kinetic Hamiltonian. This gap has yet to be bridged. For example, the lower bound of one closed orbit in the original almost-existence theorem has held without any improvements since the late $1980$s. 

Our next theorem makes some progress in this direction. We prove an almost-existence theorem with the optimal multiplicity for any Hamiltonian $H: \bR^4 \to \bR$. 

\begin{thm}\label{thm:main3}
Let $H: \bR^4 \to \bR$ be a smooth function. Then there exists a subset $\cQ \subseteq \cR_c(H)$ of full measure such that for any $s \in \cQ$, the level set $H^{-1}(s)$ contains at least two closed orbits of $X_H$. 
\end{thm}

The multiplicity in Theorem~\ref{thm:main3} is seen to be optimal by example. There exist smooth four-dimensional Hamiltonians for which every regular level set contains exactly two closed orbits\footnote{The convex four-dimensional Hamiltonian $H(x) = (|x_1|^2 + |x_2|^2)/a + (|x_3|^2 + |x_4|^2)/b$, where $a$ and $b$ are positive and rationally independent, is one such example.}. Our proof of Theorem~\ref{thm:main3} combines the ideas behind Theorems~\ref{thm:main} and \ref{thm:main2} with a careful quantitative argument; we will give a sketch in \S\ref{subsec:proof_outlines}.

\begin{rem}
\normalfont
As a corollary of Theorem~\ref{thm:main3}, any contact-type level set in $\bR^4$ contains at least two closed orbits. This has been known since work of Cristofaro-Gardiner--Hutchings \cite{CGH16} showing that any Reeb flow on a closed $3$-manifold has at least two closed orbits. Thus, Theorem~\ref{thm:main3} gives a partial generalization of \cite{CGH16} beyond the contact case. 
\end{rem}

Our last main theorem asserts that, under some additional conditions on $H$, one has almost-existence of infinitely many closed orbits. 

\begin{thm}\label{thm:main4}
Let $H: \bR^4 \to \bR$ be any smooth function such that for almost every $s \in \cR_c(H)$, any closed orbit $\gamma \subset H^{-1}(s)$ is either (i) hyperbolic or (ii) elliptic and Moser stable. Then, there exists a full measure subset $\cQ \subseteq \cR_c(H)$ such that for any $s \in \cQ$, the level set $H^{-1}(s)$ contains infinitely many closed orbits of $X_H$.
\end{thm}

We define hyperbolic, elliptic, and Moser stable closed orbits in \S\ref{sec:periodic_orbits}. Theorem~\ref{thm:main4} follows from an elementary argument using Theorem~\ref{thm:main2}. As we will explain in Appendix~\ref{sec:generic}, the set of $H$ satisfying the conditions of Theorem~\ref{thm:main4} is Baire-generic in $C^\infty(\bR^4)$. So, we have the following corollary.

\begin{cor}\label{cor:main4}
There exists a Baire-generic subset $\cG \subseteq C^\infty(\bR^4)$ with the following property. Any $H \in \cG$ admits a full measure subset $\cQ \subseteq \cR_c(H)$ such that for any $s \in \cQ$, the level set $H^{-1}(s)$ contains infinitely many closed orbits of $X_H$.
\end{cor}

\subsection{Comments on the proofs}\label{subsec:proof_outlines}

We outline the proofs of the main theorems. The outline of Theorem~\ref{thm:main} is a bit long, but provides the necessary context to give much more concise summaries of the other main results. 

\subsubsection{Dense existence of compact invariant sets} We start with Theorem~\ref{thm:main}. Fix a smooth function $H: \bR^4 \to \bR$ and a compact regular level set $Y$. To explain our method and highlight some challenges that we overcome, we review previous work on detecting closed orbits or invariant sets in $Y$. First, assume that $Y$ is contact-type (e.g. convex or star-shaped). Compactify $\bR^4$ to $\mathbb{CP}^2$ by adding a divisor at infinity. For any tame almost-complex structure $J$ and any pair of points $w_\pm \in \mathbb{CP}^2$, there exists a degree $1$ $J$-holomorphic sphere in $\mathbb{CP}^2$ passing through $w_\pm$. Place $w_+$ and $w_-$ on opposite sides so that the sphere crosses the hypersurface $Y$. Then, we perform a neck stretching procedure around $Y$. Since $Y$ is contact-type, the crossing sphere satisfies uniform energy bounds as it is strrethed. By the SFT compactness theorem \cite{BEHWZ03}, the sphere breaks into a holomorphic building. Thus, $Y$ must contain a closed orbit, since each building level is asymptotic to a non-empty union of closed orbits. This proof is illustrated in Figure~\ref{fig:sft_stretching}.

\begin{figure}[h]
\centering
\includegraphics[width=0.5\textwidth]{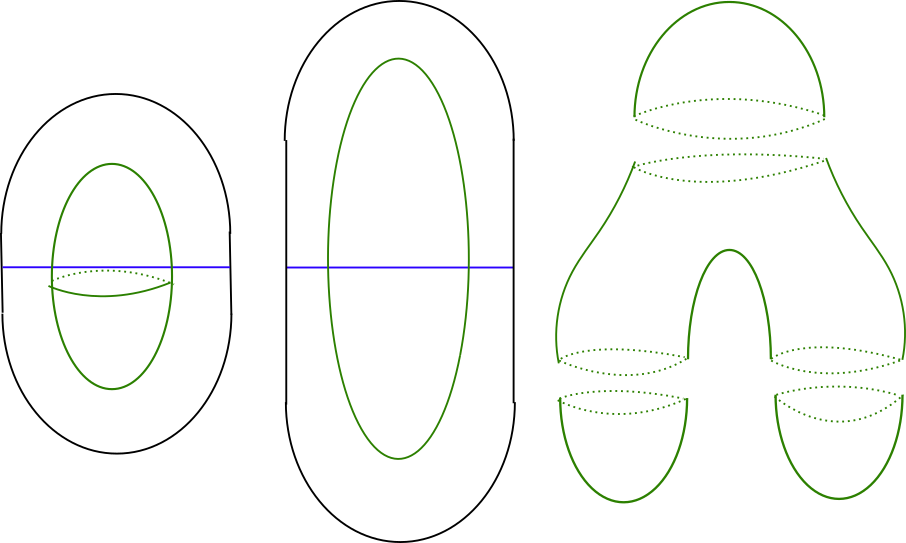}
\caption{A $J$-holomorphic sphere stretching along a contact-type hypersurface and breaking into a building (right).}
\label{fig:sft_stretching}
\end{figure}

Next, we drop the contact-type assumption on $Y$. We outline a minor variant of Fish--Hofer's proof from \cite{FH23} that $Y$ contains a proper, compact, $X_H$-invariant subset. We begin the same way, by stretching a $J$-holomorphic sphere that crosses $Y$. Without the contact-type assumption, however, we have no a priori energy bounds. The sphere exhibits wild behavior in the neck and the SFT compactness theorem fails. Fish--Hofer proved that a relatively small part of the sphere, near an end of the neck, can be controlled and limits to a single holomorphic curve in $\bR \times Y$. This holomorphic curve lives in a new class of infinite energy curves that they call \emph{feral curves}; they show that the ends of feral curves limit to $X_H$-invariant subsets of $Y$. Some additional arguments show that this particular feral curve produced by neck-stretching limits to a \emph{proper} (i.e. not equal to $Y$) compact $X_H$-invariant set $\Lambda$. This proof is illustrated in Figure~\ref{fig:fh_stretching}.

\begin{figure}[h]
\centering
\includegraphics[width=0.5\textwidth]{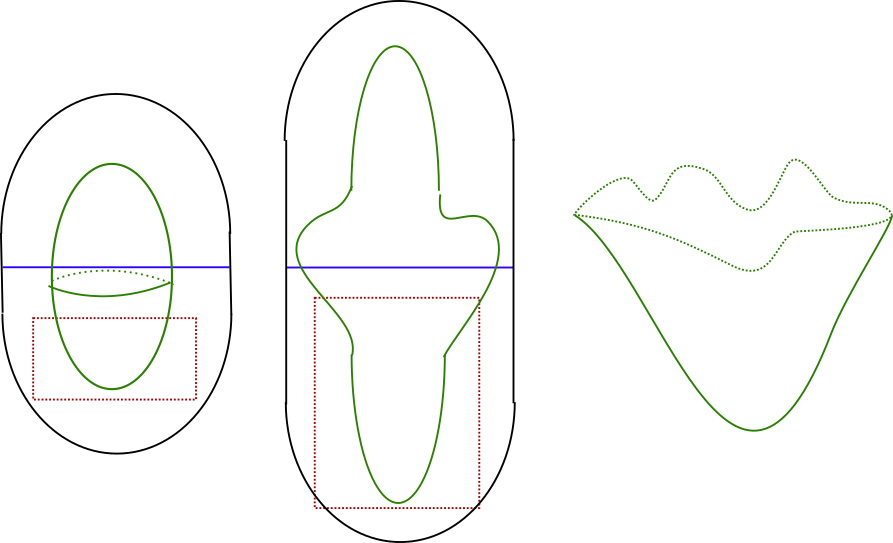}
\caption{A $J$-holomorphic sphere stretching wildly along a non contact-type hypersurface. A carefully selected sequence of controlled, but successively longer, parts (in the boxes) limits to a feral curve (right).}
\label{fig:fh_stretching}
\end{figure}

Our approach to Theorem~\ref{thm:main} relies on two observations. First, the manifold $\mathbb{CP}^2$ has holomorphic curves of every degree $d$ crossing $Y$, not just degree $1$ spheres. Indeed, the Fish--Hofer argument outlined above can be adapted to higher degree curves\footnote{We essentially do this in Proposition~\ref{prop:intersection_argument} on the way to proving Theorem~\ref{thm:main}.}. We obtain a proper compact invariant subset $\Lambda_d$ for each $d$. However, it seems difficult to tell from this viewpoint whether the subsets $\Lambda_d$ are distinct for different values of $d$, and moreover one cannot control their location in $Y$. So, a different approach is required to extract new data from higher degree curves. 

Our second observation is that the Fish--Hofer approach, by necessity, disregards a lot of information. A limiting feral curve can only be constructed by restricting to very specific pieces of a stretching holomorphic curve. Our solution, which we consider to be our main conceptual contribution, is to not attempt to extract limiting holomorphic curves at all. We introduce a topological alternative to the SFT compactification, called the ``stretched limit set'', which remedies issues with both the SFT approach and the Fish--Hofer approach. Unlike the SFT compactification, the stretched limit set exists outside of the contact-type setting. Unlike the Fish--Hofer procedure, the stretched limit set contains information from all parts of a stretching holomorphic curve, not just the parts near the ends of the neck. Informally, the stretched limit set is a collection of pairs $(\Xi, s)$, where $\Xi \subseteq (-1, 1) \times Y$ is a subsequential Hausdorff limit of height-two slices of the stretching curves, and $s \in \bR$ tracks the vertical positions of these slices. See Figure~\ref{fig:rp_stretching} for an illustration and \S\ref{subsubsec:stretched_limit_set} for a formal definition. We emphasize that $\Xi$, in general, is not a holomorphic curve. It could be a much wilder closed subset of $(-1, 1) \times Y$, such as a fractal set, a subset with non-empty interior, or even $(-1, 1) \times Y$ itself. 

\begin{figure}[h]
\centering
\includegraphics[width=0.5\textwidth]{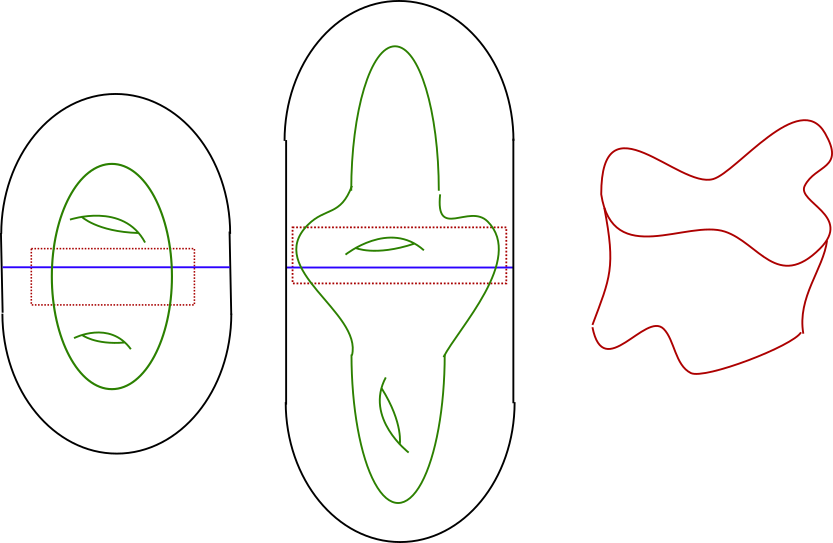}
\caption{A $J$-holomorphic curve, possibly with genus, stretching wildly along a non contact-type hypersurface. A Hausdorff limit of any sequence of slices (in the boxes) yields an element of the stretched limit set (right). This is usually not a $J$-holomorphic curve.}
\label{fig:rp_stretching}
\end{figure}

For each $d$, we apply our construction to a stretching degree $d$ curve; let $\cX_d$ denote the resulting stretched limit set. The set $\cX_d$ could be large and complicated, but with some careful analysis we extract a connected subset $\cZ_d \subseteq \cX_d$ satisfying several properties (see Proposition~\ref{prop:stretched_limit}). The most notable property is that $\cZ_d$ consists of nearly-invariant sets. For every element $(\Xi, s) \in \cZ_d$, there exists a closed $X_H$-invariant set $\Lambda \subseteq Y$ such that $\Xi$ is $\delta$-close to $(-1, 1) \times \Lambda$, where $\delta \to 0$ as $d \to \infty$. Another limiting procedure, this time taking $d \to \infty$, gives after additional arguments an infinite family of proper, compact, $X_H$-invariant subsets with dense union. The analysis of the stretched limit set has several different parts: delicate analysis of stretching holomorphic curves, estimates for holomorphic curves from \cite{FH23, CGP24}, quantitative properties of the degree $d$ moduli spaces, and at one point the intersection theory of holomorphic curves. The last two parts are what require us to work in four dimensions; we elaborate in \S\ref{subsec:further_questions}. 

\subsubsection{Comparisons to other works} Before moving on to other results, we compare the ideas sketched here to \cite{FH23, CGP24}. The only significant overlap is the use of estimates from \cite{FH23, CGP24} and an intersection theory argument inspired by \cite{FH23}\footnote{See the proof of Proposition~\ref{prop:intersection_argument}.}. The main dynamical results in \cite{CGP24} also use holomorphic curves, but they use holomorphic curves arising from the $U$-map in embedded contact homology or periodic Floer homology, and deep properties of these curves derived from the relationship of ECH and PFH with Seiberg--Witten theory. ECH and PFH are not available in our setting. Certain other kinds of ``limit sets'' appeared in \cite{FH23, CGP24}, but they are very different from the stretched limit set. The limit set in \cite{FH23} is a single compact invariant set representing the end of a feral holomorphic curve. The limit set in \cite{CGP24} is a family of invariant sets produced from a sequence of $U$-map curves with action going to $0$ and bounded topology. In contrast, the stretched limit set $\cX_d$ has many non-invariant sets for any fixed $d$. It is only by selecting a subset of the stretched limit set and passing to the $d \to \infty$ limit that we are eventually able to produce enough invariant sets to prove Theorem~\ref{thm:main}. 

\subsubsection{Global Le Calvez--Yoccoz property} To prove Theorem~\ref{thm:main2}, we use a different neck stretching procedure, inspired by the ``adiabatic neck stretching'' procedure introduced in \cite{FH22}. Instead of neck-stretching at $Y$, we simultaneously neck-stretch at every level set near $Y$. For each degree $d$, we introduce an analogue, the so-called ``adiabatic limit set'' $\widehat{\cX}_d$, of the stretched limit set $\cX_d$. We prove that almost every $s \in \cR_c(H)$ and every $d \geq 1$ there exists a well-behaved connected subset $\cZ^s_d$ (see Proposition~\ref{prop:adiabatic_limit}), where as above the elements become closer to being $X_H$-invariant as $d \to \infty$. Passing to the $d \to \infty$ limit produces a connected family $\cY^s$ of compact invariant sets in $H^{-1}(s)$ such that (i) some $\Lambda \in \cY^s$ is contained in $\overline{\cP(s)}$ and (ii) $Y \in \cY^s$. It follows that any compact invariant set containing $\overline{\cP(s)}$ is either equal to $Y$ or is not locally maximal. 

\subsubsection{Quantitative almost-existence} The proof of Theorem~\ref{thm:main3} also uses adiabatic limit sets. Assume for the sake of contradiction that there exists a positive measure subset $\cB \subseteq \cR_c(H)$ such that $H^{-1}(s)$ contains at most one closed orbit for each $s \in \cB$. Then, we prove that $\widehat\cX_d$ satisfies impossibly strong restrictions for sufficiently large $d$; see Lemma~\ref{lem:finite_accumulation} and Lemma~\ref{lem:blowup_implies_accumulation}. Here is an informal description of what is proved. After a slight refinement of $\cB$ (see Lemma~\ref{lem:bad_subset}), we prove that for all but $O(d)$ levels $s \in \cB$, there exists $(\Xi, s) \in \widehat\cX_d$ such that $\Xi = (-1, 1) \times \Lambda$, where $\Lambda \subset H^{-1}(s)$ is a closed orbit. On the other hand, there exist $d^2$ levels $s \in \cB$ on which $\Xi$ can be assumed to pass through an arbitrary point constraint. Thus, if $d$ is sufficiently large, there must exist $s \in \cB$ such that $H^{-1}(s)$ contains a closed orbit passing through an arbitrary point. This contradicts the assumption that $H^{-1}(s)$ has at most one closed orbit. Theorem~\ref{thm:main4} is proved using Theorem~\ref{thm:main2} and the fact that existence of a Moser stable closed orbit implies existence of infinitely many closed orbits. 

\subsection{Remarks and questions}\label{subsec:further_questions} We collect some remarks and follow-up questions. 

\subsubsection{Other $4$-manifolds} Our results extend with only minor modifications to other symplectic $4$-manifolds besides $\bR^4$. We explain in \S\ref{subsec:other_manifolds} how to extend Theorem~\ref{thm:main} to Hamiltonians $H: T^*\mathbb{S}^2 \to \bR$ and $H: T^*\mathbb{T}^2 \to \bR$. We explain in \S\ref{subsec:other_manifolds_adiabatic} how to extend Theorem~\ref{thm:main2} to Hamiltonians $H: \mathbb{M} \to \bR$, where $\mathbb{M}$ is a smooth symplectic $4$-manifold that symplectically embeds into a certain closed symplectic $4$-manifold $\bW$. We require that $\bW$ has $b^+ = 1$ and that the symplectic form has rational cohomology class. Examples of such $\bW$ include $\mathbb{CP}^2$ and its symplectic blowups, as well as products $\mathbb{S}^2 \times \Sigma$ where $\Sigma$ is any closed and orientable surface. We explain in \S\ref{subsec:other_manifolds_periodic} that Theorem~\ref{thm:main4} also holds for such $\mathbb{M}$, and that Theorem~\ref{thm:main3} holds for such $\mathbb{M}$ provided that (i) the symplectic form on $\mathbb{M}$ is exact and (ii) for any $s \in \cR_c(H)$, at least one component of $\mathbb{M}\,\setminus\,H^{-1}(s)$ has compact closure. 

\subsubsection{Higher dimensions} It would be of exceptional interest to extend the results in this paper to higher dimensions. However, all of our results rely on quantitative properties of moduli spaces of holomorphic curves in four dimensions that do not obviously hold in higher dimensions. We make essential use of the fact that, for any $d$, there exists a moduli space of degree $d$ curves in $\mathbb{CP}^2$ whose index is much larger than $d$, approximately $d^2$. In $\mathbb{CP}^n$ for $n > 2$, the analogue of degree is symplectic area. There exist holomorphic curves of symplectic area $d$ for each $d$, but the index of the moduli space grows linearly instead of quadratically in $d$. Also, our proof of Theorem~\ref{thm:main} uses the intersection theory of $J$-holomorphic curves, which is not available in higher dimensions. 

\subsubsection{Almost-existence of two or infinitely many closed orbits} As mentioned above, it is now known, after \cite{HWZ98, HWZ03, CGHP19, CGHHL23}, that any star-shaped regular level set of a Hamiltonian $H: \bR^4 \to \bR$ has either two closed orbits or has infinitely many closed orbits. Inspired by Theorem~\ref{thm:main3}, we ask if this too generalizes beyond the contact-type case.

\begin{namedthm*}{Question} Fix any smooth function $H: \bR^4 \to \bR$. Then, does there exist a full measure subset $\cQ \subseteq \cR_c(H)$ such that for any $s \in \cQ$, the level set $H^{-1}(s)$ either has two or infinitely many closed orbits? 
\end{namedthm*}

\subsubsection{Ergodic invariant measures} Although they are not the main focus of this work, one could pursue similar results for invariant probability measures. An invariant measure analogue of the Fish--Hofer theorem is known in any dimensiongreater than $2$. Ginzburg--Niche \cite{GN15} proved that for any smooth function $H: \bR^{2n} \to \bR$, where $n \geq 2$, each compact regular energy level $Y$ carries at least two ergodic $X_H$-invariant probability measures. Their argument is short; it combines McDuff's contact-type criterion \cite{McDuff87} with Viterbo's closed orbit theorem for contact-type hypersurfaces \cite{Viterbo87}. A holomorphic curve-based proof of their result can be found in \cite{Prasad23a, Prasad23b}. Taubes \cite{Taubes09} proved a similar statement for exact volume-preserving flows on closed $3$-manifolds. 

In four dimensions, the lower bound of two ergodic measures is almost sharp. As discussed in Remark~\ref{rem:katok} above, there exist star-shaped level sets in $\bR^4$ that carry exactly three ergodic $X_H$-invariant probability measures. By Theorem~\ref{thm:main3}, almost every compact regular level set of a Hamiltonian $H: \bR^4 \to \bR$ carries at least three ergodic $X_H$-invariant probability measures. It would be interesting to obtain the sharp lower bound for all compact regular level sets. 

\begin{namedthm*}{Question} Fix any smooth function $H: \bR^4 \to \bR$. Then, does any compact regular level set $Y$ carry at least three ergodic $X_H$-invariant probability measures? 
\end{namedthm*}

\subsection{Outline of article}\label{subsec:roadmap}

Several preliminary definitions and results required for our arguments are collected in \S\ref{sec:prelim}. Theorem~\ref{thm:main} is proved in \S\ref{sec:invariant_sets}. Theorem~\ref{thm:main2} is proved in \S\ref{sec:lcy}. Theorem~\ref{thm:main3} and \ref{thm:main4} are proved in \S\ref{sec:periodic_orbits}. We warn the reader that \S\ref{sec:periodic_orbits} is not self-contained. It makes free use of notation and results from \S\ref{sec:lcy}. Appendix~\ref{sec:closed_curves} discusses existence results for closed holomorphic curves in symplectic $4$-manifolds. Appendix~\ref{sec:generic} explains why the conditions in Theorem~\ref{thm:main4} are $C^\infty$-generic. 

\subsection{Acknowledgements}\label{subsec:acknowledgements}

I would like to thank Dan Cristofaro-Gardiner for useful discussions and for an enriching collaboration on \cite{CGP24}. I also thank Joel Fish, Viktor Ginzburg, Basak G\"urel, and Helmut Hofer for useful discussions and for their comments on earlier versions of this work. This research was supported by the Miller Institute at the University of California Berkeley.

%% file: prelim.tex
This section contains definitions and results required for the proofs of the main theorems. In \S\ref{subsec:geometric_prelim}--\ref{subsec:j_curves_symplectization}, we discuss the holomorphic curve framework that we will use. In \S\ref{subsec:estimates}, we collect some estimates for holomorphic curves from \cite{FH23, CGP24}. In \S\ref{subsec:hausdorff}, we review some facts about the Hausdorff topology for closed subsets of topological spaces. In \S\ref{subsec:neck_stretching_general}, we present a general neck stretching procedure for level sets of Hamiltonians on symplectic manifolds.  

\subsection{Geometric structures} \label{subsec:geometric_prelim} Fix a smooth, closed, oriented manifold $Y$ of dimension $2n - 1 \geq 3$.  

\subsubsection{Framed Hamiltonian structures} A \emph{framed Hamiltonian structure} on $Y$ is a pair $\eta = (\lambda, \omega)$ of a $1$-form $\lambda$ and a $2$-form $\omega$ such that
$$d\omega = 0,\quad \lambda\wedge\omega^{n-1} > 0.$$

The bundle $\xi := \ker(\lambda)$ is a $(2n-2)$-plane bundle on $Y$. The two-form $\omega$ restricts to a symplectic form on $\xi$. The \emph{Hamiltonian vector field} $R_\eta$ is defined implicitly by the equations
$$\lambda(R_\eta) \equiv 1,\quad \omega(R_\eta, -) \equiv 0.$$

\begin{exe}
If $\omega = d\lambda$, then $\lambda$ is a contact form and $R_\eta$ is equal to its Reeb vector field. 
\end{exe}

Let $\cI \subseteq \bR$ be either an open or closed interval. Write $a: \cI \times Y \to \cI$ for the projection onto the real coordinate. An almost-complex structure $J$ on $\cI \times Y$ is \emph{$\eta$-adapted} if 
\begin{enumerate}[(\roman*)]
\item $J$ is translation-invariant;
\item $J(\partial_a) = R_\eta$;
\item $J$ preserves the bundle $\xi$ and restricts to a $\omega$-compatible complex structure on $\xi$. 
\end{enumerate}

Let $\cD(Y)$ be the space of pairs $(\eta, J)$, where $\eta$ is a framed Hamiltonian structure and $J$ is an $\eta$-adapted almost-complex structure on $\bR \times Y$. Give $\cD(Y)$ the topology of uniform $C^\infty$-convergence. Given a choice of $(\eta, J) \in \cD(Y)$, we define a Riemannian metric 
$$g := da \otimes da + \lambda \otimes \lambda + \omega(-, J-).$$

\subsubsection{Realized Hamiltonian homotopies} Let $\cI \subseteq \bR$ be any closed and connected interval. A \emph{realized Hamiltonian homotopy} on $\cI \times Y$ is a pair $\widehat\eta = (\widehat\lambda, \widehat\omega)$ of a $1$-form $\widehat\lambda$ and $2$-form $\widehat\omega$ such that the following holds:
\begin{enumerate}[(\roman*)]
\item $\widehat{\lambda}(\partial_a) \equiv 0$ and $\widehat\omega(\partial_a, -) \equiv 0$;
\item $d\widehat\omega|_{\{s\} \times Y} \equiv 0$ for each $s \in \cI$;
\item $da \wedge \widehat\lambda \wedge \widehat{\omega}^{n-1} > 0$;
\item $\widehat{\lambda}$ is invariant under the flow of $\partial_a$;
\item If $\cI$ is unbounded, then there exists a compact subset $\cK \subset \cI$ such that $\widehat\omega$ is invariant under the flow of $\partial_a$ outside of $\cK \times Y$. 
\end{enumerate}

The bundle $\widehat\xi := \ker(da) \cap \ker(\widehat\lambda)$ is a $(2n-2)$-plane bundle on $\cI \times Y$. The two-form $\widehat\omega$ restricts to a symplectic form on $\widehat\xi$. The \emph{Hamiltonian vector field} $\widehat{R}_\eta$ is defined implicitly by the equations
$$da(\widehat{R}_\eta) \equiv 0, \quad \widehat\lambda(\widehat{R}_\eta) \equiv 1, \quad \widehat\omega(\widehat{R}_\eta, -).$$

\begin{exe}
A framed Hamiltonian structure $\eta = (\lambda, \omega)$ on $Y$ defines a realized Hamiltonian homotopy $\widehat\eta = (\widehat\lambda, \widehat\omega)$ on $\cI \times Y$. Let $\widehat\lambda$ and $\widehat\omega$ to be the unique $1$-form and $2$-form such that $\widehat\lambda(\partial_a) \equiv 0$, $\widehat\omega(\partial_a, -) \equiv 0$, and 
$$\widehat\lambda|_{\{s\} \times Y} = \lambda,\quad \widehat\omega|_{\{s\} \times Y} = \omega$$
for each $s \in \cI$. 
\end{exe}

A realized Hamiltonian homotopy can be regarded as a $1$-parameter family of framed Hamiltonian structures. Fix any $s \in \cI$. Define a $1$-form $\lambda^s$ and a $2$-form $\omega^s$ on $Y$ as the pullbacks of $\widehat\lambda$ and $\widehat\omega$, respectively, by the map $y \mapsto (s, y)$. Write $R^s$ for the pullback of $\widehat{R}_\eta$. The pair $\eta^s = (\lambda^s, \omega^s)$ is a framed Hamiltonian structure on $Y$ and $R^s$ is its Hamiltonian vector field. 

An almost-complex structure $\widehat{J}$ on $\cI \times Y$ is \emph{$\widehat\eta$-adapted} if 
\begin{enumerate}[(\roman*)]
\item $\widehat{J}(\partial_a) = \widehat{R}$;
\item $\widehat{J}$ preserves the bundle $\widehat\xi$ and restricts to a $\widehat\omega$-compatible complex structure on $\widehat\xi$;
\item If $\cI$ is unbounded, then there exists a compact subset $\cK \subset \cI$ such that $\widehat{J}$ is invariant under the flow of $\partial_a$ outside of $\cK \times Y$. 
\end{enumerate}

For each $s \in \cI$, write $J^s$ for the unique translation-invariant almost-complex structure on $\bR \times Y$ which coincides with $\widehat{J}$ on $\{s\} \times Y$. Observe that $J^s$ is $\eta^s$-adapted. 

Write $\cD(\cI \times Y)$ for the space of pairs $(\widehat\eta, \widehat{J})$, where $\widehat\eta$ is a Hamiltonian homotopy and $\widehat{J}$ is a $\widehat\eta$-adapted almost-complex structure on $\cI \times Y$. Equip $\cD(\cI \times Y)$ with the topology of $C^\infty$-convergence. Given a choice of pair $(\widehat\eta, \widehat{J}) \in \cD(\cI \times Y)$, define a Riemannian metric 
$$\widehat{g} := da \otimes da + \widehat\lambda \otimes \widehat\lambda + \widehat\omega(-, J-).$$

\subsection{Holomorphic curve basics}\label{subsec:holomorphic_curves}

\subsubsection{Riemann surfaces} The Riemann surfaces in this paper are allowed to have nodal singularities. As such, define a \emph{Riemann surface} to be a pair $(C, \bfn)$ where $C$ is a surface $C$ with smooth boundary $\partial C$, equipped with an integrable almost-complex structure $j$, and $\bfn$ is the set of \emph{nodal points}. This is a discrete set of mutually disjoint pairs of points $\{(\zeta_i^+, \zeta_i^-)\}$ in $C\,\setminus\,\partial C$. We often omit the nodal points $\bfn$ from the notation. A \emph{normalization} $\widetilde{C}$ of a Riemann surface $(C, \bfn)$ is a smooth surface obtained by blowing up each point $\zeta_i^{\pm}$ to a circle $\Gamma_i^{\pm}$, and then gluing $\Gamma_i^+$ to $\Gamma_i^-$ for each $i$. The diffeomorphism type of the resulting surface is independent of how $\Gamma_i^+$ and $\Gamma_i^-$ are glued together. The Riemann surface $(C, \bfn)$ is \emph{irreducible} if the underlying surface $C$ is connected and \emph{connected} if the normalization $\widetilde{C}$ is connected. 

The \emph{genus}, denoted by $G(C)$, of a compact Riemann surface $C$ is the genus of the closed surface obtained by capping off the boundary components with disks. The \emph{arithmetic genus}, denoted by $G_a(C)$, of a compact Riemann surface $C$ is the genus of any normalization $\widetilde{C}$. 

\subsubsection{Holomorphic curves in almost-complex manifolds} Fix an almost-complex manifold $(\mathbb{W}, J)$ and a Riemann surface $(C, \bfn)$. A \emph{$J$-holomorphic curve} with domain $C$ is a smooth map $u: C \to W$ such that $u(\zeta^+) = u(\zeta^-)$ for each pair $(\zeta^+, \zeta^-) \in \bfn$ and such that it solves the non-linear Cauchy--Riemann equation
$$J \circ Du = Du \circ j.$$ 

A $J$-holomorphic curve $u: C \to \mathbb{W}$ is \emph{compact}, \emph{irreducible}, or \emph{connected} if the domain is compact, irreducible, or connected, respectively. It is \emph{proper} if the preimage of any compact set is compact and \emph{boundary immersed} if $u$ immerses the boundary $\partial C$. We always assume that $J$-holomorphic curves are proper and boundary immersed. 

\subsection{Holomorphic curves in realized Hamiltonian homotopies} \label{subsec:j_curves_symplectization} Fix a closed, smooth, oriented manifold $Y$ of dimension $2n - 1 \geq 3$. Fix a closed, connected interval $\cI \subseteq \bR$ and a pair $(\widehat\eta, \widehat{J}) \in \cD(\cI \times Y)$. Let $u: C \to \cI \times Y$ be a $\widehat{J}$-holomorphic curve. 

\subsubsection{Action of a holomorphic curve}

The integral 
$$\int_C u^*\widehat\omega$$
is called the \emph{action} of $u$. Because $\widehat{J}$ is $\widehat\eta$-adapted, the $2$-form $u^*\widehat\omega$ is always non-negative on the tangent planes of $C$. Also, $u^*\widehat\omega$ vanishes at $\zeta \in C$ if and only if either $\zeta$ is a critical point of $u$ or $Du(T_\zeta C) = \operatorname{Span}(\partial_a, \widehat{R}_\eta)$. Thus, we have the following lemma.

\begin{lem}\label{lem:zero_action}
Fix a closed, connected interval $\cI \subseteq \bR$ and a pair $(\widehat\eta, \widehat{J}) \in \cD(\cI \times Y)$. Let $u: C \to \cI \times Y$ be a connected $\widehat{J}$-holomorphic curve. Then 
$$\int_C u^*\widehat\omega \geq 0$$
and is equal to $0$ if and only if there exists an orbit $\gamma \subset Y$ of $\widehat{R}_\eta$ such that $u(C) \subset \bR \times \gamma$. 
\end{lem}

Lemma~\ref{lem:zero_action} gives geometric meaning to the action. Holomorphic curves of low action, in some sense, approximate the orbits of $\widehat{R}_\eta$. 

\subsubsection{Area of a holomorphic curve} The pullback metric $u^*\widehat{g}$ is defined at any immersed point in $C$. The volume form $\operatorname{dvol}_{u^*\widehat{g}}$ is equal to the $2$-form $u^*(da \wedge \widehat\lambda + \widehat\omega)$. For any Borel subset $U \subset C$, write 
$$\operatorname{Area}_{u^*\widehat{g}}(U) := \int_{U\,\setminus\,\operatorname{Crit}(u)} \operatorname{dvol}_{u^*\widehat{g}} = \int_{U} u^*(da \wedge \widehat\lambda + \widehat\omega).$$

The equality on the right follows because $u^*(da \wedge \widehat\lambda + \widehat\omega)$ vanishes at any critical point of $u$. 

\subsection{Area and action bounds}\label{subsec:estimates} We collect some area and action bounds for holomorphic curves in realized Hamiltonian homotopies. The first two results, Proposition~\ref{prop:fh_area_bound} and \ref{prop:fh_quantization}, are from \cite{FH23}. The third result, Proposition~\ref{prop:local_area_bound}, is a slight variant of a recent result from \cite{CGP24}.

\subsubsection{Stable constants} Many results in this paper involve constants depending on a choice of $(\widehat\eta, \widehat J) \in \cD(\cI \times Y)$. In any result, such a constant is called \emph{stable} if the conclusions of the result hold with the same constant for data in a neighborhood of $(\widehat\eta, \widehat{J})$. 

\subsubsection{Exponential area bound} The following result provides a priori bounds for $\widehat\lambda$-integrals of the level sets and for the area of a $\widehat{J}$-holomorphic curve. The bounds are expressed in terms of the action and of the $\widehat{\lambda}$-integrals of the top and bottom boundaries. 

\begin{prop}[{\cite[Theorem $8$]{FH23}}] \label{prop:fh_area_bound}
Fix $(\widehat\eta, \widehat{J}) \in \cD([-8,8] \times Y)$. Fix any constants $a_+ > a_- \in [-8,8]$. Let $u: C \to [-8,8]\times Y$ be a compact $\widehat{J}$-holomorphic curve. Suppose that the following conditions are satisfied:
\begin{enumerate}[(\roman*)]
\item $(a \circ u)(C) \subset [a_-, a_+]$;
\item $(a \circ u)(\partial C) \cap (a_-, a_+) = \emptyset$; 
\item $a_+$ and $a_-$ are regular values of $a \circ u$. 
\end{enumerate}

Then there exists a stable constant $c_3 = c_3(\widehat\eta, \widehat{J}) \geq 1$ such that the following two bounds hold. First, for any $a_0 \in [a_+, a_-]$ that is a regular value of $a \circ u$, we have 
\begin{equation} \label{eq:fh_area_bound1} \int_{(a \circ u)^{-1}(a_0)} u^*\widehat\lambda \leq \Big(c_3\int_{C} u^*\widehat\omega + \min\big\{\int_{(a \circ u)^{-1}(a_+)} u^*\widehat\lambda, \int_{(a \circ u)^{-1}(a_-)} u^*\widehat\lambda\big\}\Big)e^{c_3(a_+ - a_-)}.\end{equation}

Second, we have the following area bound. 
\begin{equation} \label{eq:fh_area_bound2} \operatorname{Area}_{u^*\widehat{g}}(C) \leq \Big(c_3\min\big\{\int_{(a \circ u)^{-1}(a_+)} u^*\widehat\lambda, \int_{(a \circ u)^{-1}(a_-)} u^*\widehat\lambda\big\} + \int_C u^*\widehat\omega\Big)(e^{c_3(a_+ - a_-)} - 1) + \int_C u^*\widehat\omega.\end{equation}
\end{prop}

\subsubsection{Action quantization} The next result shows that a holomorphic curve with an interior maximum/minimum height has a positive lower bound on its action.

\begin{prop}[{\cite[Theorem $4$]{FH23}}] \label{prop:fh_quantization}
Fix $(\widehat\eta, \widehat J)\in\cD([-8,8] \times Y)$. For any $r > 0$, there exists a stable constant $\hbar = \hbar(\widehat\eta, \widehat{J}, r) > 0$ such that, for any compact, irreducible $\widehat{J}$-holomorphic curve $u: C \to [-8,8] \times Y$, we have
$$\int_C u^*\widehat\omega \geq \hbar > 0$$
provided that the following properties are satisfied for some $a_0 \in (-8 + r, 8 - r)$:
\begin{enumerate}[(\roman*)]
\item Either $\inf_{\zeta \in C}(a \circ u)(\zeta)$ or $\sup_{\zeta \in C} (a \circ u)(\zeta)$ is equal to $a_0$; 
\item $(a \circ u)(\partial C) \cap [a_0 - r, a_0 + r] = \emptyset$.
\end{enumerate}
\end{prop}

The original statement of \cite[Theorem $4$]{FH23} assumes an a priori bound on $G_a(C)$, and the constant $c_3$ depends on this bound. The original proof of Proposition~\ref{prop:fh_quantization} is proved uses Proposition~\ref{prop:fh_area_bound} and target-local Gromov compactness \cite{Fish11}; the genus bound is required to apply the latter. Our version does not require any a priori genus bound. This requirement can be removed by replacing target-local Gromov compactness with the compactness theorem for $J$-holomorphic currents; see \cite[Remark $5.20$]{Prasad23a}.

\subsubsection{Connected-local area bound for low-action holomorphic curves} 

The following technical area bound is a variant of \cite[Proposition $3.8$]{CGP24}. It was proved in the special case of annular curves in \cite[Theorem $5$]{FH23}. 

\begin{prop}[{\cite[Proposition $3.8$]{CGP24}}] \label{prop:local_area_bound}
Fix $(\widehat\eta, \widehat J) \in \cD([-8,8] \times Y)$. There exists stable constants $\epsilon_2 = \epsilon_2(\widehat\eta, \widehat{J}) >0$ and $\epsilon_3 = \epsilon_3(\widehat\eta, \widehat{J}) > 0$ with the following property. Let $u: C \to [-8,8]\times Y$ be a compact, connected $\widehat{J}$-holomorphic curve such that 
\begin{enumerate}[(\roman*)]
\item $\int_C u^*\widehat\omega \leq \epsilon_2$;
\item $(a \circ u)(\partial C) \cap [-4,4]= \emptyset$.
\end{enumerate}

Then for any $\zeta \in C$ such that $(a \circ u)(\zeta) \in (-2,2)$, we have the bound
$$\operatorname{Area}_{u^*\widehat{g}}(S_{\epsilon_3}(\zeta)) \leq \epsilon_3^{-1}(\chi(C)^2 + 1).$$
\end{prop}

\begin{proof}
There are two differences between Proposition~\ref{prop:local_area_bound} and \cite[Proposition $3.8$]{CGP24}. We explain how to address these differences and defer to \cite{CGP24} for the rest of the proof. The first difference is that \cite[Proposition $3.8$]{CGP24} is stated for framed Hamiltonian structures, while Proposition~\ref{prop:local_area_bound} is stated for realized Hamiltonian homotopies.  The former result is stated for framed Hamiltonian structures in order to make direct use of \cite[Theorem $9$]{FH23}. This result, an essential technical ingredient, is a exponential area bound like Proposition~\ref{prop:fh_area_bound}, but for a class of $C^2$-small ``tame perturbations'' of $J$-holomorphic curves. The result is stated for framed Hamiltonian structures. The proof, however, directly generalizes to tame perturbations of $J$-holomorphic curves in realized Hamiltonian homotopies. 

The second difference is that \cite[Proposition $3.8$]{CGP24} assumes that $u : C \to \bR \times Y$ is a proper map from a finitely punctured closed Riemann surface, while Proposition~\ref{prop:local_area_bound} assumes that $u$ is a map from a compact Riemann surface to $[-8,8] \times Y$. The proof of \cite[Proposition $3.8$]{CGP24}, however, extends to the case where the domain is compact, as long as $\zeta$ has vertical distance at least $1$ from $\partial C$. 

The non-compact domain is only used in the following argument. It is proved that there exists some compact surface $\widetilde{C} \subset C$ such that (i) $\zeta \in \widetilde{C}$, (ii) the vertical distance from $\zeta$ to $\partial \widetilde{C}$ is bounded away from $0$ by a stable constant in $(0, 1)$, (iii) $\chi(\widetilde{C}) \leq \chi(C)$, and (iv) some a priori geometric bounds, described in \cite[Proposition $3.9$]{CGP24}, are satisfied. The surface $\widetilde{C} \subset C$ is constructed by taking the surface $(a  \circ u)^{-1}([a_0 - \epsilon, a_0 + \epsilon])$, where $a_0 := (a \circ u)(\zeta)$ and $\epsilon > 0$ is a small stable constant, and then attaching all compact components of $C\,\setminus\,\operatorname{Int}(\widetilde{C})$ whose boundary is contained in $\partial\widetilde{C}$. Property (ii) is the only property that makes essential use of the fact that $C$ is a finitely punctured closed Riemann surface. In the case where $C$ is compact, (ii) can be still be proved if the distance between $\zeta$ and $\partial C$ is at least $1$. 
\end{proof}

\subsection{The Hausdorff topology}\label{subsec:hausdorff}

Let $Z$ denote any separable, locally compact, and metrizable space. For example, any second countable topological manifold satisfies these conditions. Let $\cK(Z)$ denote the space of closed subsets of $Z$, equipped with the topology of Hausdorff convergence. Recall that $\cK(Z)$ is compact and metrizable \cite[Corollary $2.2$]{McMullen96}. We review the definition of Hausdorff convergence and then state some basic lemmas. 

\subsubsection{Hausdorff convergence} Fix any sequence $\{\Lambda_k\}$ in $\cK(Z)$. Write $\liminf \Lambda_k \in \cK(Z)$ for the set of all $z \in Z$ such that each neighborhood intersects all but finitely many $\Lambda_k$. Write $\limsup \Lambda_k \in \cK(Z)$ for the set of all $z \in Z$ such that each neighborhood intersects infinitely many $\Lambda_k$. Observe that $\liminf \Lambda_k \subseteq \limsup\Lambda_k$. Convergence $\Lambda_k \to \Lambda$ in the Hausdorff topology occurs if and only if $\liminf \Lambda_k = \Lambda = \limsup\Lambda_k$.

\subsubsection{Sets of subsequential limit points} The following lemma discusses the topology of the set of subsequential limit points for a sequence in $\cK(Z)$. 

\begin{lem}\label{lem:accumulation_points}
Let $Z$ be a separable, locally compact, and metrizable space. Let $\{\cZ_k\}$ denote a sequence of connected subsets of $\cK(Z)$ and let $\cZ \subseteq \cK(Z)$ denote their set of subsequential limit points. Assume that there exists $\Lambda_k \in \cZ_k$ for each $k$ such that the sequence $\{\Lambda_k\}$ converges in the Hausdorff topology. Then, $\cZ$ is closed and connected. 
\end{lem}

\begin{proof}
See the proof of \cite[Lemma $5.3$]{CGP24}. 
\end{proof}

\subsubsection{Continuity lemmas} We present two lemmas about continuity of maps with respect to the Hausdorff topology. The proofs are elementary, so we omit them. The first lemma claims that taking a union with a closed set is always Hausdorff continuous. 

\begin{lem}\label{lem:union}
Let $Z$ be a separable, locally compact, and metrizable space. Then for any $\Lambda' \in \cK(Z)$, the map $\Lambda \mapsto \Lambda \cup \Lambda'$ is a continuous map $\cK(Z) \to \cK(Z)$. 
\end{lem}

Unlike the above, taking an intersection with a closed set is usually not continuous. The second lemma asserts that, however, this operation is continuous in a specific setting that comes up in our arguments. 

\begin{lem}\label{lem:intersection}
Let $Y$ be a separable, locally compact, and metrizable space. For any sequence $\{\Lambda_k\}$ of non-empty compact subsets of $Y$ such that $(-1, 1) \times \Lambda_k \to (-1, 1) \times \Lambda$ in $\cK( (-1,1) \times Y)$, we have $\Lambda_k \to \Lambda$ in $\cK(Y)$.
\end{lem}

\subsection{Neck stretching}\label{subsec:neck_stretching_general} Let $(\bW, \Omega)$ denote a symplectic manifold and let $J_*$ be a fixed $\Omega$-compatible almost-complex structure. Let $H: \bW \to \bR$ be any smooth function such that $0 \in \cR_c(H)$. Write $Y := H^{-1}(0)$. As we will explain below, these choices induce a framed Hamiltonian structure $\eta$ on $Y$ such that $R_\eta = X_H|_Y$. Broadly speaking, neck stretching is the modification of $J_*$ near $Y$ so that it is diffeomorphic to a model $\eta$-adapted almost-complex structure $J$ on $\cI \times Y$, where $\cI$ is a large compact interval. 

Neck stretching constructions are well-understood when $Y$ is contact-type (see \cite{BEHWZ03}) but they are surprisingly subtle otherwise. One must find some way to interpolate between the model almost-complex structure near $Y$ and the almost-complex structure $J_*$. This interpolation must be done very carefully, since otherwise it is quite easy to lose control of the behavior of holomorphic curves in the interpolation region. A neck stretching construction that is sufficient for the proof of Theorem~\ref{thm:main} was written down in \cite{Prasad23b}. However, it is quite different from the adiabatic neck stretching construction used to prove Theorem~\ref{thm:main2}. We take an alternate approach that gives a unified treatment of standard neck stretching and adiabatic neck stretching. 

\subsubsection{Framed Hamiltonian structure}\label{subsubsec:framed_hamiltonian} Let $g_* := \Omega(-, J_*-)$ denote the $J_*$-invariant Riemannian metric associated to $\Omega$ and $J_*$. We define a framed Hamiltonian structure $\eta = (\lambda, \omega)$ on $Y$ such that the vector field $R := R_\eta$ is equal to $X_H$. Let $\omega$ be the restriction of $\Omega$ to $Y$. Let $\lambda$ be the unique smooth one-form such that $\lambda(X_H) \equiv 1$ and $\ker(\lambda) = \xi := TY \cap J_{\operatorname{int}}(TY)$. Let $J$ be the unique $\eta$-adapted almost-complex structure on $\bR \times Y$ that coincides with $J_*$ on the bundle $\xi$. 

\subsubsection{Realized Hamiltonian homotopy}\label{subsubsec:realized_hamiltonian} For any $\delta > 0$, let $U_\delta \subseteq \bW$ denote the open set $H^{-1}( (-\delta, \delta) )$. Choose some $\delta_0$ such that $U_{\delta_0}$ does not contain any critical points of $H$. Now, let $V_H := \nabla H/|\nabla H|_{g_*}^2$ dneote the normalized gradient of $H$ with respect to $g_*$. For any $s \in (-\delta_0, \delta_0)$, the time-$s$ flow of $V_H$ restricts to a diffeomorphism $f_s: H^{-1}(s) \to Y$. 

Let $\widecheck\lambda$ denote the unique $1$-form on $U_{\delta_0}$ such that (i) $\widecheck\lambda(V_H) \equiv 0$ and (ii) for any $s \in (-\delta_0, \delta_0)$, the restriction of $\widecheck\lambda$ to $H^{-1}(s)$ is equal to $f_s^*\lambda$. Let $\widecheck\omega$ be the unique $2$-form on $U_{\delta_0}$ such that (i) $\widecheck\omega(V_H, -) \equiv 0$ and (ii) for any $s \in (-\delta_0, \delta_0)$, the restrictions of $\widecheck\omega$ and $\Omega$ to the hypersurface $H^{-1}(s)$ coincide. Observe that by definition, $\widecheck\lambda(X_H) = 1$ at any point in $Y$, so there exists some $\delta_1 \in (0, \delta_0)$ such that $\widecheck\lambda(X_H) > 0$ at any point in $U_{\delta_1}$. Define a vector field $\widecheck{X} := X_H/\widecheck\lambda(X_H)$. Define a $2$-plane bundle $\widecheck\xi := \ker(dH) \cap \ker(\widecheck\lambda)$ on $U_{\delta_1}$. Define a diffeomorphism
\begin{gather*}
\iota: (-\delta_1, \delta_1) \times Y \to U_{\delta_1}, \\
(s, y) \mapsto f_s^{-1}(y).
\end{gather*}

The map $\iota$ restricts to the identity map $(0, y) \mapsto y$ on $\{0\} \times Y$ and identifies $\{s\} \times Y$ with $H^{-1}(s)$ for every $s \in (-\delta_1, \delta_1)$. Observe that $H \circ \iota$ coincides with the coordinate projection $s: (-\delta_1, \delta_1) \times Y \to (-\delta_1, \delta_1)$ and that $\iota^*V_H = \partial_s$. Write $\widehat\lambda := \iota^*\widecheck\lambda$ and $\widehat\omega := \iota^*\widecheck\omega$. The pair $\widehat\eta = (\widehat\lambda, \widehat\omega)$ is a realized Hamiltonian homotopy on $(-\delta_1, \delta_1) \times Y$. This is straightforward to verify. Write $\widehat{X} := \iota^*\widecheck{X}$. Observe that $\widehat{X} = \widehat{R}_\eta$. For any $s \in (-\delta_1, \delta_1)$, write $\eta^s = (\lambda^s, \omega^s) \in \cD(Y)$ for the pullback of $\widehat\eta$ by the map $y \mapsto (s, y)$. Let $R^s$ denote the Hamiltonian vector field of $\eta^s$. Note that $\iota$ identifies $R^s$ with a reparameterization of $X_H$ on $H^{-1}(s)$. The following lemma computes j$\Omega$ in terms of $H$, $\widecheck\lambda$, and $\widecheck\omega$ at any point in $Y$. 

\begin{lem}\label{lem:omega_computation}
At any point in $Y$, we have $\Omega = dH \wedge \widecheck\lambda + \widecheck\omega$. 
\end{lem}

\begin{proof}
Fix any $y \in Y$. Our goal is to prove
\begin{equation}\label{eq:omega_computation}\Omega(v_1, v_2) = (dH \wedge \widecheck\lambda + \widecheck\omega)(v_1, v_2)\end{equation}
for any pair $v_1, v_2 \in T_y\mathbb{W}$. The vector $v_1$ splits as a sum $v_1 = dH(v_1)V_H + \widecheck\lambda(v_1)X_H + v_1'$ where $v_1' \in \widecheck{\xi}$. The existence of the splitting follows from the splitting $T\bW = \operatorname{Span}(V_H, X_H) \oplus \widecheck{\xi}$. The identification of the $V_H$- and $X_H$-coefficients of $v_1$ follows because $dH(V_H) = 1$ and $\widecheck\lambda(X_H) = 1$ at $y$. Similarly, $v_2 = dH(v_2)V_H + \widecheck\lambda(v_2)X_H + v_2'$, where $v_2' \in \widecheck\xi$. 

The identity \eqref{eq:omega_computation} follows from expanding the left-hand side and making several simplifications:
\begin{equation}\label{eq:omega_computation1}
\begin{split}
\Omega(v_1, v_2) &= (dH(v_1)\widecheck\lambda(v_2) - dH(v_2)\widecheck\lambda(v_1)) + \Omega(v_1', v_2') \\
&\qquad+ \Omega(v_1', dH(v_2)V_H + \widecheck\lambda(v_2)X_H) + \Omega(dH(v_1)V_H+ \widecheck\lambda(v_1), v_2') \\
&= (dH(v_1)\widecheck\lambda(v_2) - dH(v_2)\widecheck\lambda(v_1)) + \Omega(v_1', v_2') \\
&= (dH(v_1)\widecheck\lambda(v_2) - dH(v_2)\widecheck\lambda(v_1)) + \widecheck\omega(v_1', v_2') \\
&= (dH \wedge \lambda + \widecheck\omega)(v_1, v_2).
\end{split}
\end{equation}

The first line uses the identity $\Omega(V_H, X_H) = 1$. The second line uses the identities $\Omega(X_H, v') = 0$, $\Omega(V_H, v') = 0$, where $v'$ is any vector in $\widecheck\xi$. The third line uses the fact that $\Omega$ and $\widecheck\omega$ coincide on $TY$. The fourth line uses the following elementary identities. For any vector $v' \in \widecheck{\xi}$, we have 
\begin{equation}\label{eq:omega_computation2}
\widecheck\omega(V_H, v') = 0,\quad \widecheck\omega(X_H, v') = 0.
\end{equation} 
\end{proof}

\subsubsection{Base almost-complex structure}\label{subsubsec:base_acs} Write $\widecheck\xi := \ker(dH) \cap \ker(\widecheck\lambda)$. Choose a complex structure $\widecheck{J}_\xi$ on $\widecheck\xi$ which (i) coincides with the restriction of $J_*$ to $\widecheck{\xi}$ at any point in $y$ and (ii) is compatible with the restriction of $\widecheck\omega$ to $\widecheck\xi$. There exists an almost-complex structure $\widecheck{J}$ on $\bW$ satisfying the following properties:
\begin{enumerate}[(a)]
\item $\widecheck{J}$ coincides with $J_*$ outside $U_{\delta_0}$. 
\item There exists some $\delta_2 \in (0, \delta_1)$ such that, at any point in $U_{\delta_2}$, the following properties are satisfied:
\begin{itemize}
\item $\widecheck{J}$ coincides with $\widecheck{J}_\xi$ on $\widecheck{\xi}$. 
\item $\widecheck{J}(V_H) = \widecheck{X}$.
\end{itemize}
\item $\widecheck{J}$ is $\Omega$-tame at every point, i.e. $\Omega(v, \widecheck{J}v) > 0$ for any nonzero tangent vector $v$. 
\end{enumerate}

Such an almost-complex structure is constructed as follows. Choose an almost-complex structure $\widecheck{I}$ on $U_{\delta_1}$ such that (i) $\widecheck{I}$ coincides with $\widecheck{J}_\xi$ on $\widecheck\xi$ and (ii) $\widecheck{I}(V_H) = \widecheck{H}$. By Lemma~\ref{lem:omega_computation}, $\widecheck{I}$ is $\Omega$-tame on $U_{\delta_2}$ for some $\delta_2 \in (0, \delta_1)$. Let $\widecheck{J}$ be any $\Omega$-tame almost-complex structure that is equal to $J_*$ outside $U_{\delta_0}$ and is equal to $\widecheck{I}$ on $U_{\delta_2}$. Write $\widehat{J} := \iota^*\widecheck{J}$ for the pullback of $\widecheck{J}$. Note that by Property (b), $\widehat{J}$ is $\widehat\eta$-adapted on $(-\delta_2, \delta_2) \times Y$. For each $s \in (-\delta_2, \delta_2)$, write $J^s$ for the unique translation-invariant almost-complex structure on $\bR \times Y$ agreeing with the restriction of $\widehat{J}$ to $\{s\} \times Y$. The almost-complex structure $J^s$ is $\eta^s$-adapted. 

\subsubsection{Deformed almost-complex structure}\label{subsubsec:deformed_acs} Fix any smooth and positive function $\phi: \bW \to (0, 1]$ such that the function $1 - \phi$ is supported in $U_{\delta_2}$. Let $\widecheck{J}_\phi$ be the unique almost-complex structure that is (i) equal $\widecheck{J}$ outside $U_{\delta_2}$, (ii) equal to $\widecheck{J}$ on the bundle $\widecheck{\xi}$, and (iii) satisfies the identity $\widecheck{J}_\phi(V_H) = \phi^{-1}\widecheck{X}$ at any point in $U_{\delta_2}$. A useful fact, first observed in \cite{FH22}, is that if $1 - \phi$ is supported in a sufficiently small neighborhood of $Y$, then $\widecheck{J}_\phi$ is $\Omega$-tame and moreove one has a quantitative tameness estimate. 

This small neighborhood is defined as follows. Fix a constant $c_1 = c_1(H, g_*) \geq 1$ such that the bound
\begin{equation}\label{eq:tame1} \|X_H\|_{g_*} = \|V_H\|_{g_*} \leq c_1\end{equation}
is satisfied. Since $\widehat{J}$ is $\Omega$-tame, there exists a constant $\epsilon_1 = \epsilon_1(\Omega, \widecheck{J}, g_*) \in (0, 1)$ such that 
\begin{equation}\label{eq:tame2} \Omega(v, \widecheck{J}v) \geq \epsilon_1|v|_{g_*}^2 \end{equation}
for any tangent vector $v$. Define $\epsilon_2 := \min(1/4, 2^{-8}c_1\epsilon_1^{1/2})$. 
Such a constant $\delta_3$ exists Since $\widecheck\lambda(X_H) = 1$ along $Y$ and, by Lemma~\ref{lem:omega_computation}, $\Omega = dH \wedge \widecheck\lambda + \widecheck\omega$ along $Y$, there exists $\delta_3 \in (0, \delta_2)$ such that 
\begin{equation}\label{eq:tame3}
|\widecheck\lambda(X_H) - 1| \leq \epsilon_2, \quad |\Omega - (dH \wedge \widecheck\lambda + \widecheck\omega)|_{g_*} \leq \epsilon_2
\end{equation}
at any point in $U_{\delta_3}$. Here is the promised tameness lemma. 

\begin{lem}\label{lem:j_tame}
Fix any smooth and positive function $\phi: \bW \to (0, 1]$ such that $1 - \phi$ is supported in $U_{\delta_3}$. Then the almost-complex structure $\widecheck{J}_\phi$ is $\Omega$-tame. Also, for any tangent vector $v$ with base in $U_{\delta_3}$, we have the bound
\begin{equation}\label{eq:j_tame}2\Omega(v, \widecheck{J}_\phi v) \geq (dH \wedge \widecheck\lambda + \widecheck\omega)(v, \widecheck{J}_\phi v).\end{equation}
\end{lem}

\begin{proof}
Both assertions of the lemma proved nearly simultaneously via some elementary analysis. Choose a tangent vector $v$ with base in $U_{\delta_3}$. Then $v$ decomposes as a sum $v = dH(v)V_H + \widecheck{\lambda}(v)\widecheck{X} + v'$ where $v'$ lies in $\widecheck{\xi}$. We compute $\widecheck{J}_\phi v = \phi^{-1}dH(v)\widecheck{X} - \phi\widecheck{\lambda}(v)V_H + \widecheck{J}v'$. Expand the left-hand side of \eqref{eq:j_tame} with respect to these splittings: 
\begin{equation}\label{eq:j_tame1}
\begin{split}
\Omega(v, \widecheck{J}_\phi v) &= (\phi^{-1}dH(v)^2 + \phi\widecheck{\lambda}(v)^2)\Omega(V_H, \widecheck{X}) + \Omega(v', \widecheck{J}v') \\
&\quad + \Omega(v', \phi^{-1}dH(v)\widecheck{X} - \phi\widecheck{\lambda}(v)V_H) + \Omega(dH(v)V_H + \widecheck{\lambda}(v)\widecheck{X}, \widecheck{J}v') \\
&= (\phi^{-1}dH(v)^2 + \phi\widecheck{\lambda}(v)^2)\widecheck{\lambda}(X_H)^{-1} + \widecheck\omega(v, \widecheck{J}_\phi v) \\
&\quad + \Omega(v', \phi^{-1}dH(v)\widecheck{X} - \phi\widecheck{\lambda}(v)V_H) + \Omega(dH(v)V_H + \widecheck{\lambda}(v)\widecheck{X}, \widecheck{J}v'). \\
&= (\phi^{-1}dH(v)^2 + \phi\widecheck{\lambda}(v)^2)\widecheck{\lambda}(X_H)^{-1} + \widecheck\omega(v, \widecheck{J}_\phi v) \\
&\quad + \Omega(v', -\phi\widecheck{\lambda}(v)V_H) + \Omega(dH(v)V_H, \widecheck{J}v').
\end{split}
\end{equation}

The second line is a consequence of the following simplifications. First, observe that $\Omega(V_H, X_H) = 1$ and therefore that $\Omega(V_H, \widecheck{X}) = \widecheck{\lambda}(X_H)^{-1}$. Second, observe that $\widehat\omega$ and $\Omega$ restrict to the same $2$-form on any level set, so $\Omega(v', \widecheck{J}v') = \widecheck\omega(v', \widecheck{J}v')$, and then apply the identities in \eqref{eq:omega_computation2} to show $\widecheck\omega(v', \widecheck{J}v') = \widecheck\omega(v, \widecheck{J}_\phi v)$.  The third line follows from the similar fact that $\Omega(\widecheck{X}, -)$ is proportional to $dH$, so pairs to $0$ with $v'$ and $\widecheck{J}v'$. For any vector $v'' \in \operatorname{Span}(V_H, X_H)$, we have 
\begin{equation}\label{eq:j_tame2}
(dH \wedge \widecheck\lambda + \widecheck\omega)(v', v'') = (dH \wedge \widecheck\lambda + \widecheck\omega)(\widecheck{J}v', v'') = 0.
\end{equation}

The identities \eqref{eq:j_tame2} follow from \eqref{eq:omega_computation2} above. Now, set $\tau = \Omega - (dH \wedge \widecheck\lambda + \widecheck\omega)$. We use \eqref{eq:j_tame2} to estimate the cross-terms at the end of \eqref{eq:j_tame1}: 

\begin{equation}\label{eq:j_tame3}
\begin{split}
&\Omega(v', -\phi\widecheck{\lambda}(v)V_H) + \Omega(dH(v)V_H, \widecheck{J}v') \\
&\quad = \tau(v', -\phi\widecheck{\lambda}(v)V_H) + \tau(dH(v)V_H, \widecheck{J}v') \\
&\quad \geq -c_1|\tau|_{g_*}(\phi|\widecheck{\lambda}(v)||v'|_{g_*} + |dH(v)||\widecheck{J}v'|_{g_*}) \\
&\quad \geq -c_1\epsilon_2(\phi|\widecheck{\lambda}(v)||v'|_{g_*} + |dH(v)||\widecheck{J}v'|_{g_*}) \\
&\quad \geq -c_1\epsilon_1^{-1/2}\epsilon_2(\phi|\widecheck{\lambda}(v)| + |dH(v)|)\widecheck\omega(v, \widecheck{J}_\phi v)^{1/2} \\
&\quad \geq -(\phi^{-1}dH(v)^2 + \phi\widecheck{\lambda}(v)^2)/4 - \widecheck\omega(v, \widecheck{J}_\phi v)/2. 
\end{split}
\end{equation}

The second line uses the bound \eqref{eq:tame1}. The third line uses \eqref{eq:tame3} to control $|\tau|_{g_*}$. The fourth line uses \eqref{eq:tame2} and the identity $\Omega(v', \widecheck{J}v') = \widecheck\omega(v', \widecheck{J}_\phi v')$. The fifth line uses the Cauchy--Schwarz inequality, the fact that $\phi^{-1} \geq 1$, and the bound $\epsilon_2 \leq 2^{-8}c_1^{-1}\epsilon_1^{1/2}$. Plug in \eqref{eq:j_tame3} into \eqref{eq:j_tame1} and use the upper bound on $\widecheck\lambda(X_H)$ from \eqref{eq:tame3} to get
\begin{equation}\label{eq:j_tame4}
2\Omega(v, \widecheck{J}_\phi v) \geq (\phi^{-1}dH(v)^2 + \phi\widecheck\lambda(v)^2) + \widecheck\omega(v, \widecheck{J}_\phi v).
\end{equation}

The right-hand side of \eqref{eq:j_tame4} is equal to $(dH \wedge \widecheck\lambda + \widecheck\omega)(v, \widecheck{J}_\phi v)$, so \eqref{eq:j_tame} follows from \eqref{eq:j_tame4}. It remains to show that $\widecheck{J}_\phi$ is $\Omega$-tame. The right-hand side of \eqref{eq:j_tame4} is positive, so $\widecheck{J}_\phi$ is $\Omega$-tame on $U_{\delta_3}$. By definition, $\widecheck{J}_\phi = \widecheck{J}$ on the complement of $U_{\delta_3}$, so $\widecheck{J}_\phi$ is $\Omega$-tame. 
\end{proof}

\subsubsection{Stretched manifolds}\label{subsubsec:stretched_manifolds} Fix a smooth function $\phi: \bW \to (0, 1]$ satisfying the following properties:
\begin{itemize}
\item The function $1 - \phi$ is supported on $U_{\delta_3}$. 
\item For any $s \in [-\delta_3, \delta_3]$, $\phi$ is equal to a constant $\phi(s)$ on the hypersurface $H^{-1}(s)$.
\item For any $s \in [-\delta_3, \delta_3]$, we have $\phi(s) = \phi(-s)$. 
\end{itemize}

Define a smooth manifold $\bW_\phi$ and a diffeomorphism $f_\phi: \bW \to \bW_\phi$ as follows. Let $L_\phi := \int_{-\delta_3}^0 \phi(s)^{-1} ds$. Define a smooth function $\Phi: [-\delta_3, \delta_3] \to [-L_\phi, L_\phi]$ by setting $\Phi(s) := -L_\phi + \int_{-\delta_3}^s \phi(s)^{-1}$. Now, write $\bW_+ := H^{-1}( [\delta_3, \infty) )$ and $\bW_- := H^{-1}( (-\infty, -\delta_3])$ and set $\bW_\phi := \bW_+ \cup_Y [-L_\phi, L_\phi] \times Y \cup_Y \bW_-$. We use the letter $a$ to denote the $\bR$-coordinate on the neck. The diffeomorphism $f_\phi: \bW \to \bW_\phi$ is defined to be the identity on $\bW_+$ and $\bW_-$. It is defined on the neighborhood $U_{\delta_3}$ by setting
$$f_\phi \circ \iota: (-\delta_3, \delta_3) \times Y \to (-L_\phi, L_\phi) \times Y \subset \bW_\phi$$ to be the map $(s, y) \mapsto (\Phi(s), y)$. Write $\Omega_\phi := (f_\phi)_*\Omega$ and write $\widehat{J}_\phi := (f_\phi)_*\widecheck{J}_\phi$. 

The almost-complex structures $\widehat{J}_\phi$ could be very degenerate, since $\widecheck{J}_\phi(\widecheck{X}) = \phi V_H$ is very small when $\phi$ is close to $0$. The pushforward by $f_\phi$ undoes this degeneracy. Define a $1$-form $\widehat\lambda_\phi := (f_\phi)_*\widecheck\lambda$ and a $2$-form $\widehat\omega_\phi := (f_\phi)_*\widecheck\omega$ on the neck $[-L_\phi, L_\phi] \times Y$. Then the pair $\widehat\eta_\phi = (\widehat\lambda_\phi, \widehat\phi_\phi)$ is a realized Hamiltonian homotopy on $[-L_\phi, L_\phi] \times Y$. Write $\widehat{R}_\phi = (f_\phi)_*(\widecheck{X})$ for its Hamiltonian vector field and define $\widehat\xi_\phi := \ker(da) \cap \ker(\widehat\lambda_\phi)$. Then, we make the following claim. 

\begin{lem}\label{lem:phi_adapted} $\widehat{J}_\phi$ is $\widehat{\eta}_\phi$-adapted. \end{lem}

\begin{proof}
Recall the realized Hamiltonian homotopy $\widehat\eta$ and $\hat\eta$-adapted almost-complex structure $\widehat{J}$ on $[-\delta_3, \delta_3] \times Y$ that we defined above. Note that $\widehat\eta_\phi = (f_\phi \circ \iota)_*\widehat\eta$ and $\widehat{J}_\phi = (f_\phi \circ \iota)_*\widehat{J}$. Since $\widehat{J}$ is $\widehat\omega$-compatible on $\widehat\xi = \ker(ds) \cap \ker(\widehat\lambda)$, it follows from pushing forward by $f_\phi \circ \iota$ that $\widehat{J}_\phi$ is $\widehat\omega_\phi$-compatible on $\widehat\xi_\phi := \ker(da) \cap \ker(\widehat\lambda_\phi)$.

It remains to compute the action of $\widehat{J}_\phi$ on $\partial_a$:
\begin{equation}\label{eq:phi_adapted}
\begin{split}
\widehat{J}_\phi(\partial_a) &= (f_\phi)_*(\widecheck{J}_\phi(f_\phi^*\partial_a) ) = (f_\phi)_*\widecheck{J}_\phi(\phi(s) \cdot V_H) = (f_\phi)_*(\widecheck{X}) = \widehat{R}_\phi.
\end{split}
\end{equation}
\end{proof}

\subsubsection{Sequence of degenerating functions}\label{subsubsec:neck_stretching_general} Choose a sequence of smooth functions $\{\phi_k: \bW \to (0, 1]\}$ that $C^\infty$-converges to some $\phi: \bW \to [0,1]$ and satisfies the following properties: 
\begin{itemize}
\item The function $1 - \phi_k$ is supported on $U_{\delta_3}$ for each $k$. 
\item For each $k$ and each $s \in (-\delta_3, \delta_3)$, $\phi_k$ is equal to a constant $\phi_k(s)$ on the hypersurface $H^{-1}(s)$.
\item For each $k$ and each $s \in (-\delta_3, \delta_3)$, we have $\phi_k(s) = \phi_k(-s)$.
\item For each $k$, the integral $L_{\phi_k} = \int_{-\delta_3}^0 \phi_k(t)^{-1} dt$ is at least $16k$. 
\end{itemize}

\subsubsection{Convergence of geometric objects}\label{subsubsec:geometric_convergence} 

We examine the limiting behavior as $k \to \infty$ of the geometric objects associated to the sequence $\{\phi_k\}$. We simplify the notation for these objects by replacing any ``$\phi_k$'' subscripts with $k$ and removing most of the accents. Write $\widecheck{J}_k := \widecheck{J}_{\phi_k}$. Write $L_k := L_{\phi_k} = \int_{-\delta_3}^0 \phi_k(t) dt$ and define a function $\Phi_k(s) := \int_{-\delta_3}^s \phi_k(t)^{-1} dt - L_k$. Then, write $\bW_k := \bW_{\phi_k}$ for the stretched manifolds and $f_k := f_{\phi_k}$ for the diffeomorphisms $\bW \to \bW_k$. Write $\Omega_k := \Omega_{\phi_k}$ and $J_k := \widehat{J}_{\phi_k}$. Write $\omega_k := \widehat{\omega}_{\phi_k}$, $\lambda_k := \widehat{\lambda}_{\phi_k}$, $\eta_k = (\lambda_k, \omega_k)$, $R_k := \widehat{R}_{\phi_k}$, and $\xi_k := \widehat{\xi}_{\phi_k}$.

Fix any $k$ and any $a$ such that $[a - 8, a + 8] \subseteq [-L_k, L_k]$. Then, define $(\eta_k^a, J_k^a) \in \cD( [-8, 8] \times Y)$ to be the pair defined by restriction of $(\eta_k, J_k)$ to $[a - 8, a + 8] \times Y$ and then translation by $-a$. The following lemma asserts that the family $\{(\eta_k^a, J_k^a)\} \subset \cD([-8, 8] \times Y)$ has compact closure. 

\begin{lem}\label{lem:geometric_convergence}
Fix any sequence $\{a_k\}$ such that (i) $[a_k - 8, a_k + 8] \subseteq [-L_k, L_k]$ for every $k$ and (ii) the sequence $\{\Phi_k^{-1}(a_k)\}$ converges to some $\bar{s}(0) \in [-\delta_3, \delta_3]$. Then, the sequence $\{(\eta_k^{a_k}, J_k^{a_k})\}$ in $\cD([-8, 8] \times Y)$ is convergent. 
\end{lem}

\begin{proof}[Proof of Lemma~\ref{lem:geometric_convergence}]
The proof will take $4$ steps.

\noindent\textbf{Step $1$:} This step proves the following elementary claim. Consider the sequence of smooth functions $s_k: [-8, 8] \to [-\delta_3, \delta_3]$ defined by $s_k(a) = \Phi_k^{-1}(a_k + a)$. We claim that the sequence $\{s_k\}$ converges in the $C^\infty$ topology to a smooth function $\bar{s}: [-8, 8] \to [-\delta_3, \delta_3]$. Observe that $s_k$ solves the ODE $s_k'(a) = \phi_k(s_k(a))$. By assumption, the sequence of initial conditions $\{s_k(0)\}$ converges and the coefficients $\{\phi_k\}$ converge in $C^\infty$. The existence and uniqueness of solutions to ODEs then implies that $\{s_k\}$ converges. 

\noindent\textbf{Step $2$:} To simplify our notation, we write $\bar\lambda_k := \lambda_k^{a_k}$, $\bar\omega_k := \omega_k^{a_k}$, $\bar\eta_k := \eta_k^{a_k}$, and $\bar{J}_k := J_k^{a_k}$. This step establishes a necessary and sufficient criterion for the convergence of $\{\bar{\eta}_k\}$. Fix any differential form $\beta$ on $[-8, 8] \times Y$. Let $\beta^*$ denote a smooth function on $[-8, 8]$, valued in forms on $Y$, sending $a$ to the restriction $\beta|_{\{a\} \times Y}$. Then, at any given point $(a, y) \in [-8, 8] \times Y$, $\beta$ expands as a sum 
\begin{equation}\label{eq:geometric_convergence}\beta_{(a, y)} = \beta^*(a)_y + da \wedge \beta_{(a, y)}(\partial_a, -).\end{equation}

For any realized Hamiltonian homotopy $\bar\eta = (\bar\lambda, \bar\omega)$ on $[-8, 8] \times Y$, write $\bar\eta^*$ for the smooth function $a \mapsto (\bar\lambda^*(a), \bar\omega^*(a))$. Now, observe that $\bar\lambda_k(\partial_a) \equiv 0$ and $\bar\omega_k(\partial_a, -) \equiv 0$. It follows from the expansion \eqref{eq:geometric_convergence} that $\bar\eta_k \to \bar\eta = (\bar\lambda, \bar\omega)$ if and only if $\bar\lambda^*_k \to \bar\lambda^*$ and $\bar\omega^*_k \to \bar\omega^*$ in the topology of smooth form-valued functions on $[-8, 8]$. 

\noindent\textbf{Step $3$:} This step constructs a pair $\bar\eta$ such that $\bar\eta_k \to \bar\eta$. We defined a realized Hamiltonian homotopy $\widehat\eta$ on $[-\delta_3, \delta_3] \times Y$ in \S\ref{subsubsec:realized_hamiltonian} and a $\widehat\eta$-adapted almost-complex structure $\widehat{J}$ in \S\ref{subsubsec:base_acs}. For each $s \in [-\delta_3, \delta_3] \times Y$, let $\eta^s$ denote the framed Hamiltonian structure on $Y$ defined by pullback by the map $y \mapsto (s, y)$. Let $J^s$ denote the unique translation-invariant almost-complex structure on $\bR \times Y$ that coincides with $\widehat{J}$ on $\{s\} \times Y$. 

Let $\bar{s}$ denote the limit of the sequence $\{s_k\}$ from Step $1$. Define a $1$-form $\bar\lambda$ on $[-8, 8] \times Y$ by defining $\bar\lambda(\partial_a) \equiv 0$ and, for each $a \in [-8, 8]$, defining $\bar\lambda^*(a) := \lambda^{\bar{s}(a)}$. Define a $2$-form $\bar\omega$ on $[-8, 8] \times Y$ by defining $\bar\omega(\partial_a, -) \equiv 0$ and, for each $a \in [-8, 8]$, defining $\bar\omega^*(a) := \omega^{\bar{s}(a)}$. 
By the criterion of Step $2$, the convergence $\bar\eta_k \to \bar\eta$ is equivalent to $C^\infty$-convergence of the form-valued functions $\bar\lambda_k^* \to \bar\lambda^*$ and $\bar\omega_k^* \to \bar\omega$. 

For any $k$, we have 
$$\bar\lambda_k = (\iota^{-1} \circ f_k^{-1} \circ \tau_{-a_k})^*\widehat{\lambda},\quad \bar\omega_k = (\iota^{-1} \circ f_k^{-1} \circ \tau_{-a_k})^*\widehat\omega.$$

By definition, the map $f_k^{-1} \circ \tau_{-a_k}$ on the cylinder $[-8, 8] \times Y$ is given by the map 
$$(a, y) \mapsto (\Phi_k^{-1}(a_k + a), y) = (s_k(a), y).$$ 

It follows that
$$\bar\lambda_k^*(a) = \lambda^{s_k(a)} = \lambda^* \circ s_k,\quad \bar\omega_k^*(a) = \omega^{s_k(a)} = \omega^* \circ s_k$$
for each $a \in [-8, 8]$. Since $s_k \to \bar{s}$ in the $C^\infty$ topology, it follows that $\bar\lambda_k^*$ converges to $\lambda^* \circ \bar{s} = \bar\lambda^*$ and that $\bar\omega_k^*$ converges to $\omega^* \circ \bar{s} = \bar\omega^*$ as desired. 

\noindent\textbf{Step $4$:} This step defines a $\bar\eta$-adapted almost-complex structure $\bar{J}$ and proves that $\bar{J}_k \to \bar{J}$, completing the proof of the lemma.  Write $\bar\xi := \ker(da) \cap \ker(\bar\lambda)$. The bundle $\bar\xi$ has the following form. For each $s$, let $\xi^s = \ker(da) \cap \ker(\lambda^s)$ denote the translation-invariant $2$-plane bundle on $\bR \times Y$ associated to the framed Hamiltonian structure $\eta^s$. Then, for any $a \in [-8, 8] \times Y$, the restriction of $\bar\xi$ to $\{a\} \times Y$ coincides with the restriction of $\xi^{\bar{s}(a)}$. This assertion follows from the fact that the restriction of $\bar\lambda$ to $\{a\} \times Y$ is equal to $\lambda^{\bar{s}(a)}$. 

Recall the $\widehat\eta$-adapted almost-complex structure $\widehat{J}$ that we defined in \S\ref{subsubsec:base_acs}. For each $s \in [-\delta_3, \delta_3]$, we defined $J^s$ to be the unique translation-invariant almost-complex structure on $\bR \times Y$ whose restriction to $\{s\} \times Y$ coincides with $J$, and observed that $J^s$ is $\eta^s$-adapted, and therefore restricts to an $\omega^s$-compatible complex structure on the bundle $\xi^s$. Define $\bar{J}$ to be the unique $\bar\eta$-adapted almost-complex structure such that for any $a \in [-8, 8]$, when restricted to the hypersurface $\{a\} \times Y$, the action of $\bar{J}$ on $\bar{\xi}$ is identified with the action of $J^{\bar{s}(a)}$ on $\xi^{\bar{s}(a)}$. 

Now, we will prove that $\bar{J}_k \to \bar{J}$. The bundle $\xi_k := \ker(da) \cap \ker(\bar\lambda_k)$ coincides on $\{a\} \times Y$ with the bundle $\xi^{s_k(a)}$. When restricted to the hypersurface $\{a\} \times Y$, the action of $\bar{J}_k$ on $\xi_k$ is identified with the action of $J^{s_k(a)}$ on $\xi^{s_k(a)}$. The convergence $\bar{J}_k \to \bar{J}$ will follow from showing that, for any smooth vector field $v$ on $[-8, 8] \times Y$, we have $\bar{J}_k(v) \to \bar{J}(v)$ in the $C^\infty$ topology. The vector field $v$ splits uniquely as $v = da(v)\partial_a + \bar\lambda(v)\bar{X} + v'$, where $\bar{X}$ denotes the Hamiltonian vector field of $\bar\eta$ and $v' \in \bar\xi$. For each $k$, it also splits as $v = da(v)\partial_a + \bar\lambda_k(v)\bar{X}_k + v'_k$, where $\bar{X}_k$ denotes the Hamiltonian vector field of $\bar\eta_k$ and $v'_k \in \bar\xi_k$. We compute
$$\bar{J}_k(v) = da(v)\bar{X}_k - \bar\lambda_k(v)\partial_a + \bar{J}_k(v'_k),\quad \bar{J}(v) = da(v)\bar{X} - \bar\lambda(v)\partial_aa + \bar{J}(v').$$

By Step $3$, we have $\bar\eta_k \to \bar\eta$, so 
\begin{equation}\label{eq:geometric_convergence2} \bar{X}_k \to \bar{X},\quad \bar\lambda_k(v) \to \bar\lambda(v) \end{equation}
and therefore
\begin{equation}\label{eq:geometric_convergence3} v'_k = v - da(v)\partial_a - \bar\lambda_k(v)\bar{X}_k \to v - da(v)\partial_a - \bar\lambda(v)\bar{X} = v' \end{equation}
in the $C^\infty$ topology. By expanding $\bar{J}'_k(v'_k)$ and $\bar{J}(v')$ as functions of $(a, y) \in [-8, 8] \times Y$, we have
\begin{equation}\label{eq:geometric_convergence4} \bar{J}_k(v'_k(a, y)) = J^{s_k(a)}(v'_k(a,y)), \quad \bar{J}(v'(a,y)) = J^{\bar{s}(a)}(v'(a,y)). \end{equation}

It follows from \eqref{eq:geometric_convergence3}, \eqref{eq:geometric_convergence4}, and the $C^\infty$-convergence $s_k \to \bar{s}$ that
\begin{equation}\label{eq:geometric_convergence5} \bar{J}_k(v'_k) \to \bar{J}(v') \end{equation}
in the $C^\infty$ topology. The convergence $\bar{J}_k(v) \to \bar{J}(v)$ follows from \eqref{eq:geometric_convergence2} and \eqref{eq:geometric_convergence5}. 
\end{proof}

\begin{rem}\label{rem:stable_constants}
\normalfont

Many of our results involve stable constants depending on a choice of $(\widehat\eta, \widehat{J}) \in \cD([-8,8] \times Y)$. These constants can be replaced by constants independent of $(\widehat\eta, \widehat{J})$, such that the conclusions of the results hold for $(\widehat\eta, \widehat{J}) = (\eta_k^a, J_k^a)$ for any $a$ and $k$. This is a consequence of the following general principle. By stability, for any precompact subset $\cD_* \subseteq \cD([-8, 8] \times Y)$, stable constants can be replaced constants that do not vary on the family $\cD_*$. Lemma~\ref{lem:geometric_convergence} shows that the family $\{(\eta_k^a, J_k^a)\}$ is precompact.
\end{rem}

%% file: invariant_sets.tex
In this section, we will prove Theorem~\ref{thm:main}. For the remainder of the section, we fix a smooth function $H: \bR^4 \to \bR$. Fix $s_0 \in \cR_c(H)$ such that $H^{-1}(s_0)$ is connected. Assume without loss of generality that $s_0 = 0$; we reduce to this case by replacing $H$ with $H - s_0$, since adding a constant to $H$ will not change the Hamiltonian vector field. Set $Y := H^{-1}(0)$. 

\subsection{An existence result for almost cylinders}\label{subsec:almost_cyl} As we discussed in \S\ref{subsec:proof_outlines}, we construct by neck stretching holomorphic curves of high degree. We begin by introducing ``$\delta$-almost cylinders'', a convenient formal notion of nearly-invariant set, and proving some basic lemmas. We then state our main existence result for almost cylinders. 

\subsubsection{Definition of almost cylinders}

Fix any $(\hat\eta, \hat{J}) \in \cD([-1,1] \times Y)$ and any $\delta > 0$. A closed subset $\Xi \subseteq (-1,1) \times Y$ is a \emph{$\delta$-almost cylinder with respect to $(\widehat\eta, \widehat J)$} if it is non-empty and the following two bounds hold for any point $z = (t, y) \in \Xi$:
\begin{equation}\label{eq:almost_cyl_bounds} \sup_{\tau \in (-1,1)} \operatorname{dist}_{\widehat{g}}( (\tau, y), \Xi) \leq \delta,\qquad \sup_{\tau \in (-1,1)} \operatorname{dist}_{\widehat{g}}( (t, \phi^\tau(y)), \Xi) \leq \delta.\end{equation}

We will omit $(\widehat\eta, \widehat{J})$ from the notation whenever there is no risk of ambiguity.

\subsubsection{Properties of almost cylinders}

Our first lemma shows that non-empty Hausdorff limits of $\delta$-almost cylinders, for $\delta \in (0,1/2)$, are themselves $\delta$-almost cylinders. 

\begin{lem}\label{lem:almost_cyl_limits}
Fix any $\delta \in (0, 1/2)$. Let $\{(\widehat\eta_k, \widehat{J}_k)\}$ be a sequence in $\cD([-1,1] \times Y)$ converging to $(\widehat\eta, \widehat{J})$. Fix a sequence $\{\Xi_k\}$ in $\cK( (-1,1) \times Y)$ such that $\Xi_k$ is a $\delta$-almost cylinder with respect to $(\widehat\eta_k, \widehat{J}_k)$ for each $k$. Then, the set $\Xi = \limsup \Xi_k$ is a $\delta$-almost cylinder with respect to $(\widehat\eta, \widehat{J})$. 
\end{lem}

\begin{proof}
The proof of the lemma will take $2$ steps. 

\noindent\textbf{Step $1$:} This step shows that $\Xi$ is non-empty. Since $\delta \in (0, 1/2)$, it follows that for each $k$, there exists some point $z_k = (t_k, y_k) \in \Xi_k$ with $t_k \in [-1/2,1/2]$. Since $[-1/2,1/2] \times Y$ is compact, the points $z_k$ have a subsequential limit point and therefore $\Xi$ is non-empty. 

\noindent\textbf{Step $2$:} This step shows that $\Xi$ satisfies both bounds in \eqref{eq:almost_cyl_bounds}. The proofs of both bounds are similar, so we only give a full proof of the first bound. Fix any point $z = (t, y) \in \Xi$. Then, after passing to a subsequence, there exist points $z_k = (t_k, y_k) \in \Xi_k$ such that $z_k \to z$. Fix any $\tau \in (-1, 1)$.  For each $k$, there exists $z_k' \in \Xi_k$ such that $\operatorname{dist}_{\widehat{g}_k}( (\tau, y_k), z_k') \leq \delta$, where $\widehat{g}_k$ denotes the Riemannian metric induced by $(\widehat{\eta}_k, \widehat{J}_k)$. Note that $\limsup_{k \to \infty} \operatorname{dist}_{\widehat{g}_k}(z_k', \Xi) = 0$. Then, using the triangle inequality, we obtain the following bound:
\begin{equation*}
\begin{split}
\operatorname{dist}_{\widehat{g}}( (\tau, y), \Xi) &\leq \limsup_{k \to \infty} \operatorname{dist}_{\widehat{g}}( (\tau, y_k), \Xi) \\
&\leq \limsup_{k \to \infty} (\operatorname{dist}_{\widehat{g}}( (\tau, y_k), z_k') + \operatorname{dist}_{\widehat{g}}(z_k', \Xi) ) \\
&= \limsup_{k \to \infty} (\operatorname{dist}_{\widehat{g}_k}( (\tau, y_k), z_k') + \operatorname{dist}_{\widehat{g}_k}(z_k', \Xi) ) \leq \delta.
\end{split}
\end{equation*}

The third line uses the convergence $\widehat{g}_k \to \widehat{g}$. 
\end{proof}

The next lemma confirms the expected fact that $\delta$-almost cylinders become cylinders over compact invariant sets as $\delta \to 0$. 

\begin{lem}\label{lem:almost_cyl_to_cyl}
Fix some $(\widehat\eta, \widehat{J}) \in \cD([-1,1] \times Y)$. Assume that $\Xi$ is a $\delta$-almost cylinder with respect to $(\widehat\eta, \widehat{J})$ for every $\delta > 0$. Then $\Xi = (-1, 1) \times \Lambda$, where $\Lambda \in \cK(Y)$ is a non-empty, compact, $X_H$-invariant subset of $Y$. 
\end{lem}

\begin{proof}
For any $z = (t, y) \in \Xi$, taking $\delta \to 0$ in \eqref{eq:almost_cyl_bounds} implies that $(\tau, y),\,(t, \phi^\tau(y)) \in \Xi$. 
\end{proof}

\subsubsection{Existence of almost cylinders} We now state our main existence result for almost cylinders. 

\begin{prop}\label{prop:main}
There exists a pair $(\eta, J) \in \cD(Y)$ such that $R_\eta = X_H$ and such that the following holds. Fix a finite set of points $\mathbf{p} \subset Y$ and a positive integer $n > 2$. Then there exists a connected subset $\cZ_{\mathbf{p}, n} \subseteq \cK((-1,1)\times Y)$ with the following properties:
\begin{enumerate}[(a)]
\item There exists $\Xi \in \cZ_{\mathbf{p}, n}$ such that $\{0\} \times \mathbf{p} \subset \Xi$. 
\item There exists $\Xi \in \cZ_{\mathbf{p}, n}$ equal to $(-1, 1) \times \Lambda$, where $\Lambda$ is a proper, compact, $X_H$-invariant subset of $Y$. 
\item Each $\Xi \in \cZ_{\mathbf{p}, n}$ is a $1/n$-almost cylinder with respect to $(\eta, J)$. 
\end{enumerate}
\end{prop}

\subsection{Proof of Theorem~\ref{thm:main}}\label{subsec:main_proof} We defer the proof of Proposition~\ref{prop:main} to \S\ref{subsec:almost_cyl_proof} and first explain how to use it to prove Theorem~\ref{thm:main}. 

\begin{proof}
The proof will take $3$ steps.

\noindent\textbf{Step $1$:} This step uses Proposition~\ref{prop:main} to construct a connected family of invariant subsets of $Y$ satisfying several properties. Fix a finite set of points $\mathbf{p} \subset Y$. Let $\{\cZ_{\mathbf{p}, n}\}$ be the sequence of subsets of $\cK( (-1,1) \times Y)$ from Proposition~\ref{prop:main}. After passing to a subsequence in $n$, we may assume that there exists $\Xi_n \in \cZ_{\mathbf{p}, n}$ such that the sequence $\{\Xi_n\}$ converges.

Let $\cZ_{\bfp} \subset \cK( (-1,1)\times Y)$ be the set of all subsequential limit points of the sequence $\{\cZ_{\mathbf{p}, n}\}$ in $\cK( (-1,1)\times Y)$. That is, $\Xi\in \cZ_{\mathbf{p}}$ if and only if there exists a subsequence $n_j$ and elements $\Xi_{j} \in \cZ_{\mathbf{p}, n_j}$ such that $\lim_{j \to \infty} \Xi_{j} = \Xi$. Since each $\cZ_{\mathbf{p}, n}$ is connected, $\cZ_{\mathbf{p}}$ is connected by Lemma~\ref{lem:accumulation_points}. We claim that $\cZ_{\mathbf{p}}$ satisfies the following properties:
\begin{enumerate}[(a)]
\item There exist some $\Xi \in \cZ_{\mathbf{p}}$ such that $\{0\} \times \mathbf{p} \subset \Xi$.
\item There exists a compact $X_H$-invariant set $\Lambda \in \cK(Y)$ such that
\begin{enumerate}[(\roman*)]
\item The set $(-1,1) \times \Lambda$ is an element of $\cZ_{\mathbf{p}}$;
\item There exists a convergent sequence $\Lambda_j \to \Lambda$ in $\cK(Y)$ such that $\Lambda_j$ is a proper, compact, $X_H$-invariant set for each $j$. 
\end{enumerate}
\item Each $\Xi \in \cZ_{\mathbf{p}}$ is equal to $(-1,1) \times \Lambda$, where $\Lambda \in \cK(Y)$ is a non-empty, compact, $X_H$-invariant set. 
\end{enumerate}

We explain how these properties follow from Proposition~\ref{prop:main}(a--c) and the properties of almost cylinders discussed above. Property (a) is a direct consequence of Proposition~\ref{prop:main}(a). To prove Property (b), we observe that by Proposition~\ref{prop:main}(b), there exists for each $n$ a proper compact $X_H$-invariant set $\Lambda_{n}$ such that $(-1,1) \times \Lambda_n \in \cZ_{\mathbf{p}, n}$. There exists a subsequence $\{n_j\}$ such that $(-1,1) \times \Lambda_{n_j} \to (-1, 1) \times \Lambda \in \cZ_{\mathbf{p}}$. Then apply Lemma~\ref{lem:intersection}. To prove Property (c), we observe that, by Proposition~\ref{prop:main}(c) and Lemma~\ref{lem:almost_cyl_limits}, each element $\Xi \in \cZ_{\mathbf{p}}$ must be a $1/n$-almost cylinder for each $n$. Then, apply Lemma~\ref{lem:almost_cyl_to_cyl}. 

\noindent\textbf{Step $2$:} This step uses another limit construction to construct a large connected family $\cY$ of compact $X_H$-invariant subsets. Fix a sequence of finite subsets $\mathbf{p}_\ell \subset Y$ that converge to $Y$ in the Hausdorff topology. More simply put, $\mathbf{p}_\ell$ becomes increasingly dense in $Y$ as $\ell \to \infty$. For each $\ell$, there exists a connected subset $\cZ_{\bfp_\ell} \subseteq \cK( (-1, 1) \times Y)$ satisfying Properties (a--c) from Step $1$. After passing to a subsequence, we may assume that there exists $\Xi_\ell \in \cZ_{\bfp_\ell}$ such that the sequence $\{\Xi_\ell\}$ converges. Let $\cZ$ denote the set of subsequential limit points of the sequence $\{\cZ_{\bfp_\ell}\}$ as $\ell \to \infty$. Then, let $\cY \subseteq \cK(Y)$ denote the image of $\cZ$ under the map $\Xi \mapsto \{0\} \times \Xi$. 

Note that $\cZ$ is connected by Lemma~\ref{lem:accumulation_points} and that $\cY$ is connected by Lemma~\ref{lem:intersection} and Property (c) from Step $1$. The following properties of $\cY$ are deduced from Properties (a--c) from Step $1$:
\begin{enumerate}[(a')]
\item $Y \in \cY$.
\item There exists a convergent sequence $\Lambda_j \to \Lambda$ in $\cK(Y)$ such that $\Lambda \in \cY$ and each $\Lambda_j$ is a proper, compact, $X_H$-invariant set.
\item Each $\Lambda \in \cY$ is a non-empty, compact, $X_H$-invariant set. 
\end{enumerate}

\noindent\textbf{Step $3$:} This step finishes the proof by considering two opposite cases and resolving each one separately. First, assume that $\cY$ consists of a single element. By Property (a') from Step $2$, we have $\cY = \{Y\}$. By Property (b'), $Y$ is the limit of a sequence $\{\Lambda_j\}$ of proper, compact, $X_H$-invariant sets. The union of such a sequence is dense, and such a sequence must have infinitely many elements, so the theorem is proved in this case. Second, assume that $\cY$ does not consist of a single element. By Property (a'), we have $Y \in \cY$. Since $\cY$ is connected, $Y$ is not an isolated point in $\cY$. Therefore, $Y$ is a Hausdorff limit of a sequence $\{\Lambda_j\}$, where each $\Lambda_j \in \cK(Y)$ is a proper compact subset. By Property (c'), each $\Lambda_j$ is $X_H$-invariant. As in the first case, this suffices to prove the theorem. 
\end{proof}

\subsection{Geometric setup}\label{subsec:geom_setup} We have proved Theorem~\ref{thm:main} assuming that Proposition~\ref{prop:main} is true. To set the stage, we embed $Y$ into $\mathbb{CP}^2$, and then stretch the neck around $Y$ via the procedure introduced in \S\ref{subsec:neck_stretching_general}. Then, we introduce the key new object in our method: the ``stretched limit set''. 

\subsubsection{Compactification}\label{subsubsec:compactification} Choose $B > 0$ such that $Y$ lies inside the open ball of symplectic volume $B$ centered at the origin. Denote this ball by $\mathbb{B}$. Let $\bW$ denote the complex projective space $\mathbb{CP}^2$ and let $\Omega$ denote the Fubini--Study symplectic form, normalized so that $(\bW, \Omega)$ has symplectic volume $1$. The ball $\mathbb{B}$ is symplectomorphic to $\bW\,\setminus\,\bD$, where $\bD$ is a complex line such that $[\bD] \in H_2(\bW; \bZ)$ is Poincar\'e dual to $B^{-1}[\Omega] \in H^2(\bW; \bZ)$. Passing through the symplectomorphism $\mathbb{B} \simeq \bW\,\setminus\,\bD$, we regard $Y$ as a hypersurface in $\bW$ that is disjoint from $\bD$. After modifying $H$ outside of a neighborhood of $Y$, we may assume without loss of generality that it extends to a smooth function on $\bW$, also denoted by $H$, such that $0 \in \cR_c(H)$, $Y = H^{-1}(0)$, and $H > 0$ on $\bD$. For simplicity, we assume that $B = 1$. Rescaling $\Omega$ to $B^{-1}\Omega$ rescales $X_H$ by a constant, which does not change its invariant subsets. 

\subsubsection{Recollections from \S\ref{subsec:neck_stretching_general}} Let $J_*$ be an $\Omega$-compatible almost-complex structure on $\bW$ such that $\bD$ is $J_*$-holomorphic. Choose $\delta_0 > 0$ such that $(-\delta_0, \delta_0) \subseteq \cR_c(H)$ and such that $U_{\delta_0} = H^{-1}( (-\delta_0, \delta_0))$ is disjoint from $\bD$. We repeat the setup from \S\ref{subsubsec:framed_hamiltonian}--\ref{subsubsec:stretched_manifolds}. 
\begin{itemize}
\item In \S\ref{subsubsec:framed_hamiltonian} we defined a framed Hamiltonian structure $\eta = (\lambda, \omega)$ on $Y$. 
\item In \S\ref{subsubsec:realized_hamiltonian} we extended $\eta$ to a pair $\widecheck\eta = (\widecheck\lambda, \widecheck\omega)$ on $U_{\delta_0}$, fixed collar coordinates $\iota: (-\delta_1, \delta_1) \times Y \to U_{\delta_1}$, and defined a realized Hamiltonian homotopy $\widehat\eta = (\widehat\lambda, \widehat\omega)$ to be the pullback $\iota^*\widecheck\eta$. We defined $\widehat\Omega = \iota^*\Omega$. 
\item In \S\ref{subsubsec:base_acs} we defined an $\Omega$-tame almost-complex structure $\widecheck{J}$ on $\bW$ and fixed $\widehat{J} = \iota^*\widecheck{J}$; recall that $\widehat{J}$ is $\widehat\eta$-adapted on $(-\delta_2, \delta_2) \times Y$. 
\item In \S\ref{subsubsec:deformed_acs} we defined deformations $\widecheck{J}_\phi$, agreeing with $\widecheck{J}$ on the bundle $\widecheck{\xi} = \ker(ds) \cap \ker(\widecheck{\lambda})$, which by Lemma~\ref{lem:j_tame} are tame if $1 - \phi$ is supported in $U_{\delta_3}$. 
\item In \S\ref{subsubsec:stretched_manifolds}, we defined stretched manifolds $\bW_\phi$, containing long necks $[-L_\phi, L_\phi] \times Y$, and diffeomorphisms $f_\phi: \bW \to \bW_\phi$. Let $\widehat\Omega_\phi$, $\widehat\eta_\phi = (\widehat\lambda_\phi, \widehat\omega_\phi)$, $\widecheck{\xi}_\phi$, $\widehat{J}_\phi$ denote the pushforwards by $f_\phi$ of $\Omega$, $\widecheck\eta = (\widecheck\lambda, \widecheck\omega)$, $\widecheck\xi$, $\widecheck{J}$, respectively.
\end{itemize} 

\subsubsection{Neck stretching}\label{subsubsec:neck_stretching} 

We repeat the construction in \S\ref{subsubsec:neck_stretching_general} with some extra conditions. Choose a constant $\delta_4 \in (0, \delta_3/2)$ and, for each $k > 4\delta_4^{-1}$, a smooth function $\phi_k: \bW \to (0 ,1]$ with the following properties:
\begin{itemize}
\item The function $1 - \phi_k$ is supported on $U_{\delta_4}$. 
\item For any $s \in (-\delta_3, \delta_3)$, $\phi_k$ is equal to a constant $\phi_k(s)$ on the hypersurface $H^{-1}(s)$.
\item For any $s \in (-\delta_3, \delta_3)$, we have $\phi_k(s) = \phi_k(-s)$.
\item The integral $L_{\phi_k} = \int_{-\delta_3}^0 \phi_k(t)^{-1} dt$ is at least $16k$. 
\item $\phi_k(s) = k^{-2}$ for every $s \in (-k^{-1}, k^{-1})$. 
\end{itemize}

We require the sequence $\{\phi_k\}$ to converge as $k \to \infty$ to some smooth function $\phi: \bW \to [0, 1]$. Fix any $k$. Define $\widecheck{J}_k$, $L_k$, $\Phi_k$, $\bW_k$, $f_k$, $\Omega_k$, $J_k$, $\lambda_k$, $\eta_k$, $R_k$, and $\xi_k$ as in \S\ref{subsubsec:geometric_convergence}. Recall the pairs $(\eta_k^a, J_k^a) \in \cD([-8,8] \times Y)$. The following convergence result is a consequence of Lemma~\ref{lem:geometric_convergence}. 

\begin{cor}\label{cor:geometric_convergence}
Fix any sequence $\{a_k\}$ such that $a_k \in (-k, k)$ for each $k$. Then, we have $(\eta_k^{a_k}, J_k^{a_k}) \to (\eta, J)$ in $\cD([-8, 8] \times Y)$. 
\end{cor}

\begin{proof}
Recall the functions $s_k$ from the proof of Lemma~\ref{lem:geometric_convergence}. The functions $\Phi_k$ restrict to diffeomorphisms from $(-k, k)$ to $(-k^{-1}, k^{-1})$. It follows that the sequence $\{s_k\}$ converges in $C^\infty$ to the function $\bar{s} \equiv 0$. Now apply Lemma~\ref{lem:geometric_convergence}. Step $3$ of its proof shows that the limiting pair is $(\eta, J)$. 
\end{proof}

\subsubsection{Stretched limit set}\label{subsubsec:stretched_limit_set} We introduce the stretched limit set of a sequence of $J_k$-holomorphic curves. Given $a_0 \in \bR$, let $\tau_{a_0}$ denote the shift map $(a, y) \mapsto (a - a_0, y)$ on $\bR \times Y$. 

For any sequence $\{k_j\}$, fix a closed, connected Riemann surface $C_j$ and a $J_{k_j}$-holomorphic curve $u_j: C_j \to \bW_{k_j}$. The \emph{stretched limit set} $\cX \subseteq \cK( (-1,1) \times Y) \times (-1, 1)$ is the collection of pairs $(\Xi, s)$ for which there exists a sequence $a_j \in (-k_j, k_j)$ such that: 
\begin{enumerate}[(\roman*)]
\item $k_j^{-1}a_j \to s$;
\item A subsequence of the slices
$$\tau_{a_j} \cdot \Big( u_j(C_j) \cap (a_j - 1, a_j + 1) \times Y\Big) \subseteq (-1, 1) \times Y$$
converge in $\cK( (-1,1) \times Y)$ to $\Xi$. 
\end{enumerate}

Write $\pi_{\cK}$ and $\pi_{\mathbb{R}}$ for the projections of $\cK( (-1, 1) \times Y) \times (-1, 1)$ onto its factors. The following lemma asserts that a subsequence $\{k_j\}$ can always be chosen such that the stretched limit set $\cX$ is well-connected. 

\begin{lem}\label{lem:lim_set_conn}
Fix a sequence $\{k_j\}$ and a sequence $u_j: C_j \to \bW_{k_j}$ of $J_{k_j}$-holomorphic curves. Then there exists a subsequence $\{u_{j_\ell}\}$ whose stretched limit set $\cX$ has the following property. For any closed, connected interval $\cJ \subseteq \Xi$ of positive length, the subset $\pi_{\bR}^{-1}(\cJ) \subseteq \cX$ is connected. 
\end{lem}

\begin{proof}
For each $j$, define a map 
\begin{equation*}
\begin{gathered}
\cS_j: (-1,1) \to \cK( (-1,1) \times Y), \\
s \mapsto \tau_{k_j s} \cdot u_{k_j}^{-1}( (k_j s - 1, k_j s + 1) \times Y).
\end{gathered}
\end{equation*}

An elementary argument shows that $\cS_j$ is continuous; see \cite[Lemma $5.2$]{CGP24}. Choose a subsequence $\{u_{j_\ell}\}$ as follows. Fix a countable and dense subset $\bfs \subset (-1,1)$. For each $s \in \mathbf{s}$, we require the sequence of sets $\{\cS_{j_\ell}(s)\}$ to converge. 

Let $\cX$ denote the stretched limit set of the sequence $\{u_{j_\ell}\}$ and let $\cJ$ be any closed, connected interval of positive length. The interval $\cJ$ contains some $s \in \bfs$. Set $\cX' := \pi_{\bR}^{-1}(\cJ)$. The sequence $\{\cS_{j_\ell}(s)\}$ converges to some $\Xi$; it follows that $(\Xi, s) \in \cX'$. 

We claim that $\cX'$ is connected. We prove this claim by showing that for any element $(\Xi', s') \in \cX'$, there exists a connected subset $\cX'' \subseteq \cX'$ containing both $(\Xi, s)$ and $(\Xi', s')$. By definition of $\cX'$, there exists a further subsequence $\{j_\ell'\}$ and a sequence $\{a_\ell\}$ such that $(j_\ell')^{-1}a_\ell \to s'$ and $\cS_{j_\ell'}((j_\ell')^{-1}a_\ell) \to \Xi'$ in the Hausdorff topology. 

For each large $\ell$, define $\cJ_\ell \subseteq \cJ$ to be the closed interval with endpoints $(j_\ell')^{-1}a_\ell$ and $s$ for each $i$. Then, set $\cZ_\ell := \cS_{j_\ell'}(\cJ_\ell) \subseteq \cK( (-1,1) \times Y)$. Since $\cS_{j_\ell'}$ is continuous and $\cJ_\ell$ is connected, it follows that $\cZ_\ell$ is connected for each $i$. Now, define $\cX''$ to be the set of subsequential limit points of the sequence $\{\cZ_\ell\}$. The set $\cX''$ contains both $(\Xi, s)$ and $(\Xi', s')$, and it is connected by Lemma~\ref{lem:accumulation_points}. To finish the proof of the claim, we only need to verify that $\cX'' \subseteq \cX'$. The set of subsequential limit points of the sequence $\{\cJ_\ell\}$ is the closed interval $\cJ''$ with endpoints at $s$ and $s'$. Since $\cJ$ is closed and $\cJ_\ell \subseteq \cJ$ for each $\ell$, we have $\cJ'' \subseteq \cJ$, which implies that $\cX'' \subseteq \cX'$. 
\end{proof}

\subsection{Proof of Proposition~\ref{prop:main}}\label{subsec:almost_cyl_proof} Fix a finite set of points $\mathbf{p} \subset Y$ and an integer $n \geq 1$ as in the statement of the proposition. Let $m := \#\mathbf{p}$ denote the cardinality of $\mathbf{p}$. 

\subsubsection{Closed holomorphic curves with point constraints} 

For any integer $d \geq 1$ and any integer $k \geq 1$, define a collection of points $\mathbf{w}_{d,k} \subset \bW_k$ as follows. Let $\bfa_{d,k} := \{-ikd^{-2}\,|\,i\in\mathbb{Z} \cap [-d^2, d^2]\}$ denote a finite set of $2d^2 + 1$ equally spaced points in $[-k, k]$. Choose points $w_+ \in \mathbb{W}_+$ and $w_- \in \mathbb{W}_-$. Then, set $\mathbf{w}_{d,k} := (\bfa_{d,k} \times \bfp)\,\cup\,\{w_+, w_-\}$. The set $\bfa_{d,k} \times \mathbf{p}$ is the set of points $(a, p) \in [-k,k] \times Y \subset \bW_k$ such that $a \in \bfa_{d,k}$ and $p \in \mathbf{p}$. 

We construct holomorphic curves $u_{d,k}$ passing through $\mathbf{w}_{d,k}$ using the following well-known existence result. To state it, we define for any integer $e \geq 1$ a pair of integers $I(e) := (e^2 + 3e)/2$ and $g(e) := (e-1)(e-2)/2$. 

\begin{prop}\label{prop:closed_curves}
Let $A \in H_2(\bW; \bZ)$ denote the Poincar\'e dual of $\Omega$. Fix any integer $e \geq 1$ and any finite subset $\mathbf{w} \subset \bW$ of size at most $I(e)/2$. Then, for any $\Omega$-tame almost-complex structure $\bar{J}$, there exists a closed, connected Riemann surface $C$ and a $\bar{J}$-holomorphic curve $u: C \to \bW$ such that (i) $G_a(C) = g(e)$, (ii) $u_*[C] = eA$ and (iii) $\mathbf{w} \subset u(C)$. 
\end{prop}

\begin{proof}
The proposition is proved in \S\ref{subsec:closed_curves_proof}. 
\end{proof}

Observe that $\#\mathbf{w}_{d,k} = 2(md^2 + m + 2) \leq I(4md)/2$. Recall also that $\widecheck{J}_k$ is $\Omega_k$-tame for each $k$ by Lemma~\ref{lem:j_tame}. Apply Proposition~\ref{prop:closed_curves} with $e = 4md$, $\mathbf{w} = f_k^{-1}(\mathbf{w}_{d,k})$, and $\bar{J} = \widecheck{J}_k$. Composing the resulting $\widecheck{J}_k$-holomorphic curve with $f_k$, we deduce the following corollary.

\begin{cor}\label{cor:closed_curves}
Fix any $d \geq 1$ and any large $k$. Let $A_k \in H_2(\bW_k; \bZ)$ denote the Poincar\'e dual of $[\Omega_k]$. Then there exists a closed, connected Riemann surface $C_{d,k}$ and a $J_k$-holomorphic curve $u_{d,k}: C_{d,k} \to \bW_k$ such that (i) $G_a(C) = g(4md)$, (ii) $(u_{d,k})_*[C_{d,k}] = 4md A_k$ and (iii) $\mathbf{w}_{d,k} \subset u_{d,k}(C_{d,k})$. 
\end{cor}

\subsubsection{Construction of stretched limit sets} For each fixed $d$, let $\{u_{d,k}\}$ denote the sequence of curves from Corollary~\ref{cor:closed_curves}. For any subsequence $\bfk = \{k_j\}$, let $\cX_d(\bfk)$ denote the stretched limit set of the sequence $\{u_{d,k_j}\}$.

\subsubsection{Main technical proposition} The following proposition concerns the structure of $\cX_d(\bfk)$ when $d$ is large.  

\begin{prop}\label{prop:stretched_limit}
Fix any integer $n \geq 1$. Then there exists a large integer $d \gg 1$, a sequence $\bfk$, and a closed, connected interval $\cJ \subseteq (-1,1)$ such that $\pi_{\bR}^{-1}(\cJ) \subseteq \cX_d(\bfk)$ has the following properties:
\begin{enumerate}[(a)]
\item There exists $(\Xi, s) \in \pi_{\bR}^{-1}(\cJ)$ such that $\{0\} \times \mathbf{p} \subset \Xi$;
\item There exists $(\Xi, s) \in \pi_{\bR}^{-1}(\cJ)$ such that $\Xi = (-1, 1) \times \Lambda$, where $\Lambda \subseteq Y$ is a proper, compact, $R_\eta$-invariant set. 
\item For each $(\Xi, s) \in \pi_{\bR}^{-1}(\cJ)$, the set $\Xi$ is a $1/n$-almost cylinder. 
\item $\pi_{\bR}^{-1}(\cJ)$ is connected. 
\end{enumerate}
\end{prop}

\subsubsection{Proof of Proposition~\ref{prop:main}}

We defer the proof of Proposition~\ref{prop:stretched_limit} to \S\ref{subsec:stretched_limit_proof}. We first use Proposition~\ref{prop:stretched_limit} to prove Proposition~\ref{prop:main}. 

\begin{proof}
Now, let $d$, $\bfk$, and $\cJ$ be as in Proposition~\ref{prop:stretched_limit}. Set $\mathcal{W} := \pi^{-1}(\cJ) \subseteq \cX_d(\bfk)$. The set $\mathcal{W}$ is connected by Proposition~\ref{prop:stretched_limit}(d). Define $\cZ_{\mathbf{p}, n} \subseteq \cK( (-1,1)\times Y)$ to be equal to $\pi_{\cK}(\mathcal{W})$. Since $\cW$ is connected and $\pi_{\cK}$ is continuous, $\cZ_{\mathbf{p}, n}$ is connected. Proposition~\ref{prop:main}(a--c) each follow from Proposition~\ref{prop:stretched_limit}(a--c) since $R_\eta = X_H$.
\end{proof}

\subsection{Proof of Proposition~\ref{prop:stretched_limit}}\label{subsec:stretched_limit_proof}

\subsubsection{Capped slices and accumulation sets}\label{subsubsec:capped_accumulation} We introduce several definitions and notations to prepare for the proof of Proposition~\ref{prop:stretched_limit}.

Let $\cR \subseteq [-1,1]$ denote the set of levels $t$ such that, for each $d$ and $k$, we have (i) $kt$ is a regular value of $a \circ u_{d,k}$ and (ii) the subset $(a \circ u_{d,k})^{-1}(kt)$ does not contain any nodal points. Now, given any closed interval $\mathcal{I} \subseteq (-1,1)$ with endpoints in $\cR$, we associate to it a compact sub-surface $C^{\mathcal{I}}_{d,k} \subseteq C_{d,k}$, called a \emph{capped slice}. Write $\Sigma := (a \circ u_{d,k})^{-1}(k \cdot \mathcal{I})$. The set $\Sigma$ is non-empty, because by Corollary~\ref{cor:closed_curves}, $u_{d,k}(C_{d,k})$ is connected and passes through points in both components of $\bW_k\,\setminus\,[-L_k, L_k] \times Y$. Moreover, $\Sigma$ is a smooth, compact surface because the endpoints of $\cI$ lie in $\cR$. Call an irreducible component of $Z$ of $C_{d,k}\,\setminus\,\operatorname{Int}(\Sigma)$ \emph{short} if (i) $u_{d,k}(Z) \subset (-k,k) \times Y$ and (ii) $\sup_{\zeta \in Z} (a \circ u_{d,k})(\zeta) - \inf_{\zeta \in Z} (a \circ u_{d,k})(\zeta) \leq 2$. 

Let $\Delta \subseteq C_{d,k}\,\setminus\,\operatorname{Int}(\Sigma)$ denote the union of all short connected components. Then, set $C^{\mathcal{I}}_{d,k} := \Sigma \cup \Delta$. We track the level sets where the action and topology of the curves $u_{d,k}$ accumulate as $d \to \infty$. Given $d \geq 1$, a real number $\epsilon > 0$, and a sequence $\bfk = \{k_j\}$, define a subset $\mathbf{r}_\omega(d, \epsilon; \bfk) \subset (-1,1)$ as follows. We say $s \in \mathbf{r}_\omega(d, \epsilon; \bfk)$ if and only if there exists a sequence of intervals $\mathcal{L}_j$ satisfying the following properties:
\begin{enumerate}[(\roman*)]
\item The sequence $\{\cL_j\}$ converges to $\{s\}$ in $\cK(\bR)$; 
\item We have the action bound $\limsup_{j \to \infty} \int_{C^{\mathcal{L}_j}_{d,k_j}} u_{d,k_j}^*\omega_{k_j} > \epsilon$.
\end{enumerate}

The subset $\bfs_\omega(d, \epsilon; \bfk)$ tracks the accumulation of action.  Define $\bfs_\omega(d; \bfk) := \bigcup_{\epsilon > 0} \bfs_\omega(d, \epsilon; \bfk)$. 

Next, given $d \geq 1$, an integer $b \geq 1$, and a sequence $\bfk = \{k_j\}$, define a subset $\mathbf{r}_\chi(d, b; \bfk)$ as follows. We say $s \in \mathbf{r}_\chi(d, b; \bfk)$ if it admits a sequence of intervals $\mathcal{L}_j$ satisfying the following properties:
\begin{enumerate}[(\roman*)]
\item The sequence $\{\cL_j\}$ converges to $\{s\} \in \cK(\bR)$; 
\item There exists a sequence of irreducible components $Z_{d,j} \subseteq C^{\mathcal{L}_{j}}_{d,k_j}$ such that 
$$\limsup_{j \to \infty} \chi(Z_{d,j}) < -b.$$
\end{enumerate}

Note that, for any sequence $\bfk$, any subsequence $\bfk'$, any $d$, $b$, and $\epsilon$, we have
\begin{equation}\label{eq:inclusion} \bfs_\omega(d, \epsilon;\bfk') \subseteq \bfs_\omega(d, \epsilon;bfk),\quad \bfs_\omega(d; \bfk') \subseteq \bfs_\omega(d; \bfk),\quad \bfs_\chi(d, b;\bfk') \subseteq \bfs_\chi(d,b;\bfk).\end{equation} 

We now prove bounds on the size of these subsets. The global area and topology bounds from Corollary~\ref{cor:closed_curves} are an essential ingredient in our arguments. 

\begin{lem}\label{lem:pigeonhole_omega}
For any $d \geq 1$, $\epsilon > 0$, and sequence $\bfk$, there exists a subsequence $\bfk'$ such that
\begin{equation}\label{eq:action_bounds} \#\mathbf{r}_\omega(d, \epsilon; \bfk') \leq 8m\epsilon^{-1}d.\end{equation}
\end{lem}

\begin{proof}
Define $N := \lfloor 8m\epsilon^{-1}d \rfloor + 1$. Assume for the sake of contradiction that for any subsequence $\bfk'$, we have $\#\bfs_\omega(d, \epsilon; \bfk') \geq N$. The proof will take $3$ steps.

\noindent\textbf{Step $1$:} This step proves that, given our assumptions, there exists $\bfk' = \{k_j\} \subseteq \bfk$ and a finite subset $\{s_1, \ldots, s_N\}$ with the following property. For each $i$, there exists a sequence of intervals $\cL_{j,i}$ such that (i) $\cL_{j,i} \to s_i$ and (ii) for each $j$, the surfaces $C_{d,j,i} := C^{\cL_{j,i}}_{d,k_j}$ have action $\int_{C_{d,j,i}} u_{d,k_j}^*\omega \geq \epsilon$. 

The proof follows from repeated application of the following inductive step. For any subsequence $\bfk'$, call a point $s \in (-1, 1)$ satisfying (i) and (ii) a $(\bfk', \epsilon)$-point. Note that if $s$ is a $(\bfk', \epsilon)$-point, it is a $(\bfk'', \epsilon)$-point for any subsequence $\bfk'' \subseteq \bfk'$. Suppose a subsequence $\bfk'$ has exactly $N'$ $(\bfk', \epsilon)$-points $\{s_1, \ldots, s_{N'}\}$. Then, if $N' < N$, we claim that there exists a subsequence $\bfk'' \subseteq \bfk'$ with at least $N' + 1$ $(\bfk'', \epsilon)$-points. To prove this, it suffices to find some $\bfk''$ with a $(\bfk'', \epsilon)$-point not equal to any of the $s_i$. By our assumptions, we have $\bfs_\omega(d, \epsilon; \bfk') > N'$, so there must exist some $s \in \bfs_\omega(d, \epsilon; \bfk')$ which is not equal to any of the $s_i$. By definition of $\bfs_\omega(d, \epsilon; \bfk')$, there must exist a subsequence $\bfk''$ such that $s$ is a $(\bfk'', \epsilon)$-point. 

By applying the above inductive step at most $N$ times, we find a subsequence $\bfk'$ that has a set $\{s_1, \ldots, s_N\}$ of $N$ distinct $(\bfk', \epsilon)$-points. 

\noindent\textbf{Step $2$:} Let $\bfk' = \{k_j\}$ and $\{s_1, \ldots, s_N\}$ be the subsequence and levels from Step $1$. This step proves that, for sufficiently large $k$, the surfaces $C_{d,j,i}$ are disjoint. The surface $C_{d,j,i}$ is constructed by capping $\Sigma_{d,j,i} := (a \circ u_{d,k_j})^{-1}(k_j \cdot \cL_{j,i})$ with short connected components of $C_{d,k_j}\,\setminus\,\operatorname{Int}(\Sigma_{d,j,i})$. Thus, every point in $(a \circ u_{d,k_j})(C_{d,j,i})$ is a distance of at most $2$ from $k_j \cdot \cL_{j,i}$. The intervals $k_j \cdot \cL_{j,i}$ and $k_j \cdot \cL_{j,i'}$ are very far from each other for $i \neq i'$ and $k$ sufficiently large, so it follows that $C_{d,j,i}$ and $C_{d,j,i'}$ are disjoint. 

\noindent\textbf{Step $3$:} This step finishes the proof. By Corollary~\ref{cor:closed_curves} and Lemma~\ref{lem:j_tame}, we have 
\begin{equation}\label{eq:action_pigeonhole}4md = \int_{C_{d,k_j}} u_{d,k_j}^*\Omega_{k_j} \geq \frac{1}{2}\sum_{i=1}^N \int_{C_{d,j,i}} u_{d,k_j}^*\omega_{k_j} \geq N\epsilon/2 > 4md.\end{equation}

This is the desired contradiction. 
\end{proof}

Lemma~\ref{lem:pigeonhole_omega} implies that, outside of a countable set of levels, no action accumulates at all.

\begin{lem}\label{lem:pigeonhole_omega2}
Fix any $d \geq 1$ and any sequence $\bfk$. Then, there exists a subsequence $\bfk' \subseteq \bfk$ such that the set $\bfs_\omega(d; \bfk')$ is countable. 
\end{lem}

\begin{proof}
By Lemma~\ref{lem:pigeonhole_omega} and a diagonal argument, there exists a subsequence $\bfk'$ such that $\bfs_\omega(d, 1/n; \bfk')$ is finite for every $n \geq 1$. It follows that $\bfs_\omega(d; \bfk') = \bigcup_{n \geq 1} \bfs_\omega(d, 1/n; \bfk')$ is countable. 
\end{proof}

Now, we bound the size of $\bfs_\chi(d, b; \bfk)$. 

\begin{lem}\label{lem:pigeonhole_chi}
There exists a constant $c_4 \geq 1$ such that the following holds for any $d \geq 1$, $b \geq 1$, and sequence $\bfk$. There exists a subsequence $\bfk' \subseteq \bfk$ such that
\begin{equation}\label{eq:topology_bounds} \#\mathbf{r}_\chi(d, b; \bfk') \leq 4m(c_4 d + 4m b^{-1}d^2).\end{equation}
\end{lem}

To prove Lemma~\ref{lem:pigeonhole_chi}, we need the following technical lemma. It claims that any $Z \subset C_{d,k}$ containing an interior point with vertical distance at least $1$ from its boundary has a lower bound on its symplectic area.

\begin{lem}\label{lem:quantization}
There exists a constant $c_4 \geq 1$ such that the following holds for any $d$ and $k$. For any compact and irreducible Riemann surface $Z \subset C_{d,k}$, we have
\begin{equation}\label{eq:quantization}\int_Z u_{d,k}^*\Omega_k \geq c_4^{-1} > 0.\end{equation}
provided that:
\begin{enumerate}[(\roman*)]
\item There exists $a_0 \in (-k, k)$ such that $u_{d,k}(\partial Z) \subset \{a_0\} \times Y$. 
\item There exists some $\zeta \in Z$ such that $u_{d,k}(\zeta) \not\in [a_0 - 2, a_0 + 2] \times Y$. 
\end{enumerate}
\end{lem}

\begin{proof}
Fix a compact and irreducible surface $Z \subset C_{d,k}$ satisfying (i) and (ii). We prove the bound \eqref{eq:quantization} in two cases. The proof will take $2$ steps. Each step focuses on one case.

\noindent\textbf{Step $1$:} This step proves \eqref{eq:quantization} in the case where $u_{d,k}(Z) \subset [-L_k, L_k] \times Y$. Fix any $a_1 \in [-L_k, L_k]$ such that $a_1$ is either equal to $\inf_{\zeta \in Z} (a \circ u_{d,k})(\zeta)$ or $\sup_{\zeta \in Z} (a \circ u_{d,k})(\zeta)$ and fix any $\zeta_*$ such that $(a \circ u_{d,k})(\zeta_*) = a_1$. By (ii), it follows that $a_1 \not\in [a_0 - 2, a_0 + 2]$. Choose some $r \in (1, 2)$ such that $a_1 \pm r$ are regular values of $a \circ u_{d,k}$ and define $Z_*$ to be the irreducible component of $(a \circ u_{d,k})^{-1}([a_1 - r, a_1 + r]) \cap Z$ containing $\zeta_*$. 

By (i), it follows that $(a \circ u_{d,k})(\partial Z_*) \cap [a_1 - 1, a_1 + 1] = \emptyset$. It follows from Proposition~\ref{prop:fh_quantization} and Remark~\ref{rem:stable_constants} that there exists a constant $c_* \geq 1$ such that $\int_{Z_*} u_{d,k}^*\omega_k \geq 2c_*^{-1}$. By Lemma~\ref{lem:j_tame}, it follows that 
$$\int_Z u_{d,k}^*\Omega_k \geq \frac{1}{2}\int_{Z_*}u_{d,k}^*\omega_k \geq c_*^{-1}.$$

\noindent\textbf{Step $2$:} This step proves \eqref{eq:quantization} in the case where $u_{d,k}(Z)$ is not contained in $[-L_k, L_k] \times Y$. Let $v: Z \to \bW$ denote the $\widecheck{J}_k$-holomorphic curve defined by the restriction of $f_k^{-1} \circ u_{d,k}$ to $Z$. Choose a point $\zeta \in Z$ such that $v(\zeta) \not\in U_{\delta_3}$. Choose some $\delta \in (\delta_4, 2\delta_4)$ which is a regular value of $H \circ v$ and define $Z_*$ to be the irreducible component of $Z\,\setminus\,v^{-1}(U_\delta)$ containing $\zeta$. Let $v_*$ be the restriction of $v$ to $Z_*$. Note that $v_*$ is $\widecheck{J}$-holomorphic, contains a point outside $U_{\delta_3}$, and has $v_*(\partial Z_*) \cap \bW\,\setminus\,U_{2\delta_4} = \emptyset$. It follows from the monotonicity bound \cite[Proposition $3.4$]{Fish11} that there exists a constant $c_* = c_*(\Omega, \widecheck{J}, \delta_3, \delta_4) > 0$ such that
$$\int_Z u_{d,k}^*\Omega_k \geq \int_{Z_*} u_{d,k}^*\Omega \geq c_*^{-1}.$$
\end{proof}

Now we prove Lemma~\ref{lem:pigeonhole_chi}. 

\begin{proof}[Proof of Lemma~\ref{lem:pigeonhole_chi}]
Let $c_4$ be the constant from Lemma~\ref{lem:quantization}. Assume for the sake of contradiction that, for any subsequence $\bfk' \subseteq \bfk$, we have $\#\bfs_\chi(d, b; \bfk') > 4mb^{-1}d(4md + c_4)$. Given this assumption, an analogous algorithm to Step $1$ of the proof of Lemma~\ref{lem:pigeonhole_omega} produces a subsequence $\bfk' = \{k_j\}$ and a finite subset $\{s_1, \ldots, s_N\}$, where $N := \lfloor 4mb^{-1}d(4md + c_4)\rfloor $, with the following property. For each $i$, there exists a sequence of intervals $\cL_{j,i}$ such that (i) $\cL_{j,i} \to \{s_i\}$ and (ii) for each $j$, the surfaces $C_{d,j,i} := C^{\cL_{j,i}}_{d,k_j}$ have an irreducible component $Z_{d,j,i}$ with $\chi(Z_{d,j,i}) \leq -b$. 

Define the surfaces $\Sigma_{d,j,i} := (a \circ u_{d,k_j})^{-1}(k_j \cdot \cL_{k_j,i})$ as in the proof of Lemma~\ref{lem:pigeonhole_omega}. Define $S_{d,j} := C_{d,k_j}\,\setminus\,\bigcup_{i=1}^N Z_{d,j,i}$ to be the complement of the surfaces $Z_{d,j,i}$. We claim that 
\begin{equation}\label{eq:top_pigeonhole}\chi(S_{d,j}) \leq 4m c_4 d.\end{equation}

Each connected component $F$ of $S_{d,j}$ is either (i) a compact surface with $\chi(F) \leq 0$ or (ii) a closed disk sharing a boundary component with one of the surfaces $Z_{d,j,i}$. Therefore, prove \eqref{eq:top_pigeonhole}, it suffices to bound the number of components of the second type. Any component $F$ of the second type contains an irreducible component $Z$ of $C_{d,j,i}\,\setminus\,\Sigma_{d,j,i}$ which is not ``short''. Such a component $Z$ satisfies the conditions of Lemma~\ref{lem:quantization}. It follows from Lemma~\ref{lem:quantization} that $\int_F u_{d,k_j}^*\Omega_{k_j} \geq c_4^{-1}$. By Corollary~\ref{cor:closed_curves}, we have $\int_{C_{d,k_j}} u_{d,k_j}^*\Omega_{k_j} = 4md$, so there are at most $4m c_4 d$ components of the second type. The bound \eqref{eq:top_pigeonhole} follows. 

By Corollary~\ref{cor:closed_curves}, we have
\begin{equation}\label{eq:top_pigeonhole_2}\chi(C_{d,k_j}) = 2\#\pi_0(C_{d,k_j}) - 2G(C_{d,k_j}) \geq - 2G_a(C_{d,k_j}) \geq -16m^2d^2.\end{equation}

The first inequality follows because arithmetic genus is bounded below by genus. By \eqref{eq:top_pigeonhole}, we have 
\begin{equation}\label{eq:top_pigeonhole_3}\chi(C_{d,k_j}) = \chi(S_{d,j}) + \sum_{i=1}^N \chi(Z_{d,j,i}) \leq 4mc_4d - Nb < -16m^2 d^2.\end{equation}

The bounds \eqref{eq:top_pigeonhole_2} and \eqref{eq:top_pigeonhole_3} cannot hold simultaneously, so we arrive at a contradiction. 
\end{proof}

\subsubsection{Controlled accumulation implies almost cylinders}

Define a constant $\epsilon_d := 128m d^{-1}$ for every $d \geq 1$ and define $b_* := 64m^2$. We prove that every element of the stretched limit set outside of the accumulation sets $\bfs_\omega(d, \epsilon_d; \bfk) \cup \bfs_\chi(d, b_*; \bfk)$ is a $\delta$-almost cylinder, where $\delta \to 0$ as $d \to \infty$. 

\begin{prop}\label{prop:almost_cylinders}
For any integer $n \geq 1$, there exists some $d_n \geq 1$ such that the following holds for all $d \geq d_n$ and any sequence $\bfk = \{k_j\}$. Fix any $s \in (-1, 1) \,\setminus\,(\mathbf{r}_\omega(d, \epsilon_d; \bfk) \cup \mathbf{r}_\chi(d, b_*; \bfk))$. Then, for any $(\Xi, s) \in \cX_d(\bfk)$, the set $\Xi$ is a $1/n$-almost cylinder. 
\end{prop}

The key technical input to Proposition~\ref{prop:almost_cylinders} is the following result, which asserts that holomorphic curves with bounded Euler characteristic and sufficiently low action are $1/n$-almost cylinders away from the boundary. We prove it using Proposition~\ref{prop:local_area_bound}. 

\begin{prop}\label{prop:low_action_almost_cyl}
Fix any $(\bar\eta, \bar{J}) \in \cD([-8,8] \times Y)$ and integers $n \geq 1$ and $b \geq 0$. Then there exists a stable constant $\epsilon_5 = \epsilon_5(\bar\eta, \bar{J}, n, b) > 0$ such that the following holds. Let $u: C \to [-8, 8] \times Y$ be a compact, irreducible $\hat{J}$-holomorphic curve such that
\begin{enumerate}[(\roman*)]
\item $(a \circ u)(\partial C) \cap [-4, 4] \neq \emptyset$;
\item $\chi(C) \geq -b$;
\item $\int_{C} u^*\bar{\omega} \leq \epsilon_5$.
\end{enumerate}

Then the set $u(C) \cap (-1, 1) \times Y \in \cK( (-1, 1) \times Y)$ is a $1/n$-almost cylinder with respect to $(\bar\eta, \bar{J})$. 
\end{prop}

\begin{proof}
Assume for the sake of contradiction that the proposition is false. Then, there exists a convergent sequence $(\bar{\eta}_k, \bar{J}_k) \to (\bar{\eta}, \bar{J})$ and compact, irreducible $\bar{J}_k$-holomorphic curves $u_k: C_k \to [-8, 8] \times Y$ satisfying (i), (ii), and the bound $\int_{C_k} u_k^*\bar{\omega}_k \leq k^{-1}$, such that the sets $u_k(C_k) \cap (-1, 1) \times Y$ are not $1/n$-almost cylinders with respect to $(\bar\eta_k, \bar{J}_k)$ for any $k$. The remainder of the proof will derive a contradiction in $2$ steps.

\noindent\textbf{Step $1$:} This step proves that, after passing to a subsequence, there exists some non-empty, compact, $\bar{R}_\eta$-invariant set $\Lambda \subseteq Y$ such that
\begin{equation}\label{eq:almost_cylinders_1} u_k(C_k) \cap (-2, 2) \times Y \to (-2, 2) \times \Lambda\end{equation}
in the Hausdorff topology on $\cK( (-2, 2) \times Y)$. We prove \eqref{eq:almost_cylinders_1} using Proposition~\ref{prop:local_area_bound}. The proof is all but identical to other recent results (see \cite[Proposition $4.47$]{FH23} or \cite[Theorem $6$]{CGP24}), so we only provide a sketch. After passing to a subsequence, the slices $u_k(C_k) \cap (-2, 2) \times Y$ converge to $\Xi \in \cK((-2,2) \times Y)$. Since each $C_k$ is irreducible, it follows from (i) that $u_k(C_k) \cap [-1,1] \times Y$ is non-eempty for each $k$. Therefore, the set $\Xi$ is non-empty. 

To prove \eqref{eq:almost_cylinders_1}, it suffices to show that for any point $z = (a, y) \in \Xi$, there exists $\delta > 0$ such that for any $\tau \in (-\delta, \delta)$, we have $(a + \tau, y) \in \Xi$ and $(a, \phi^\tau(y)) \in \Xi$. Then there is a sequence of points $\zeta_k \in C_k$ such that $z_k := u_k(\zeta_k) \in (-2,2)\times Y$ and $z_k \to z$. The surfaces $S_k := S_{\epsilon_3}(\zeta_k)$ have uniformly bounded area by Proposition~\ref{prop:local_area_bound}, and have uniformly bounded genus since $\chi(C_k) \geq -b$ for each $k$. Thus, by target-local Gromov compactness \cite{Fish11}, the restrictions $v_k := u_k|_{S_k}$, after passing to a subsequence and shrinking the surfaces $S_k$ slightly, converge to a map $v: S \to [-8,8] \times Y$ with $\int_{S} v^*\bar\omega = 0$ and $z \in v(S)$. By Lemma~\ref{lem:zero_action}, $v(S)$ lies inside $\bR \times \Gamma(\bR)$, where $\Gamma: \bR \to Y$ is the unique $R_\eta$-trajectory with $\Gamma(0) = y$, and the desired property of $z$ follows.  

\noindent\textbf{Step $2$:} Fix any sequence of points $z_k = (a_k, y_k) \in u_k(C_k) \cap (-2,2) \times Y$ and any sequence $\tau_k \in (-2, 2)$. This step completes the proof by showing 
\begin{equation}\label{eq:almost_cylinders_2} 
\begin{split}
\limsup_{k \to \infty} \operatorname{dist}_{\bar{g}_k}( (\tau_k, y_k), u_k(C_k) \cap (-2,2)\times \Lambda) = 0,\\ \limsup_{k \to \infty} \operatorname{dist}_{\bar{g}_k}( (a_k, \phi^{\tau_k}(y_k)), u_k(C_k) \cap (-2,2)\times \Lambda) = 0.
\end{split}
\end{equation}

It is sufficient to establish \eqref{eq:almost_cylinders_2} to complete the proof, because \eqref{eq:almost_cylinders_2} implies that $u_k(C_k) \cap (-1,1)\times Y$ is a $1/n$-almost cylinder for large enough $k$, which gives a contradiction. It follows from the Hausdorff convergence in \eqref{eq:almost_cylinders_1} that
$$\limsup_{k \to \infty} \operatorname{dist}_{\bar{g}_k}(z_k, [-2,2] \times \Lambda) = 0.$$

For each $k$, choose a point $z'_k = (a_k', y_k') \in [-2, 2] \times \Lambda$ such that 
$$\operatorname{dist}_{\bar{g}_k}(z_k, z'_k) = \operatorname{dist}_{\bar{g}_k}(z_k, [-1,1] \times \Lambda).$$

Thus, we have $|a_k - a_k'| \to 0$ and $\operatorname{dist}(y_k, y_k') \to 0$, where the distance on $Y$ is taken with respect to some arbitrary but fixed Riemannian metric. It follows that
\begin{equation}\label{eq:almost_cylinders_3} \limsup_{k \to \infty} \operatorname{dist}_{\bar{g}_k}( (\tau_k, y_k), (\tau_k, y_k')) = 0,\quad \limsup_{k \to \infty} \operatorname{dist}_{\bar{g}_k}( (a_k, \phi^{\tau_k}(y_k)), (a_k, \phi^{\tau_k}(y'_k)) = 0.\end{equation}

Since $(\tau_k, y_k')$ and $(a_k, \phi^{\tau_k}(y_k'))$ lie in $(-2, 2) \times \Lambda$ for each $k$, it follows from \eqref{eq:almost_cylinders_3} that
\begin{equation}\label{eq:almost_cylinders_4} \limsup_{k \to \infty} \operatorname{dist}_{\bar{g}_k}( (\tau_k, y_k), (-2, 2) \times \Lambda) = 0,\quad \limsup_{k \to \infty} \operatorname{dist}_{\bar{g}_k}( (a_k, \phi^{\tau_k}(y_k)), (-2, 2) \times \Lambda) = 0.\end{equation}

Combining \eqref{eq:almost_cylinders_4} with \eqref{eq:almost_cylinders_1} proves \eqref{eq:almost_cylinders_2}. 
\end{proof}

Now, we are ready to prove Proposition~\ref{prop:almost_cylinders}.

\begin{proof}[Proof of Proposition~\ref{prop:almost_cylinders}]
Let $d$ be sufficiently large so that $64md^{-1} \leq \epsilon_5(\eta, J, n, b_*)$, where $\epsilon_5$ denotes the stable constant from Proposition~\ref{prop:low_action_almost_cyl}. Fix any $(\Xi, s) \in \cX_d(\bfk)$, where $s \not\in \mathbf{r}_\omega(d, \epsilon_d; \bfk) \cup \mathbf{r}_\chi(d, b_*; \bfk)$. Then, there exists a sequence $a_j \in (-k_j, k_j)$ such that $k_j^{-1}a_j \to s$ and the slices
$$\Xi_j := \tau_{a_j} \cdot (u_{d,k_j}(C_{d,k_j}) \cap (a_j - 1, a_j + 1) \times Y)$$
converge in the Hausdorff topology to $\Xi$. For each $j$, choose a closed, connected interval $\mathcal{L}_j$ with endpoints in $\cR$, such that
$$[k_j^{-1}(a_j - 4), k_j^{-1}(a_j + 4)] \subset \mathcal{L}_j \subset [k_j^{-1}(a_j - 6), k_j^{-1}(a_j + 6)].$$

Consider the capped slices $Z_j := C_{d,k_j}^{\mathcal{L}_j}$. Note that $u_{d,k_j}(Z_j) \subset [a_j - 8, a_j + 8] \times Y$ for each $j$. Since $s \not\in \mathbf{r}_\chi(d, b_*; \bfk)$, we have that, for sufficiently large $j$, each irreducible component of $Z_j$ has Euler characteristic at most $-64m^2$. Since $s \not\in \mathbf{r}_\omega(d, \epsilon_d; \bfk)$, we have that, for sufficiently large $j$, the action $\int_{Z_j} u_{d,k_j}^*\omega_{k_j}$ is at most $64md^{-1} \leq \epsilon_5(\eta, J, n, b_*)$. By Proposition~\ref{prop:low_action_almost_cyl}, the set $\Xi_j = u_{d,k_j}(Z_j) \cap (-1, 1) \times Y$ is a $1/n$-almost cylinder with respect to $(\eta_{k_j}^{a_j}, J_{k_j}^{a_j})$ for each $j$. By Lemma~\ref{lem:almost_cyl_limits} and Corollary~\ref{cor:geometric_convergence}, it follows that $\Xi$ is a $1/n$-almost cylinder with respect to $(\eta, J)$. 
\end{proof}

\subsubsection{Existence of a single proper compact invariant set} The following result asserts that at levels where there is no $\omega$-accumulation and controlled $\chi$-accumulation, any element of the stretched limit set is a cylinder over a proper, compact, $R_\eta$-invariant set. A key part of the proof is the use of intersection theory to show that the invariant set is proper. This part is inspired by \cite[Theorem $7$]{FH23}, which shows that certain feral holomorphic curves have ends limiting to proper invariant sets. The proof of that result required some curvature bounds (see \cite[Theorem $6$]{FH23}) for feral curves that are not available to us. We use alternative arguments to circumvent this issue.

\begin{prop}\label{prop:intersection_argument}
Fix $d \geq 1$, $\bfk = \{k_j\}$, and $s \in (-1,1)\,\setminus\,(\bfs_\omega(d; \bfk) \cup \mathbf{r}_\chi(d, b_*; \bfk))$. Then, for any $(\Xi, s) \in \cX_d(\bfk)$, the set $\Xi$ is equal to $(-1, 1) \times \Lambda$, where $\Lambda \subseteq Y$ is a proper, non-empty, compact, $R_\eta$-invariant set. 
\end{prop}

As preparation, we introduce the moduli space of degree $1$ holomorphic spheres and its relevant properties. For each $k$, let $\cM_k$ denote the moduli space of $J_k$-holomorphic spheres in $\bW_k$ that are Poincar\'e dual to $[\Omega]$ and pass through a fixed point $w_\infty \in \bD$. By the adjunction inequality \cite[Theorem $1.3$]{McDuff91}, each element of $\cM_k$ is embedded. 

\begin{lem}\label{lem:spheres}
The moduli space $\cM_k$ satisfies the following properties for each $k$:
\begin{enumerate}[(a)]
\item For any $k$ and any $\mathbb{S} \in \cM_k$, the algebraic intersection number of $u_{d,k}$ and $\mathbb{S}$ is equal to $4md$. 
\item For any $k$ and any $a_0 \in [-k,k]$, there exists some $\mathbb{S} \in \cM_k$ such that 
$$\inf_{z \in \mathbb{S} \cap [-k,k] \times Y} a(z) = a_0.$$
\end{enumerate}
\end{lem}

\begin{proof}
Property (a) follows from Corollary~\ref{cor:closed_curves}. Property (b) follows from an open-closed argument as in the very similar \cite[Proposition $3.5$]{FH23}. Extend the coordinate $a$ to a smooth function $a: \bW_k \to \bR$, such that $a \geq k$ on $\bW_+$ and $a \leq -k$ on $\bW_-$. The range $a(\bW_k)$ is a compact interval $\cI \subset \bR$ containing $[-k, k]$. Let $\mathbf{a}$ denote the set of all $a_0 \in \cI$ for which there exists $\mathbb{S} \in \cM_k$ such that $\inf_{z \in \mathbb{S}} a(z) = a_0$. Since $\bD \in \cM_k$, the set $\mathbf{a}$ is non-empty. By automatic transversality of the moduli space \cite{HLS97}, the set $\mathbf{a}$ is open. By the Gromov compactness theorem, and the lack of bubbling due to the degree restriction, the set $\mathbf{a}$ is closed. We conclude that $\mathbf{a} = \cI$. Therefore, $[-k, k] \subset \mathbf{a}$, which is equivalent to Property (b). 
\end{proof}

The following lemma proves everything but properness of $\Lambda$. 

\begin{lem}\label{lem:intersection_argument}
Fix $d$, $\bfk$, $s$ as in Proposition~\ref{prop:intersection_argument}. Then, for any $(\Xi, s) \in \cX_d(\bfk)$, the set $\Xi$ is equal to $(-1, 1) \times \Lambda$, where $\Lambda \subseteq Y$ is a non-empty, compact, $R_\eta$-invariant set. 
\end{lem}

\begin{proof}
Fix any $n \geq 1$. Observe that $s \not\in \mathbf{r}_\omega(d, \epsilon_5; \bfk)$, where $\epsilon_5 = \epsilon_5(\eta, J, n, b_*) > 0$ is the constant from Proposition~\ref{prop:low_action_almost_cyl}. By Proposition~\ref{prop:almost_cylinders}, $\Xi$ is a $1/n$-almost cylinder. Therefore, $\Xi$ is a $1/n$-almost cylinder for every $n$, so the claim follows from Lemma~\ref{lem:almost_cyl_to_cyl}. 
\end{proof}

We are now ready to prove the proposition. 

\begin{proof}
By Lemma~\ref{lem:intersection_argument}, $\Xi = (-1, 1) \times \Lambda$ where $\Lambda \subseteq Y$ is a non-empty, compact, $R_\eta$-invariant set. It suffices to show that $\Lambda$ is a proper subset of $Y$. Assume for the sake of contradiction that $\Lambda = Y$. The proof will take $5$ steps. 

\noindent\textbf{Step $1$:} For each $j$, define $a_j := k_j s$. By Lemma~\ref{lem:spheres}(b), there exists for each $j$ a sphere $\mathbb{S}_j \in \cM_{k_j}$ such that $\inf_{z \in \mathbb{S}_j \cap [-k_j, k_j] \times Y} a(z) = a_j$. This step extracts a certain limiting disk from the sequence $\{\mathbb{S}_j\}$. For simplicity, assume that $\mathbb{S}_j$ intersects $\{a_j + 1/2\} \times Y$ transversely for each $j$. Choose $z_j \in \mathbb{S}_j$ such that $a(z_j) = a_j$. Let $\Sigma_j$ denote the connected component of $z_j$ in $\mathbb{S}_j \cap [a_j, a_j + 1/2] \times Y$. Let $v_j: \Sigma_j \to (-1, 1) \times Y$ denote the composition of $\Sigma_j \hookrightarrow \mathbb{S}_j$ with $\tau_{a_j}$. 

Each map $v_j$ is $J_{k_j}^{a_j}$-holomorphic; see \S\ref{subsubsec:geometric_convergence}. The surfaces $\Sigma_j$ each have zero genus. They have uniformly bounded area by Proposition~\ref{prop:fh_area_bound} and Remark~\ref{rem:stable_constants}. Therefore, by target-local Gromov compactness \cite{Fish11} and Corollary~\ref{cor:geometric_convergence}, after passing to a subsequence in $j$, there exists a sequence of surfaces $\overline{\Sigma}_j \subseteq \Sigma_j$ and a compact, connected $J$-holomorphic curve $\overline{v}: \overline{\Sigma} \to (-1, 1) \times Y$ such that (i) $\overline{\Sigma}_j$ contains $(a \circ v_j)^{-1}([1/4,1/4] \times Y)$ and (ii) the restrictions $\overline{v}_j := v_j|_{\overline{\Sigma}_j}$ converge to $\overline{v}$ in the Gromov topology. 

For any $r > 0$, let $D(r) \subset \mathbb{C}$ denote the closed disk of radius $r$ centered at the origin. Choose some small $r_0 > 0$ and an embedding $\theta: D(r_0) \hookrightarrow \overline{\Sigma}$ with image disjoint from any nodal points and from $\partial\overline{\Sigma}$, such that $\overline{v} \circ \theta$ is an embedding and such that $(\overline{v} \circ \theta)^*\omega \neq 0$ at any point in $D(r_0)$. 

\noindent\textbf{Step $2$:} This step introduces the so-called ``vertical foliation'' and uses it to define tubular neighborhood coordinates for the map $\overline{v} \circ \phi$. For any $y \in Y$, let $\Gamma_y: \bR \to Y$ denote the unique trajectory of $R_\eta$ for which $\Gamma_y(0) = y$. The planes $\bR \times \Gamma_y(\bR)$ form a smooth $2$-dimensional foliation of $\bR \times Y$ called the \emph{vertical foliation}. Each leaf is an immersed $J$-holomorphic plane, parameterized as follows. For any point $z = (a, y) \in \bR \times Y$, let $P_z: \mathbb{C} \to \bR \times Y$ be the map $s + t \cdot i \mapsto (a + s, \Gamma_y(t))$. The map $P_z$ is an injective $J$-holomorphic immersion with image $\bR \times \Gamma_y(\bR)$. 

Since $(\overline{v} \circ \theta)^*\omega \neq 0$ everywhere, the map $\overline{v} \circ \phi$ is transverse to the vertical foliation. We construct a tubular neighborhood such that each fiber is a $J$-holomorphic disk inside a vertical leaf. We formulate this precisely as follows. There exists an embedding $\Theta: D(r_0) \times D(r_0) \hookrightarrow \bR \times Y$ with the following properties:
\begin{itemize}
\item For any $\zeta \in D(r_0)$, we have $\Theta(\zeta, 0) = (\overline{v} \circ \theta)(\zeta)$;
\item For any $\zeta \in D_{r_0}$, write $z = (\overline{v} \circ \theta)(\zeta)$. The map $\Theta(\zeta, -): D(r_0) \to \bR \times Y$ is equal to the restriction of $P_z$ to the disk $D(r_0)$. 
\end{itemize}

The images $\overline{v}_j(\overline{\Sigma}_j)$ are graphical over $D(r_0/2)$ for large $j$. Since $\theta(D(r_0)) \subset \overline{\Sigma}$ does not contain any nodal points or boundary points, there exist embeddings $\theta_j: D(r_0) \to \overline{\Sigma}_j$ such that $\overline{v}_j \circ \theta_j$ converges to $\overline{v} \circ \theta$ in the $C^\infty$ topology. Setting $r_1 := r_0/2$, we observe that, for sufficiently large $j$, there exists a smooth map $x_j: D(r_1) \to D(r_1) \times D(r_0)$ such that 
$$(\Phi \circ x_j)(\zeta) = (\overline{v}_j \circ \theta_j)(\zeta).$$
for any $\zeta \in D(r_1)$. Moreover, the maps $x_j$ converge in the $C^\infty$ topology to the map $x: \zeta \mapsto (\zeta, 0)$. 

\noindent\textbf{Step $3$:} With the setup from Steps $1$ and $2$ complete, this step explores the consequences of our assumption that $\Lambda = Y$. We show that there exists $8md$ pieces of the curves $u_{d,k_j}(C_{d,k_j})$ that converge as $j \to \infty$ to vertical surfaces positively intersecting the disk $\overline{v} \circ \theta(D(r_0))$. 

Choose a finite set of $8md$ points $\{p_\ell\} \subset D(r_1)\,\setminus\,\partial D(r_1)$ and write $z_\ell := (\overline{v} \circ \theta)(p_\ell)$ for each $\ell$. By assumption, $\Lambda = Y$, so for each $\ell$, there exists a sequence of points $z_{\ell, j} \in u_{d,k_j}(C_{d,k_j}) \cap (a_j - 1, a_j + 1) \times Y$ such that $\tau_{a_j}(z_{\ell, j}) \to z_\ell$. For each $\ell$ and $j$, fix $\zeta_{\ell, j} \in u_{d,k_j}^{-1}(z_{\ell, j})$ and define $S_{\ell,j} := S_{\epsilon_3}(\zeta_{\ell,j})$. Write $w_{\ell,j}: S_{\ell,j} \to (-1, 1) \times Y$ for the composition of $u_{d,k_j}$, restricted to $S_{\ell,j}$, with the shift $\tau_{a_j}$. 

By Proposition~\ref{prop:closed_curves}, $G(S_{\ell,j})$ has an $\ell$- and $j$-independent upper bound. By Proposition~\ref{prop:local_area_bound}, the areas of the maps $w_{\ell,j}$ are uniformly bounded. By target-local Gromov compactness \cite{Fish11}, after passing to a subsequence in $j$, there exists a sequence of surfaces $\overline{S}_{\ell,j}$ and a compact connected $J$-holomorphic curve $\overline{w}_\ell: \overline{S}_\ell \to (-1, 1) \times Y$ such that (i) $\overline{w}_\ell(\partial\overline{S}_\ell)$ is disjoint from $B_{\epsilon_3/2}(z_\ell)$ and (ii) the restrictions $\overline{w}_{\ell,j} := w_{\ell,j}|_{\overline{S}_{\ell,j}}$ converge to $\overline{w}_\ell$ in the Gromov topology. 

For any $\ell$, we have $\int_{\overline{S}_\ell} \overline{w}_\ell^*\omega_{k_j}^{a_j} = 0$. By lemma~\ref{lem:zero_action}, we have $\overline{w}_{\ell}(\overline{S}_\ell) \subset P_{z_\ell}(\mathbb{C})$. By the open mapping theorem, $P_{z_\ell}^{-1} \circ \overline{w}_\ell(\overline{S}_\ell)$ contains an open neighborhood of $0$. Choose $r_2 < \min(r_0, \epsilon_3/4)$ such that, for any $\ell$, the image $(P_{z_\ell}^{-1} \circ \overline{w}_\ell)(\overline{S}_\ell)$ contains $D(r_2)$. Write $\hat{S}_\ell := (\overline{w}_\ell^{-1} \circ \Theta)(D(r_1) \times D(r_2))$. After shrinking $r_2$, we may assume without loss of generality that $\partial\hat{S}_\ell$ contains no nodal points or critical points.

Fix $\hat{S}_{\ell,j} := (\overline{w}_{\ell,j}^{-1} \circ \Theta)(D(r_1) \times D(r_2))$ for each $j$. Since $\partial\hat{S}_\ell$ contains no nodal points or critical points, it follows that $\overline{w}_{\ell, j}|_{\hat{S}_{\ell, j}}$ converges in the Gromov topology to $\overline{w}_\ell|_{\hat{S}_\ell}$. There exists a normalization $\widetilde{S}_\ell$ of $\hat{S}_\ell$ such that the following holds. Let $y_\ell: \widetilde{S}_\ell \to D(r_1) \times D(r_2)$ denote the lift of $\Theta^{-1} \circ \overline{w}_\ell$ to the normalization. Then for sufficiently large $j$, there are normalizations $\widetilde{S}_{\ell,j}$ of $\hat{S}_{\ell,j}$, lifts $h_{\ell,j}$ of $\Theta^{-1} \circ \overline{w}_{\ell,j}$, and diffeomorphisms $\psi_{\ell,j}: \widetilde{S}_\ell \to \widetilde{S}_{\ell,j}$ such that the maps $y_{\ell,j} := h_{\ell,j} \circ \psi_{\ell,j}$ converge to $y_\ell$ in the $C^0$ topology. 

\noindent\textbf{Step $4$:} The step proves that there exists some sufficiently large $j$ such that, for each $\ell$, the maps $x_j$ and $y_{\ell,j}$ have positive algebraic intersection number. Denote the algebraic intersection number by a dot, e.g. $x_j \cdot y_{\ell,j}$. For any $\ell$, the map $\Theta^{-1} \circ \overline{w}_\ell$ is a surjective holomorphic map from $\hat{S}_\ell$ onto $\Theta(\{p_\ell\} \times D(r_2))$. Thus, the lift $y_\ell$ is a positive degree cover of $\{p_\ell\} \times D(r_2)$. 

Recall that $x$ denotes the map $\zeta \mapsto (\zeta, 0)$. We have $x \cdot y_\ell > 0$. For any sufficiently large $j$, $x_j$ is $C^0$-close to $x$ and $y_{\ell,j}$ is $C^0$-close to $y_\ell$. The map $x_j$ is homotopic to $x$ via the straight-line homotopy $t \mapsto (1 - t)x_j + tx$. The map $y_{\ell,j}$ is homotopic to $y_\ell$ via the straight-line homotopy $t \mapsto (1-t)y_{\ell,j} + ty_\ell$. The only intersections along these homotopies are interior intersections. Thus, the algebraic intersection number is homotopy invariant, so 
$$x_j \cdot y_{\ell,j} = x \cdot y_\ell > 0$$
for sufficiently large $j$. 

\noindent\textbf{Step $5$:} This step completes the proof. By Step $4$, $x_j \cdot y_{\ell,j} > 0$ for each $\ell$ and for each large $j$. It follows that $\mathbb{S}_j \cdot w_{\ell,j} > 0$ for each sufficiently large $j$. The local intersections of $\mathbb{S}_j$ and $u_{d,k_j}$ are non-negative \cite[Theorem $1.1$]{McDuff91}. By summing over $\ell$, we conclude that $\mathbb{S}_j \cdot u_{d,k_j} \geq 8md$ for sufficiently large $j$. This contradicts Lemma~\ref{lem:spheres}(a). 

\end{proof}

\subsubsection{Completing the proof} We prove Proposition~\ref{prop:stretched_limit} using Lemma~\ref{lem:pigeonhole_omega}, Lemma~\ref{lem:pigeonhole_omega2}, Lemma~\ref{lem:pigeonhole_chi}, Proposition~\ref{prop:almost_cylinders}, and Proposition~\ref{prop:intersection_argument}. 

\begin{proof}[Proof of Proposition~\ref{prop:stretched_limit}]

Fix any $n \geq 1$ and any $d \geq d_n$, where $d_n$ is as in Proposition~\ref{prop:almost_cylinders}. By Lemma~\ref{lem:lim_set_conn}, Lemma~\ref{lem:pigeonhole_omega}, Lemma~\ref{lem:pigeonhole_omega2}, and Lemma~\ref{lem:pigeonhole_chi}, there exists a sequence $\bfk = \{k_j\}$ such that (i) the bounds \eqref{eq:action_bounds} and \eqref{eq:topology_bounds} are satisfied, (ii) $\bfs_\omega(d; \bfk)$ is countable, and (iii) $\pi_{\bR}^{-1}(\cJ) \subseteq \cX_d(\bfk)$ is connected for any closed, connected interval $\cJ$ of positive length. Write $\cX_d := \cX_d(\bfk)$. The proof will take $2$ steps. 

\noindent\textbf{Step $1$:} This step constructs a closed, connected interval $\cJ \subset (-1, 1)$ of positive length such that (a) $\cJ$ contains $id^{-2}$ for some $i \in [-d^2, d^2] \cap \mathbb{Z}$, (b) $\cJ$ does not intersect $\bfs_\omega(d, \epsilon_d; \bfk) \cup \bfs_\chi(d, b_*; \bfk)$, where $\epsilon_d = 128md^{-1}$ and $b_* = 64m^2$, and (c) $\cJ$ contains some $s \not\in \bfs_\omega(d; \bfk)$. Let $c_4$ denote the constant from Lemma~\ref{lem:pigeonhole_chi} and take $d \geq 64mc_4$. It follows from \eqref{eq:action_bounds} and \eqref{eq:topology_bounds} that
$$\#(\mathbf{r}_\omega(d; \bfk) \cup \mathbf{r}_\chi(d; \bfk)) \leq d^2/2.$$

From this bound, it follows that there exists a closed, connected interval $\cJ \subset (-1, 1)$ of length $2d^{-2}$ satisfying (b), i.e. $\cJ$ does not intersect $\bfs_\omega(d, \epsilon_d; \bfk) \cup \bfs_\chi(d, b_*; \bfk)$. We verify that $\cJ$ also satisfies (a) and (c). Since $\cJ$ has length $2d^{-2}$, (a) follows from the pigeonhole principle. Property (c) follows because, by (ii) above, $\bfs_\omega(d; \bfk)$ is countable, while $\cJ$ is uncountable.

\noindent\textbf{Step $2$:} This step completes the proof of the proposition. We verify that $\cJ$ satisfies each of Proposition~\ref{prop:stretched_limit}(a--d). By Property (a) from Step $1$, there exists $s = id^{-2}$ lying in $\cJ$. Observe that $k_js \in \bfa_{d,k_j}$ for each $j$, so the curves $u_{d,k_j}$ pass through $\{k_j s\} \times \bfp$ for each $j$. Take $\Xi$ to be any subsequential limit of the slices
$$\tau_{k_js} \cdot \Big(u_{d,k_j}(C_{d,k_j}) \cap (k_js - 1, k_js + 1) \times Y\Big) \subseteq (-1,1) \times Y.$$

Then we have $(\Xi, s) \in \cX_d$ and $\{0\} \times \mathbf{p} \subset \Xi$, which implies Proposition~\ref{prop:stretched_limit}(a). Proposition~\ref{prop:stretched_limit}(b) follows from Property (b) of Step $1$ and Proposition~\ref{prop:almost_cylinders}. Proposition~\ref{prop:stretched_limit}(c) follows from Property (c) of Step $1$ and Proposition~\ref{prop:intersection_argument}. Proposition~\ref{prop:stretched_limit}(d) follows from Property (iii) of the sequence $\bfk$. 
\end{proof}

\subsection{Other $4$-manifolds}\label{subsec:other_manifolds} We explain how to adapt the proof of Theorem~\ref{thm:main} to Hamiltonians on either $\mathbb{M}_0 := T^*\mathbb{S}^2$ or on $\mathbb{M}_1 := T^*\mathbb{T}^2$. 

\subsubsection{Compactification} We compactify to $(\mathbb{S}^2 \times \mathbb{S}^2, \Omega)$. The symplectic structure $\Omega$ is equal to $\omega \times \omega$, where $\omega$ is an area form on $\mathbb{S}^2$ with area $1$. Write $L_0 \subset \mathbb{S}^2 \times \mathbb{S}^2$ for the Lagrangian anti-diagonal and write $L_1 := \mathbb{S}^1 \times \mathbb{S}^1$, where we view $\mathbb{S}^1 \subset \mathbb{S}^2$ as the equator. Write $\mathbb{D}_0 \subset \mathbb{S}^2 \times \mathbb{S}^2$ for the diagonal and write $\mathbb{D}_1 := \{w\} \times \mathbb{S}^2$, where $w \in \mathbb{S}^2\,\setminus\,\mathbb{S}^1$. For each $i$, a small neighborhood of the zero section in $\mathbb{M}_i$ is symplectomorphic to a small neighborhood of the Lagrangian $L_i$, and is disjoint from the symplectic divisor $\mathbb{D}_i$. For each $i$, after an appropriate rescaling about the zero section, we view $Y$ as a hypersurface in $\mathbb{S}^2 \times \mathbb{S}^2$ which is disjoint from the divisor $\mathbb{D}_i$. 

\subsubsection{The curves $u_{d,k}$} The closed curves are defined using Proposition~\ref{prop:closed_curves_general}, a version of Proposition~\ref{prop:closed_curves} that holds for more symplectic $4$-manifolds, including $\mathbb{S}^2 \times \mathbb{S}^2$. 

\subsubsection{Intersection theory argument} We need to adapt Lemma~\ref{lem:pigeonhole_omega}, Lemma~\ref{lem:pigeonhole_chi}, Proposition~\ref{prop:low_action_almost_cyl}, and Proposition~\ref{prop:intersection_argument}. The first three results go through with minimal modifications, but Proposition~\ref{prop:intersection_argument} requires some more care. The method of proof, which involved intersecting the curves $u_{d,k}$ with a moduli space $\cM_k$ of embedded $J_k$-holomorphic spheres in $\mathbb{CP}^2$ remains the same, but we need a suitable analogue of $\cM_k$. When we are working in $\bM_0$, we replace $\cM_k$ with the moduli space of stable $J_k$-holomorphic spheres representing the class $[\{*\} \times \mathbb{S}^2] + [\mathbb{S}^2 \times \{*\}]$. Denote this space by $\cM_k^0$. It consists either of embedded $J_k$-holomorphic spheres representing $[\{*\} \times \mathbb{S}^2] + [\mathbb{S}^2 \times \{*\}]$ or of pairs of spheres connected by a nodal point, where the map on one sphere is an embedding representing the class $[\{*\} \times \mathbb{S}^2]$, and the map on the other sphere is an embedding representing the class $[\mathbb{S}^2 \times \{*\}]$. When we are working in $\bM_1$, we replace $\cM_k$ with the moduli space of embedded spheres representing the class $[\mathbb{S}^2 \times \{*\}]$. Note that $\mathbb{D}_i \in \cM_k^i$ for each $k$ and each $i$. The analogues of Lemma~\ref{lem:spheres} hold for the spaces $\cM_k^i$, and from here onwards the same argument goes through without modification.

%% file: lcy.tex
The goal of this section is to prove Theorem~\ref{thm:main2}. We prove the following equivalent result. 

\begin{thm}\label{thm:main2_local}
Let $H: \bR^4 \to \bR$ be a smooth function. Then, for any $s_0 \in \cR_c(H)$, there exists some $\delta_5 > 0$ and a full measure subset $\mathcal{Q} \subseteq \cR_c(H) \cap (s_0 - \delta_5, s_0 + \delta_5)$ such that the following holds. Fix any $s \in \mathcal{Q}$ and set $Y = H^{-1}(s)$. Then any compact $X_H$-invariant subset $\Lambda \subseteq Y$ containing $\cP(s)$ is either equal to $Y$ or is not locally maximal. 
\end{thm}

For the remainder of the section, fix a smooth function $H: \bR^4 \to \bR$. Fix any $s_0 \in \cR_c(H)$. Assume without loss of generality that $s_0 = 0$; we reduce to this case by replacing $H$ with $H - s_0$. Set $Y := H^{-1}(0)$. 

\subsection{An existence result for almost cylinders} To state our result, we recall several definitions and terms introduced in the previous section. We repeat the setup in \S\ref{subsubsec:compactification}, \S\ref{subsubsec:framed_hamiltonian}, and \S\ref{subsubsec:realized_hamiltonian}. We compactify to $\bW = \mathbb{CP}^2$, define a collar map $\iota: (-\delta_1, \delta_1) \times Y \to \bW$, and define a pair $(\widehat\eta, \widehat{J}) \in \cD( (-\delta_1, \delta_1) \times Y)$. Recall that, for any $s \in (-\delta_1, \delta_1)$, $(\eta^s, J^s) \in \cD(Y)$ denotes the pullback of $(\widehat\eta, \widehat{J})$ by the map $y \mapsto (s, y)$. Recall the definition of $\delta$-almost cylinders from \S\ref{subsec:almost_cyl}. Now, we state a new existence result for almost cylinders. 

\begin{prop}\label{prop:main2}
There exists a constant $\delta_5 \in (0, 1)$ such that the following holds for any finite set of points $\mathbf{p} \subset Y$ and any positive integer $n \geq 1$. There exists a subset $\mathcal{Q}_{\mathbf{p}, n} \subseteq (-\delta_5, \delta_5)$ of measure at least $(2 - 2^n)\delta_5$, and for each $s \in \mathcal{Q}_{\mathbf{p}, n}$ a connected subset $\cZ_{\mathbf{p}, n}^s \subseteq \cK((-1,1)\times Y)$ with the following properties:
\begin{enumerate}[(a)]
\item There exists some $\Lambda_n \in \cZ_{\mathbf{p}, n}^s$ such that $\{0\} \times \mathbf{p} \subset \Lambda_n$.
\item There exists some $\Lambda_n \in \cZ_{\mathbf{p}, n}^s$ equal to $(-1,1) \times \Xi_n$, where $\Xi_n$ is a finite union of closed orbits of $R^s$. 
\item Each $\Lambda \in \cZ_{\mathbf{p}, n}^s$ is a $1/n$-almost cylinder with respect to the pair $(\eta^s, J^s) \in \cD(Y)$. 
\end{enumerate}
\end{prop}

\subsection{Proof of Theorem~\ref{thm:main2_local}}

We defer the proof of Proposition~\ref{prop:main2} to \S\ref{subsec:almost_cyl_adiabatic_proof}. We first use it to complete the proof of Theorem~\ref{thm:main2_local}. 

\begin{proof}
Fix any finite set of points $\mathbf{p} \subset Y$ and any $n \geq 1$, let $\mathcal{Q}_{\mathbf{p},n} \subseteq (-\delta_5, \delta_5)$ be the set from Proposition~\ref{prop:main2} and, for any $s \in \mathcal{Q}_{\bfp, n}$, let $\cZ^s_{\bfp, n}$ be the collection of almost-cylinders from Proposition~\ref{prop:main2}. Define the set
$$\mathcal{Q}_{\bfp} := \bigcup_{\ell \geq 1} \bigcap_{n \geq \ell} \cQ_{\bfp, n} \subseteq (-\delta_5, \delta_5).$$ 

The rest of the proof will take $3$ steps. 

\noindent\textbf{Step $1$:} This step proves that $\cQ_{\bfp}$ has full measure in $(-\delta_5, \delta_5)$. Since $|\cQ_{\bfp,n}| \geq (2 - 2^{-n})\delta_5$ for each $n$, we have $|\bigcap_{n \geq \ell} \cQ_{\bfp,n} \geq (2 - 2^{-\ell+1})\delta_5|$. Taking the union over all $\ell$, we have $|\cQ_{\bfp}| = 2\delta_5$. 

\noindent\textbf{Step $2$:} For any $s \in \cQ_{\bfp}$, choose $n_j \to \infty$ such that (i) $s \in \cQ_{\bfp,n_j}$ for each $j$ and (ii) there exists $\Xi_j \in \cZ^s_{\bfp, n_j}$ such that the sequence $\{\Xi_j\}$ converges. Let $\cZ^s_{\bfp}$ denote the set of subsequential limit points of the sequence $\{\cZ^s_{\bfp, n_j}\}$ as $j \to \infty$. Since $\cZ^s_{\bfp, n_j}$ is connected for each $j$, the set $\cZ^s_{\bfp}$ is connected by (ii) and Lemma~\ref{lem:accumulation_points}. We claim that $\cZ^s_{\bfp}$ has the following properties:
\begin{enumerate}[(a)]
\item There exists some $\Xi \in \cZ_{\bfp}^s$ such that $\{0\} \times \bfp \subset \Xi$.
\item There exists a compact subset $\Lambda \subseteq \overline{\cP(s)}$ such that $(-1,1) \times \Lambda \in \cZ^s_{\bfp}$.
\item Each $\Xi \in \cZ_{\bfp}^s$ is equal to $(-1,1) \times \Lambda$, where $\Lambda \in \cK(Y)$ is a non-empty compact $R^s$-invariant set. 
\end{enumerate}

These properties follow from Proposition~\ref{prop:main2}(a--c), Lemma~\ref{lem:almost_cyl_limits}, and Lemma~\ref{lem:almost_cyl_to_cyl}. 

\noindent\textbf{Step $3$:} This step finishes the proof. Fix a sequence of finite subsets $\bfp_\ell \subset Y$ that converge to $Y$ in the Hausdorff topology as $\ell \to \infty$. Define $\cQ := \bigcap_\ell \cQ_{\bfp_\ell}$. Since $\cQ$ is a countable intersection of full measure subsets, it has full measure itself. Fix any $s \in \cQ$. After passing to a subsequence in $\ell$, there exists $\Xi_\ell \in \cZ^s_{\bfp_\ell}$ for each $\ell$ such that the sequence $\{\Xi_\ell\}$ converges. Define $\cZ^s$ to be the set of subsequential limit points of the sequence $\{\cZ^s_{\bfp_\ell}\}$. By Lemma~\ref{lem:accumulation_points}, $\cZ^s$ is connected. Let $\cY^s \subseteq \cK(Y)$ denote the image of $\cZ^s$ under the map $\Xi \mapsto \Xi \cap \{0\} \times Y$. By Property (c) above and Lemma~\ref{lem:intersection}, the set $\cY^s$ is connected. The set $\cY^s$ has the following properties, which follow from Properties (a--c) above:
\begin{enumerate}[(a')]
\item $Y \in \cY^s$.
\item There exists a compact $R^s$-invariant subset $\Lambda \subseteq \overline{\cP(s)}$ such that $\Lambda \in \cY^s$.
\item Each $\Lambda \in \cY^s$ is a non-empty compact $R^s$-invariant set. 
\end{enumerate}

Fix $s \in \cQ$ and an $R^s$-invariant set $\Lambda$ containing $\cP(s)$. We will show that if $\Lambda \neq Y$, then the set $\Lambda$ is not locally maximal. Let $\cY$ denote the image of $\cY^s$ under the map $\Lambda' \mapsto \Lambda \cup \Lambda'$. The family $\cY$ is connected by Lemma~\ref{lem:union}, consists of $R^s$-invariant sets by Property (c'), and contains both $Y$ and $\Lambda$ by Properties (a') and (b'). It follows that $\Lambda$ is not locally maximal. 
\end{proof}

\subsection{Geometric setup}\label{subsec:geom_setup_adiabatic} 

The geometric setup for this section mostly follows \S\ref{subsec:geom_setup}. However, we will use a different neck stretching procedure. Instead of stretching the neck around the hypersurface $Y$ as in \S\ref{subsec:geom_setup} above, we simultaneously stretch the neck around each of the hypersurfaces $H^{-1}(s)$ for small $s$. On a technical level, this only requires minor changes to the setup. The most significant change is the replacement of the sequence $\{\phi_k\}$ with a new sequence $\{\widehat{\phi}_k\}$.  

\subsubsection{Recollections from \S\ref{sec:invariant_sets}} We repeat the setup from \S\ref{subsec:geom_setup}, starting from \S\ref{subsubsec:base_acs} and ending at \S\ref{subsubsec:stretched_manifolds}.
\begin{itemize}
\item In \S\ref{subsubsec:base_acs} we defined an $\Omega$-tame almost-complex structure $\widecheck{J}$ on $\bW$ and fixed $\widehat{J} = \iota^*\widecheck{J}$; recall alaso that $\widehat{J}$ is $\widehat\eta$-adapted on $(-\delta_2, \delta_2) \times Y$. 
\item In \S\ref{subsubsec:deformed_acs} we defined deformations $\widecheck{J}_\phi$, agreeing with $\widecheck{J}$ on the bundle $\widecheck{\xi} = \ker(ds) \cap \ker(\widecheck{\lambda})$, which by Lemma~\ref{lem:j_tame} are tame if $1 - \phi$ is supported in $U_{\delta_3}$. 
\item In \S\ref{subsubsec:stretched_manifolds}, we defined stretched manifolds $\bW_\phi$, containing long necks $[-L_\phi, L_\phi] \times Y$, and diffeomorphisms $f_\phi: \bW \to \bW_\phi$. Let $\widehat\Omega_\phi$, $\widehat\eta_\phi = (\widehat\lambda_\phi, \widehat\omega_\phi)$, $\widecheck{\xi}_\phi$, $\widehat{J}_\phi$ denote the pushforwards by $f_\phi$ of $\Omega$, $\widecheck\eta = (\widecheck\lambda, \widecheck\omega)$, $\widecheck\xi$, $\widecheck{J}$, respectively.
\end{itemize} 

The remainder of our setup consists of minor variations on \S\ref{subsubsec:neck_stretching}, \S\ref{subsubsec:geometric_convergence}, and \S\ref{subsubsec:stretched_limit_set}. 

\subsubsection{Adiabatic neck stretching}\label{subsubsec:neck_stretching_adiabatic} Choose a constant $\delta_4 \in (0, \delta_3/2)$, a constant $\delta_5 \in (0, \delta_4/2)$ and, for each $k > 4\delta_5^{-1}$, a smooth function $\widehat\phi_k: \bW \to (0 ,1]$ with the following properties:
\begin{itemize}
\item The function $1 - \widehat\phi_k$ is supported on $U_{\delta_4}$.
\item For any $s \in (-\delta_3, \delta_3)$, $\widehat\phi_k$ is equal to a constant $\widehat\phi_k(s)$ on the hypersurface $H^{-1}(s)$.
\item For any $s \in (-\delta_3, \delta_3)$, we have $\widehat\phi_k(s) = \widehat\phi_k(-s)$.
\item The integral $L_{\widehat\phi_k} = \int_{-\delta_3}^0 \widehat\phi_k(t)^{-1} dt$ is at least $16k$.
\item $\widehat\phi_k(s) = \delta_5^{-1}k^{-1}$ for every $s \in (-k^{-1}, k^{-1})$.
\end{itemize}

We require the sequence $\{\widehat\phi_k\}$ to converge in the $C^\infty$ topology as $k \to \infty$. The sequence $\{\widehat\phi_k\}$ is identical to the sequence $\{\phi_k\}$ from \S\ref{subsubsec:neck_stretching}, apart from the last property. This change gives a different geometric convergence result, see Corollary~\ref{cor:geometric_convergence_adiabatic} below. 

We simplify the notation for the various geometric objects associated with $\widehat\phi_k$. Write $\widecheck{J}_k := \widecheck{J}_{\widehat\phi_k}$\footnote{This coincides with the notation for $\widecheck{J}_{\phi_k}$ from the previous section, but the two almost-complex structures never appear together so the ambiguity does not impact clarity.}. Define $\widehat L_k := L_{\widehat\phi_k} = \int_{-\delta_3}^0 \widehat\phi_k(t) dt$ and define a function $\widehat\Phi_k(s) := \int_{-\delta_3}^s \widehat\phi_k(t)^{-1} dt - L_k$. Then, write $\widehat\bW_k := \bW_{\widehat\phi_k}$ for the stretched manifolds and $\widehat f_k := f_{\widehat\phi_k}$ for the diffeomorphisms $\bW \to \widehat\bW_k$. Write $\widehat\Omega_k := \Omega_{\widehat\phi_k}$ and $\widehat J_k := \widehat{J}_{\phi_k}$. Write $\widehat\omega_k := \widehat{\omega}_{\phi_k}$, $\widehat\lambda_k := \widehat{\lambda}_{\phi_k}$, $\widehat\eta_k = (\widehat\lambda_k, \widehat\omega_k)$, $\widehat{R}_k := \widehat{R}_{\phi_k}$, and $\widehat\xi_k := \widehat{\xi}_{\phi_k}$. 

Fix any $k$ and any $a$ such that $[a - 8, a + 8] \subseteq [-\widehat{L}_k, \widehat{L}_k]$. Then, we define $(\widehat\eta_k^a, \widehat{J}_k^a) \in \cD( [-8, 8] \times Y)$ to be the pair defined by restriction of $(\widehat\eta_k, \widehat{J}_k)$ to $[a - 8, a + 8] \times Y$ and then translation by $-a$. The following convergence result is a consequence of Lemma~\ref{lem:geometric_convergence}. 

\begin{cor}\label{cor:geometric_convergence_adiabatic}
Fix any sequence $\{a_k\}$ such that $a_k \in (-k, k)$ for each $k$ and the sequence $\{\delta_5 k^{-1}a_k\}$ converges to some $s \in [-\delta_5, \delta_5]$. Then, we have $(\widehat\eta_k^{a_k}, \widehat{J}_k^{a_k}) \to (\eta^s, J^s)$ in $\cD([-8, 8] \times Y)$.
\end{cor}

\begin{proof}
Let $\widehat{s}_k: [-8, 8] \to [-\delta_3, \delta_3]$ denote the functions $a \mapsto \widehat\Phi_k^{-1}(a_k + a)$. Observe that $\widehat\Phi_k^{-1}(a) = \delta_5 k^{-1}a$ for any $a \in (-k, k)$. Then, given our assumptions on $\{a_k\}$, the sequence $\{\widehat{s}_k\}$ converges in the $C^\infty$ topology to $\widehat{s} \equiv s$. Now apply Lemma~\ref{lem:geometric_convergence}. Step $3$ of its proof shows that the limiting pair is $(\eta^s, J^s)$. 
\end{proof}

\subsubsection{Adiabatic limit set} We introduce an adiabatic analogue of the stretched limit set from \S\ref{subsubsec:stretched_limit_set}. For any sequence $\{k_j\}$, fix a closed, connected Riemann surface $\widehat{C}_j$ and a $\widehat{J}_{k_j}$-holomorphic curve $\widehat{u}_j: \widehat{C}_j \to \widehat{\bW}_{k_j}$. The \emph{adiabatic limit set} $\widehat{\cX} \subseteq \cK( (-1,1) \times Y) \times (-\delta_5, \delta_5)$ is the collection of pairs $(\Xi, s)$ for which there exists a sequence $\{a_j\}$ such that 
\begin{enumerate}[(\roman*)]
\item $\delta_5 k_j^{-1}a_j \to s$
\item A subsequence of the slices
$$\tau_{a_j} \cdot \Big(\widehat{u}_{j}(\widehat{C}_j) \cap (a_j - 1, a_j + 1) \times Y\Big) \subseteq (-1,1)\times Y$$
converges in $\cK((-1,1)\times Y)$ to $\Xi$. 
\end{enumerate}

\subsection{Proof of Proposition~\ref{prop:main2}}\label{subsec:almost_cyl_adiabatic_proof} Fix a finite set of points $\bfp \subset Y$. Let $m := \#\bfp$ denote the cardinality of $\bfp$. 

\subsubsection{Closed holomorphic curves with point constraints}\label{subsubsec:closed_curves_adiabatic} Define point constraints $\widehat{\bfw}_{d,k} \subset \widehat{\bW}_k$ as follows. Let $\widehat{\bfa}_{d,k} \subset [-k, k]$ denote a finite subset of $2d^2 + 1$ equally spaced points: 
$$\widehat{\bfa}_{d,k} := \{-ik d^{-2}\,|\,i\in \bZ \cap [-d^2, d^2]\}.$$

Choose points $w_\pm \in \bW_{\pm}$ and set $\widehat{\bfw}_{d,k} := \widehat{\bfa}_{d,k} \times \bfp \cup \{w_+, w_-\}$. Using Lemma~\ref{lem:j_tame} and Proposition~\ref{prop:closed_curves} above, we obtain the following analogue of Corollary~\ref{cor:closed_curves}. 

\begin{cor}\label{cor:closed_curves_adiabatic}
Fix any integer $d \geq 1$ and any $k$ sufficiently large so that $\widehat{J}_k$ is $\widehat{\Omega}_k$-tame. Then there exists a closed, connected Riemann surface $\widehat{C}_{d,k}$ and a $\widehat{J}_k$-holomorphic curve $\widehat{u}_{d,k}: \widehat{C}_{d,k} \to \widehat{\bW}_k$ such that (i) $G_a(\widehat{C}_{d,k}) = g(4md)$, (ii) the class $(\widehat{u}_{d,k})_*[\widehat{C}_{d,k}]$ is Poincar\'e dual to $4md[\widehat{\Omega}_k]$ and (iii) $\widehat{\bfw}_{d,k} \subset \widehat{u}_{d,k}(\widehat{C}_{d,k})$.
\end{cor} 

\subsubsection{Construction of adiabatic limit sets}\label{subsubsec:adiabatic_limit_set_construction} For each fixed $d$, let $\{\widehat{u}_{d,k}\}$ denote the sequence of curves from Corollary~\ref{cor:closed_curves_adiabatic}. For any sequence $\bfk$, let $\widehat\cX_d(\bfk)$ denote the adiabatic limit set of this sequence $\{\widehat{u}_{d,k_j}\}$. The analogue of Lemma~\ref{lem:lim_set_conn} holds for adiabatic limit sets, with the same proof. 

\begin{lem}\label{lem:adiabatic_limit_set_conn}
For any fixed $d \geq 1$ and any sequence $\bfk = \{k_j\}$, there exists a subsequence $\bfk' \subseteq \bfk$ such that for any closed, connected interval $\cJ$ of positive length, the subset $\pi_{\bR}^{-1}(\cJ) \subseteq \widehat\cX_d(\bfk')$ is connected.
\end{lem}

\subsubsection{Main technical proposition} The following proposition concerns the structure of $\widehat{\cX}_d(\bfk)$ when $d$ is large. It is analogous to Proposition~\ref{prop:stretched_limit}. 

\begin{prop}\label{prop:adiabatic_limit}
Fix any integer $n \geq 1$. Then there exists a large integer $d \gg 1$, a sequence $\bfk$, and a subset $\cQ \subset (-\delta_5, \delta_5)$ of measure at least $(2 - 2^{-n-1})\delta_5$ such that the following holds for any $s \in \cQ$. Define $\cJ := [-\delta_5, \delta_5] \cap [s - d^{-2}, s + d^{-2}]$. Then $\pi_{\bR}^{-1}(\cJ) \subseteq \widehat\cX_d(\bfk)$ satisfies the following properties: 
\begin{enumerate}[(a)]
\item There exists $(\Xi, s') \in \pi_{\bR}^{-1}(\cJ)$ such that $\{0\} \times \bfp \subset \Xi$. 
\item There exists $(\Xi, s') \in \pi_{\bR}^{-1}(\cJ)$ such that $\Xi = (-1, 1) \times \Lambda$, where $\Lambda \subseteq Y$ is a finite union of closed orbits of $R^s$. 
\item For every $(\Xi, s') \in \pi_{\bR}^{-1}(\cJ)$, the set $\Xi$ is a $1/n$-almost cylinder with respect to $(\eta^s, J^s)$.
\item The set $\pi_{\bR}^{-1}(\cJ)$ is connected. 
\end{enumerate}
\end{prop}

\subsubsection{Proof of Proposition~\ref{prop:main2}} We defer the proof of Proposition~\ref{prop:adiabatic_limit} to \S\ref{subsec:adiabatic_lim_proof}. We first use Proposition~\ref{prop:adiabatic_limit} and Lemma~\ref{lem:adiabatic_limit_set_conn} to prove Proposition~\ref{prop:main2}.   

\begin{proof}
Fix $n \geq 1$ and some large $d \geq 2^{8n}\delta_5^{-1/2}$. Let $\bfk$ be the sequence and let $\cQ \subseteq (-\delta_5, \delta_5)$ be the set of measure at least $(2 - 2^{-n-1})\delta_5$ produced by Proposition~\ref{prop:adiabatic_limit}. Define $\cQ_{\bfp, n} := (-\delta_5 + d^{-2}, \delta_5 - d^{-2}) \cap \cQ$. Since $d \geq 2^{8n}$, it follows that $\cQ_{\bfp, n}$ has measure at least $(2 - 2^{-n})\delta_5$. Fix any $s \in \cQ_{\bfp, n}$ and let $\cJ \subset (-\delta_5, \delta_5)$ denote the interval $(s - d^{-2}, s + d^{-2})$. Set $\cW := \pi_{\bR}^{-1}(\cJ) \subseteq \widehat\cX_d(\bfk)$. The set $\cW$ is connected by Proposition~\ref{prop:adiabatic_limit}(d). Define $\cZ^s_{\bfp, n} \subseteq \cK( (-1,1) \times Y)$ to be equal to $\pi_{\cK}(\cW)$. The set $\cW$ is connected and $\pi_{\cK}$ is continuous, so $\cZ^s_{\bfp, n}$ is connected. Now, Proposition~\ref{prop:main2}(a--c) each follow from Proposition~\ref{prop:adiabatic_limit}(a--c). 
\end{proof}

\subsection{Proof of Proposition~\ref{prop:adiabatic_limit}}\label{subsec:adiabatic_lim_proof}

The proof of Proposition~\ref{prop:adiabatic_limit} follows a similar format to the proof of Proposition~\ref{prop:stretched_limit}. 

\subsubsection{Capped slices, accumulation, and blowup}\label{subsubsec:capped_acc_blowup} We introduce several definitions and notations. Let $\widehat\cR \subseteq (-\delta_5, \delta_5)$ denote the set of levels $s$ such that, for each $d$ and $k$, we have (i) $\delta_5^{-1}ks$ is a regular value of $\widehat{u}_{d,k}$ and (ii) the subset $(a \circ \widehat{u}_{d,k})^{-1}(\delta_5^{-1} ks)$ does not contain any nodal points. Now, given any closed interval $\cI \subseteq (-\delta, \delta)$ with endpoints in $\widehat{\cR}$, we associate to it a capped slice $\widehat{C}^{\cI}_{d,k} \subseteq \widehat{C}_{d,k}$, defined as in \S\ref{subsubsec:capped_accumulation}, but scaling $\cI$ by $\delta_5^{-1}k$ instead of $k$. Define accumulation sets $\widehat\bfs_\omega(d, \epsilon; \bfk)$ and $\widehat\bfs_\chi(d, b; \bfk)$ as follows. We say $s \in \widehat{\bfs}_\omega(d, \epsilon; \bfk)$ if and only if there exists a sequence of intervals $\mathcal{L}_j$ satisfying the following properties:
\begin{enumerate}[(\roman*)]
\item The sequence $\{\cL_j\}$ converges to $\{s\}$ in $\cK(\bR)$;
\item We have $\limsup_{j \to \infty} \int_{\widehat{C}^{\mathcal{L}_j}_{d,k_j}} \widehat{u}_{d,k_j}^*\widehat{\omega}_{k_j} > \epsilon$.
\end{enumerate}

Define $\widehat{\bfs}_\omega(d; \bfk) := \bigcup_{\epsilon > 0} \widehat{\bfs}_\omega(d, \epsilon; \bfk)$. We say $s \in \widehat{\bfs}_\chi(d, b; \bfk)$ if it admits a sequence of intervals $\mathcal{L}_j$ satisfying the following properties:
\begin{enumerate}[(\roman*)]
\item The sequence $\{\cL_j\}$ converges to $\{s\} \in \cK(\bR)$; 
\item There exists a sequence of irreducible components $Z_{d,j} \subseteq \widehat{C}^{\mathcal{L}_j}_{d,k_j}$ such that 
$$\limsup_{j \to \infty} \chi(Z_{d,j}) < -b.$$
\end{enumerate}

The following lemma bounds the sizes of the accumulation sets. 

\begin{lem}\label{lem:pigeonhole_adiabatic}
There exists a positive constant $c_5 \geq 1$ such that the following holds for any $d \geq 1$, $b \geq 1$, $\epsilon > 0$, and sequence $\bfk$. There exists a subsequence $\bfk' \subseteq \bfk$ such that
\begin{equation}\label{eq:at_bounds_adiabatic} \#\widehat{\bfs}_\omega(d, \epsilon; \bfk') \leq 4m\epsilon^{-1}d,\qquad \#\widehat\bfs_{\chi}(d, b; \bfk') \leq 4m(c_5 d + 4mb^{-1}d^2).\end{equation}
\end{lem}

\begin{proof}
The lemma follows from repeating the proofs of Lemmas~\ref{lem:pigeonhole_omega} and \ref{lem:pigeonhole_chi}, changing the notation, and making a couple of other modifications that we list here. The constant $c_5$ replaces the constant $c_4$ from Lemma~\ref{lem:pigeonhole_chi}. It arises from an analogue of Lemma~\ref{lem:quantization}, with an identical proof. We replace the uses of Corollary~\ref{cor:closed_curves} with Corollary~\ref{cor:closed_curves_adiabatic}. 
\end{proof}

Lemma~\ref{lem:pigeonhole_adiabatic} has the following consequence.

\begin{lem}\label{lem:pigeonhole_adiabatic2}
Fix any $d \geq 1$ and any sequence $\bfk$. Then, there exists a subsequence $\bfk' \subseteq \bfk$ such that $\widehat{\bfs}_\omega(d; \bfk')$ is countable. 
\end{lem}

We now introduce a new set that did not appear in \S\ref{subsec:stretched_limit_proof}, which tracks the levels sets on which the integral of $\widehat\lambda$ blows up. To ease some technical complications, our definition will use a smoothed version of the $\widehat\lambda$-integral. Define a smooth function $r: [-2, 2] \to [0,1]$ such that (i) $r = 0$ in a neighborhood of $-2$, (ii) $r = 1$ in a neighborhood of $2$, (iii) $r'(a) \in [0, 1]$ for each $a \in [-2, 2]$, and (iv) $r(a) = a$ for $a \in [-1/2, 1/2]$. For any $k$ and any $s \in (-\delta_5, \delta_5)$, define $a_{s,k} := \delta_5 k^{-1} s$ and let $r_{s,k} := r \circ \tau_{a_{s,k}}$ denote the smooth function $a \mapsto r(a - a_{s,k})$. For any $d \geq 1$, any $k$, and any $s \in (-\delta_5, \delta_5)$, define the surface
\begin{equation}\label{eq:height_4_slice}
C_{s,d,k} := (a \circ \widehat{u}_{d,k})^{-1}( a_{s,k} - 4, a_{s,k} + 4).
\end{equation}

Define a smooth function 
\begin{equation}\label{eq:lambda_integrals}
\begin{gathered}
E_{d,k,\lambda}: (-\delta_5, \delta_5) \to (0, \infty), \\
s \mapsto \int_{C_{s,d,k}} \widehat{u}_{d,k}^*(r_{s,k}'(a)(da \wedge \widehat\lambda_k)).
\end{gathered}
\end{equation}

For any $d \geq 1$, any sequence $\bfk = \{k_j\}$, and any constant $A > 0$, let $\widehat\bfs_\lambda(d, A; \bfk)$ denote the set of all $s \in (-\delta_5, \delta_5)$ with the following property. For any sequence $s_j \to s$, we have
$$\liminf_{j \to \infty} E_{d,k_j,\lambda}(s_j) > A.$$

Define $\widehat\bfs_\lambda(d; \bfk) := \bigcap_{A > 0} \widehat\bfs_\lambda(d, A; \bfk)$. This is a null set. 

\begin{lem}\label{lem:pigeonhole_blowup}
For any $d \geq 1$ and any sequence $\bfk$, the set $\widehat\bfs_\lambda(d; \bfk)$ has zero Lebesgue measure. 
\end{lem}

Lemma~\ref{lem:pigeonhole_blowup} is proved from the following bound on $\widehat\bfs_\lambda(d, A; \bfk)$. 

\begin{lem}\label{lem:pigeonhole_blowup2}
Fix any $d \geq 1$, any sequence $\bfk = \{k_j\}$, and any $A > 0$. Then the set $\widehat\bfs_\lambda(d, A; \bfk)$ is open and has Lebesgue measure at most $512 A^{-1}md$. 
\end{lem}

\begin{proof}
The proof will take $3$ steps.

\noindent\textbf{Step $1$:} This step proves that that $\widehat\bfs_\lambda(d, A; \bfk)$ is open. For each $j$, let $E_j := E_{d,k_j,\lambda}$. Let
$$\cG_j := \{(s, E_{d,k_j,\lambda}(s))\,:\,s \in (-\delta_5(1 - 2k_j^{-1}), \delta_5(1-2k_j^{-1}))\} \subset \bR^2$$ 
denote the graph of $E_j$. Let $\cG := \limsup \overline{\cG_j}$ denote the set of limit points of the sequence of subsets $\{\cG_j\}$; this is a closed subset of $\bR^2$. Then, a point $s \in (-\delta_5, \delta_5)$ lies in $\widehat\bfs_\lambda(d, A; \bfk)$ if and only if there is no $E \in [0, A]$ such that $(s, E) \in \cG$. Since $\cG$ is closed, this is an open condition, and therefore $\widehat\bfs_\lambda(d, A; \bfk)$ is open. 

\noindent\textbf{Step $2$:} This step proves the following claim. Fix any compact subset $\widehat\bfs \subseteq \widehat\bfs_\lambda(d, A; \bfk')$. Then, we claim that there exists some $j_{\bfs} \geq 1$ such that if $j \geq j_{\bfs}$, then $E_j(s) > A$ for every $s \in \bfs$. This is equivalent to showing that $\cG_j \cap \bfs \times [0, A]$ is empty for sufficiently large $j$. Observe that $\cG$ is disjoint from $\bfs \times [0, A]$. Any limit point of the sequence of sets $\{\overline{\cG_j} \cap \bfs \times [0, A]\}$ must lie in $\cG \cap \bfs \times [0, A]$. It follows that the sequence has no limit points. Thus, all but finitely many elements of are empty. 

\noindent\textbf{Step $3$:} This step completes the proof. For any open subset $U \subseteq \bR$, the Lebesgue measure $|U|$ is equal to the supremum of $|K|$ over all compact subsets $K \subset U$. Therefore, it suffices to prove the bound
\begin{equation}\label{eq:pigeonhole_blowup5}|\widehat\bfs| \leq 200 A^{-1}md \end{equation} 
for any compact subset $\widehat\bfs \subseteq \widehat\bfs_\lambda(d, A; \bfk)$. Since $\widehat\bfs$ is compact, by Step $2$ there exists some $j_{\bfs} \geq 1$ such that for any $j \geq j_{\bfs}$ and any $s$, we have $E_j(s) > A$. It follows that 
$$\int_{C_{s,d,k_j}} \widehat{u}_{d,k_j}^*(r_{s,k_j}'(a) da \wedge \widehat{\lambda}_{k_j}) > A$$
for any such $s$ and $j$. Since $r_{s,k}'(a) \in [0,1]$ for every $a$, we deduce the bound
\begin{equation}\label{eq:pigeonhole_blowup} \int_{(a \circ \widehat{u}_{d,k_j})^{-1}( (a_{s,k_j} - 2, a_{s,k_j} + 2))} \widehat{u}_{d,k_j}^*(da \wedge \widehat\lambda_{k_j}) \geq A.\end{equation}

Corollary~\ref{cor:closed_curves_adiabatic} and Lemma~\ref{lem:j_tame} imply a global bound on the integral of $da \wedge \widehat\lambda_{k_j}$ for each $j$:
\begin{equation}\label{eq:pigeonhole_blowup2} \int_{(a \circ \widehat{u}_{d,k_j})^{-1}( (-k_j, k_j))} \widehat{u}_{d,k_j}^*(da \wedge \widehat\lambda_{k_j}) \leq 8\delta_5^{-1}mk_jd. \end{equation}

Now, for each $j$, choose a maximal collection of points $\{s_{j,i}\}_{i=1}^{m_{j,i}}$ in $\widehat\bfs$ such that the intervals $(s_{j,i} - 4\delta_5 k_j^{-1}, s_{j,i} + 4\delta_5 k_j^{-1})$ are pairwise disjoint. It follows from the maximality property that the tripled intervals $(s_{j,i} - 12 \delta_5 k_j^{-1}, s_{j,i} + 12\delta_5 k_j^{-1})$ cover the set $\widehat\bfs$.  It follows that 
\begin{equation}\label{eq:pigeonhole_blowup3}|\widehat\bfs| \leq 64\delta_5 k_j^{-1}m_{j,i}.\end{equation}

By \eqref{eq:pigeonhole_blowup}, we have
\begin{equation}\label{eq:pigeonhole_blowup4} \sum_{i=1}^{m_{j,i}} \int_{(a \circ \widehat{u}_{d,k_j})^{-1}( (\delta_5^{-1}k_js_{j,i} - 2, \delta_5^{-1}k_js_{j,i} + 2))} \widehat{u}_{d,k_j}^*(da \wedge \widehat\lambda_{k_j}) \geq Am_{j,i}.\end{equation}

The sets $(a \circ \widehat{u}_{d,k})^{-1}( (k_js_{j,i}- 4, k_js_{j,i} + 4) )$ are disjoint by construction. It follows from \eqref{eq:pigeonhole_blowup2} that the left-hand side of \eqref{eq:pigeonhole_blowup4} is bounded above by $8m\delta_5^{-1}k_jd$. After re-arranging, we have $m_{j,i} \leq 8A^{-1}\delta_5^{-1}mk_jd$. Plug this into \eqref{eq:pigeonhole_blowup3} to show $|\widehat\bfs| \leq 200 A^{-1}md$, proving \eqref{eq:pigeonhole_blowup5}. 
\end{proof}

\subsubsection{Controlled accumulation implies almost cylinders} Fix any $d \geq 1$ and any $n \geq 1$. Define constants $\epsilon_{d, n} := 2^{2n}\delta_5^{-1}m d^{-1/2}$ and $b_{n} = 2^{2n+2}m^2\delta_5^{-1}$. The following result is an analogue of Proposition~\ref{prop:almost_cylinders}. 

\begin{prop}\label{prop:almost_cylinders_adiabatic}
For any integer $n \geq 1$, there exists some $d_n \geq 1$ such that the following holds for all $d \geq d_n$ and any sequence $\bfk$. Fix any $s \in (-\delta_5,\delta_5)\,\setminus\,(\widehat{\bfs}_\omega(d, \epsilon_{d,n}) \cup \widehat{\bfs}_\chi(d, b_n)$. Then, for any $(\Xi, s) \in \widehat\cX_d(\bfk)$, the set $\Xi$ is a $1/n$-almost cylinder with respect to $(\eta^s, J^s)$. 
\end{prop}

\begin{proof} 
The proposition follows from repeating the proof of Proposition~\ref{prop:almost_cylinders} and changing the notation. Corollary~\ref{cor:geometric_convergence_adiabatic} is used instead of Corollary~\ref{cor:geometric_convergence}. 
\end{proof}

\subsubsection{No blowup implies closed orbits} The following proposition has no analogue in \S\ref{subsec:stretched_limit_proof}. It asserts that, at any level lying outside the blowup set and $\omega$-accumulation set, the adiabatic limit set contains a union of closed orbits. The statement and proof are inspired by \cite[Proposition $3.9$]{FH22}. 

\begin{prop}\label{prop:closed_orbits_adiabatic}
For any $d \geq 1$, any sequence $\bfk$, and any $s \not\in \widehat{\bfs}_\lambda(d; \bfk) \cup \widehat{\bfs}_\omega(d; \bfk)$, there exists $(\Xi, s) \in \widehat\cX_d(\bfk)$ such that $\Xi = (-1, 1) \times \Lambda$, where $\Lambda$ is a finite union of closed orbits of $R^s$. 
\end{prop}

We prove Proposition~\ref{prop:closed_orbits_adiabatic} using a more technical and general result, Proposition~\ref{prop:closed_orbits_quantitative} below. This result is formulated using the language of currents and geometric convergence. For any $s$, a \emph{closed $R^s$-orbit set} is a finite collection $\cO = \{(\gamma_i, m_i)\}$ of pairs $(\gamma_i, m_i)$ where $\gamma_i$ is a closed orbit of $R^s$ and $m_i$ is a positive integer. The orbit set $\cO$ is naturally a \emph{$1$-current} on $Y$. A $1$-current on $Y$ is a continuous linear functional on the space of smooth $1$-forms on $Y$. The pairing of $\cO$ with a smooth $1$-form $\alpha$ is defined to be $\cO(\alpha) := \sum_i m_i \int_{\gamma_i} \alpha$. Consider two orbit sets $\cO$ and $\cO'$ to be equivalent if they are equivalent as $1$-currents. 

We also work with \emph{$2$-currents} on $\mathcal{I} \times Y$, where $\cI \subset \bR$ is any connected interval. These are continuous linear functionals on the space of compactly supported $2$-forms on $Y$. For example, any closed $R^s$-orbit set $\cO$ defines a $2$-dimensional current $\cI \times \cO$ on $\cI \times Y$; the pairing with a smooth $2$-form $\beta$ is defined by
$$(\cI \times \cO)(\beta) := \sum_i m_i \int_{\cI \times \gamma_i} \beta.$$

The \emph{support} $\operatorname{supp}(\cC) \subseteq \cI \times Y$ of a $2$-current $\cC$ is the complement of the largest open set $U$ such that $\cC(\beta) = 0$ for any $\beta$ compactly supported in $U$. A sequence of $2$-currents $\{\cC_j\}$ on $\cI \times Y$ \emph{geometrically converges} to a $2$-current $\cC$ if and only if i) $\cC_j(\beta) \to \cC(\beta)$ for any compactly supported $2$-form $\beta$ and ii) $\operatorname{supp}(\cC_j) \cap K \to \operatorname{supp}(\cC) \cap K$ in $\cK(K)$ for any compact subset $K \subset \cI \times Y$. 

For each $d$, $k$, and $s \in (-\delta_5, \delta_5)$, we define a $2$-current $\cC_{s,d,k}$ on $(-2, 2) \times Y$ by the formula 
$$\cC_{s,d,k}(\beta) := \int_{C_{s,d,k}} (\tau_{s,k} \circ \widehat{u}_{d,k})^*\beta.$$

Now, we state our technical proposition. 

\begin{prop}\label{prop:closed_orbits_quantitative}
Fix any $d \geq 1$, any sequence $\bfk = \{k_j\}$, and any $s \not\in \widehat{\bfs}_\omega(d; \bfk)$. Fix any sequence $s_j \to s$ such that $\liminf_{j \to \infty} E_{d,k_j,\lambda}(s_j) < \infty$. Then, there exists a closed $R^s$-orbit set $\cO$ such that the sequence of currents $\{\cC_{s_j,d,k_j}\}$ has a subsquence converging geometrically to $(-2, 2) \times \cO$. 
\end{prop}

We explain why Proposition~\ref{prop:closed_orbits_quantitative} implies Proposition~\ref{prop:closed_orbits_adiabatic}. 

\begin{proof}[Proof of Proposition~\ref{prop:closed_orbits_adiabatic}] 
Since $s \not\in \widehat\bfs_\lambda(d; \bfk)$, there exists a sequence $s_j \to s$ such that $\liminf_{j \to \infty} E_{d,k_j,\lambda}(s_j) < \infty$. By Proposition~\ref{prop:closed_orbits_quantitative}, after passing to a subsequence, there exists a closed $R^s$-orbit set $\cO$ such that $\cC_{s_j,d,k_j} \to (-2, 2) \times \cO$. For each $j$, we write $a_j := a_{s_j,k_j}$ and $\cC_j := \cC_{s_j,d,k_j}$ for brevity. For each $j$, the slices
$$\Xi_j := \tau_{a_j} \cdot \Big(\widehat{u}_{d,k}(\widehat{C}_{d,k}) \cap (a_j - 1, a_j + 1) \times Y \Big)$$
are equal to $\operatorname{supp}(\cC_j) \cap (-1, 1) \times Y$. We claim that $\Xi_j \to (-1, 1) \times \cO$, which implies the proposition. First, we prove $\limsup \Xi_j \subseteq (-1, 1) \times \cO$. Fix any point $z \in \limsup \Xi_j$. Then, any neighborhood of $z$ intersects infinitely many $\Xi_j$. This implies that any neighborhood of $z$ intersects infinitely many of the sets $\operatorname{supp}(\cC_j) \cap [-1, 1] \times Y$, so by geometric convergence it follows that $z \in [-1, 1] \times \cO$. Since $z \in (-1, 1) \times Y$, it follows that $z \in (-1, 1) \times \cO$. Second, we prove $(-1, 1) \times \cO \subseteq \liminf \Xi_j$. Choose any point $z \in (-1, 1) \times \cO$. By geometric convergence, we have $[-1, 1] \times \cO = \lim \operatorname{supp}(C_j) \cap [-1, 1] \times Y$, so any neighborhood of $z$ intersects all but finitely many of the sets $\operatorname{supp}(C_j) \cap [-1, 1] \times Y$. It follows that any neighborhood of $z$ intersects all but finitely many of the $\Xi_j$, so $z \in \liminf \Xi_j$. 
\end{proof}

The proof of Proposition~\ref{prop:closed_orbits_quantitative} relies on the compactness properties of $J$-holomorphic currents and some arguments from \cite[\S $5$]{Prasad23a}. We discuss a special class of $J$-holomorphic currents, similar to the $J$-holomorphic currents from ECH \cite[\S$3.1$]{ech_notes}, and state some key results. We refer the reader to \cite{DW21} or to \cite[\S $5$]{Prasad23a} for a more detailed and general exposition. Given an open interval $\cI$ and a pair $(\bar\eta, \bar{J}) \in \cD(\cI \times Y)$, a \emph{$\bar{J}$-holomorphic current} is a finite collection $\cC = \{(u_i, n_i)\}$ where each $C_i$ is an irreducible, non-compact Riemann surface without boundary, each $u_i: C_i \to \cI \times Y$ is a proper $\bar{J}$-holomorphic curve, and each $n_i$ is a positive integer. The collection $\cC$ defines a $2$-current by the formula
$$\cC(\beta) := \sum_i n_i \int_{C_i} u_i^*\beta.$$

For example, taking $(\bar\eta, \bar{J}) = (\eta^s, J^s)$ and $\cO = \{(\gamma_i, m_i)\}$ to be a closed $R^s$-orbit set, the $2$-current $\cI \times \cO = \{( \cI \times \gamma_i, m_i)\}$ is $J^s$-holomorphic. The currents $\cC_{s,d,k}$ defined above form another class of examples: they are $\widehat{J}_k^{a_{s,k}}$-holomorphic currents. Consider two $\bar{J}$-holomorphic currents $\cC$ and $\cC'$ to be equivalent if they are equivalent as $2$-currents. The \emph{area} of a $\bar{J}$-holomorphic current is 
$$\operatorname{Area}_{\bar{g}}(\cC) := \sum_i n_i \int_{C_i} \operatorname{dvol}_{u_i^*\bar{g}} = \sum_i n_i \int_{C_i} u_i^*(da \wedge \bar{\lambda} + \bar{\omega}) \in [0, \infty].$$

Taubes \cite{Taubes98} proved that a sequence of holomorphic currents with bounded area in a $4$-manifold has a geometrically convergent subsequence. 

\begin{prop}[{\cite[Proposition $3.3$]{Taubes98}}]\label{prop:currents_compactness}
Fix a finite constant $c > 0$. Fix an open interval $\cI$ and a convergent sequence $(\bar\eta_k, \bar{J}_k) \to (\bar\eta, \bar{J}) \in \cD(\cI \times Y)$. For each $k$, let $\cC_k$ be a $\bar{J}_k$-holomorphic current on $\cI \times Y$ such that $\operatorname{Area}_{\bar{g}_k}(\cC_k) \leq c$. Then, after passing to a subsequence, the sequence $\{\cC_k\}$ geometrically converges to a $\bar{J}$-holomorphic current $\cC$ such that $\operatorname{Area}_{\bar{g}}(\cC) \leq c$. 
\end{prop}

\begin{proof}
Repeat the proof in \cite[Proposition $3.3$]{Taubes98} to extract a subsequence that geometrically converges to a $\bar{J}$-holomorphic current $\cC = \{(u_i: C_i \to \cI \times Y, n_i)\}$. By the lower semicontinuity of mass \cite[Lemma $5.5$]{Prasad23a}, we have the area bound $\operatorname{Area}_{\bar{g}}(\cC) \leq \limsup_{k \to \infty} \operatorname{Area}_{\bar{g}_k}(\cC_k)$. 
\end{proof}

The next lemma asserts that holomorphic currents with zero action are cylinders over closed orbit sets. 

\begin{lem}\label{lem:zero_action_currents}
Fix an open interval $\cI$ and any $s \in [-\delta_3, \delta_3]$. If $\cC$ is a $J^s$-holomorphic current on $\cI \times Y$ such that (i) $\cC(\theta \cdot \omega^s) = 0$ for any compactly supported function $\theta: \cI \times Y \to [0, \infty)$ and (ii) $\operatorname{Area}_{g^s}(\cC) < \infty$, then there exists a closed $R^s$-orbit set $\cO$ such that $\cC = \cI \times \cO$.
\end{lem}

\begin{proof}
The lemma follows from the arguments in \cite[Lemmas $5.27$, $5.28$, $5.29$]{Prasad23a}. 
\end{proof}

The key to Proposition~\ref{prop:closed_orbits_quantitative} is the following estimate, which controls the area of $\cC_{s,d,k}$ by the action and by the integral $E_{d,k,\lambda}(s)$. 

\begin{lem}\label{lem:lambda_currents}
There exists a constant $c_6 \geq 1$ such that the following bound holds for any $d$, $k$, and $s \in (-\delta_5, \delta_5)$:
$$\operatorname{Area}_{\widehat{g}_k^{a_{s,k}}}(\cC_{s,d,k}) \leq c_6(E_{d,k,\lambda}(s) + \int_{C_{s,d,k}} \widehat{u}_{d,k}^*\widehat\omega_k).$$
\end{lem}

\begin{proof}
Let $\cR \subset (a_{s,k} - 2, a_{s,k} + 2)$ denote the set of regular values of $a \circ \widehat{u}_{d,k}$ in the interval $(a_{s,k} - 2, a_{s,k} + 2)$. By the co-area formula \cite[Lemma $4.13$]{FH23}, we have
\begin{equation}\label{eq:lambda_currents2}E_{d,k,\lambda}(s) = \int_{\cR} r'(t - a_{s,k})\Big(\int_{(a \circ \widetilde{u}_{d,k})^{-1}(t)} \widehat{u}_{d,k}^*\widehat{\lambda}_k\Big) dt.\end{equation}

Let $\cE \subseteq \cR$ denote the subset of $t$ such that $r'(t - a_{s,k}) > 1/2$. Then $\cE$ is non-empty and open. Write $|\cE|$ for its Lebesgue measure. Note that $|\cE|$ does not depend on $s$ or $k$. Define 
$$L := \inf_{t \in \cR}  \int_{(a \circ \widetilde{u}_{d,k})^{-1}(t)} \widehat{u}_{d,k}^*\widehat{\lambda}_k.$$

It follows from \eqref{eq:lambda_currents2} that $E_{d,k,\lambda}(s) \geq \frac{1}{2}L|\cE|$. Re-arrange this bound to show $L \leq 2|\cE|^{-1} E_{d,k,\lambda}(s)$. Using the area bound from Proposition~\ref{prop:fh_area_bound} and Remark~\ref{rem:stable_constants}, we find
\begin{equation}
\label{eq:lambda_currents3}
\begin{split}
\operatorname{Area}_{\widehat{g}_k^{a_{s,k}}}(\cC_{s,d,k}) &\leq \operatorname{Area}_{\widehat{u}_{d,k}^*\widehat{g}_k}(C_{s,d,k}) \leq c_*(L + \int_{C_{s,d,k}} \widehat{u}_{d,k}^*\widehat\omega_k) \leq c_*(2|\cE|^{-1}E_{d,k,s} + \int_{C_{s,d,k}} \widehat{u}_{d,k}^*\widehat\omega_k)
\end{split}
\end{equation}
for some $s$-, $d$-, and $k$-independent constant $c_* \geq 1$. The second inequality uses Proposition~\ref{prop:fh_area_bound} and Remark~\ref{rem:stable_constants}. The third inequality uses the bound on $L$. By \eqref{eq:lambda_currents3}, the lemma holds with $c_6 = c_* \cdot \max(2|\cE|^{-1}, 1)$. 
\end{proof}

We now prove Proposition~\ref{prop:closed_orbits_quantitative}. 

\begin{proof}[Proof of Proposition~\ref{prop:closed_orbits_quantitative}]
Fix $d$, $\bfk$, and $s_j \to s$ as in the statement of the proposition. We simplify our notation as follows. For each $j$, write $a_j := a_{s_j,k_j}$, $C_j := C_{s_j,d,k_j}$, $\cC_j := \cC_{s_j, d, k_j}$, and $E_j := E_{d,k_j,\lambda}$. Write $u_j := \widehat{u}_{d,k_j}$, $\omega_j := \widehat{\omega}_{k_j}^{a_j}$, and $g_j := \widehat{g}_{k_j}^{a_j}$ for each $j$. Recall that $\liminf_{j \to \infty} E_{d,k_j,\lambda}(s_j) < \infty$. Since $s \not\in \widehat\bfs_\omega(d; \bfk)$, we have $\lim_{j \to \infty} \int_{C_j} u_j^*\widehat\omega_{k_j} = 0$. Combine these two bounds with Lemma~\ref{lem:lambda_currents}. We deduce that $\liminf_{j \to \infty} \operatorname{Area}_{g_j}(\cC_j) < \infty$. 

By Proposition~\ref{prop:currents_compactness} and Corollary~\ref{cor:geometric_convergence_adiabatic}, after passing to a subsequence, the currents $\cC_j$ geometrically converge to a $J^s$-holomorphic current $\cC$ with finite area. Next, we prove that $\cC(\theta\omega^s) = 0$ for any compactly supported smooth function $\theta: (-2, 2) \times Y \to [0, \infty)$. Using the triangle inequality and Corollary~\ref{cor:geometric_convergence_adiabatic}, we have
\begin{equation*}
\begin{split}
|\cC(\theta\omega^s)| &= \lim_{j \to \infty} |\cC_j(\theta\omega^s)| \leq \limsup_{j \to \infty} |\cC_j(\theta\omega_j)| + \limsup_{j \to \infty} \operatorname{Area}_{g_j}(\cC_j) \cdot \|\theta(\omega_j - \omega^s)\|_{g_j} \\
&\leq \limsup_{j \to \infty} |\cC_j(\theta\omega_j)| \leq \limsup_{j \to \infty} \int_{C_j} u_j^*\widehat{\omega}_{k_j} = 0.
\end{split}
\end{equation*}

Since $\cC$ has zero action and finite area, the proposition follows from Lemma~\ref{lem:zero_action_currents}. 
\end{proof}

\subsubsection{Completing the proof} We prove Proposition~\ref{prop:adiabatic_limit} using Lemma~\ref{lem:adiabatic_limit_set_conn}, Lemma~\ref{lem:pigeonhole_adiabatic}, Lemma~\ref{lem:pigeonhole_adiabatic2}, Lemma~\ref{lem:pigeonhole_blowup}, Proposition~\ref{prop:almost_cylinders_adiabatic}, and Proposition~\ref{prop:closed_orbits_adiabatic}. 

\begin{proof}[Proof of Proposition~\ref{prop:adiabatic_limit}]
Fix any $n > 1$. By Lemma~\ref{lem:adiabatic_limit_set_conn}, Lemma~\ref{lem:pigeonhole_adiabatic}, and Lemma~\ref{lem:pigeonhole_adiabatic2}, there exists a sequence $\bfk = \{k_j\}$ such that (i) the bounds \eqref{eq:at_bounds_adiabatic} are satisfied for $\epsilon = \epsilon_{d,2n}$ and $b = b_{2n}$, (ii) $\bfs_\omega(d; \bfk)$ is countable for any $d \geq 1$, and (iii) for any $d$ and any closed, connected interval $\cJ$ of positive length, the set $\pi_{\bR}^{-1}(\cJ) \subseteq \widehat\cX_d(\bfk)$ is connected. Set $\widehat\cX_d := \widehat\cX_d(\bfk)$ for any $d \geq 1$. 

For any $d$, let $\cQ_{d,n}$ denote the set of $s \in (-\delta_5, \delta_5)$ such that the interval $(s - d^{-2}, s + d^{-2})$ does not intersect $\widehat{\bfs}_\omega(d, \epsilon_{d,2n}; \bfk)$ or $\widehat{\bfs}_\chi(d, b_{2n}; \bfk)$. We will show that, when $d$ is sufficiently large, $\cQ_{d,n}$ satisfies the assertions of the proposition. The proof will take $5$ steps. 

\noindent\textbf{Step $1$:} This step shows that $\cQ_{d,n}$ has measure at least $(2 - 2^{-n})\delta_5$ when $d$ is sufficiently large. Let $c_5 \geq 1$ denote the constant from Lemma~\ref{lem:pigeonhole_adiabatic}. It follows from \eqref{eq:at_bounds_adiabatic} that
$$\#(\widehat{\bfs}_\omega(d, \epsilon_{d,2n}) \cup \widehat{\bfs}_\chi(d, b_{2n})) \leq (2^{-4n} + O(d^{-1/2}))\delta_5 d^2.$$

It follows that, when $d$ is sufficiently large, the complement of $\cQ_{d,n}$ has measure at most $2^{-4n+1}\delta_5 \leq 2^{-n-1}\delta_5$. This implies that $\cQ_{d,n}$ has measure at least $(2 - 2^{-n-1})\delta_5$. 

\noindent\textbf{Step $2$:} Fix any $s \in \cQ_{d,n}$ and define the interval $\cJ = (s - d^{-2}, s + d^{-2})$. This step shows that $\cJ$ satisfies Proposition~\ref{prop:adiabatic_limit}(a). Recall that the curves $\widehat{u}_{d,k}$ satisfy point constraints on the levels $\widehat{\bfa}_{d,k}$ from \S\ref{subsec:geom_setup_adiabatic}. The set $\delta_5 k^{-1} \cdot \widehat{\bfa}_{d,k}$ is, for any $k$, a set of $2d^2 + 1$ equally spaced points in $(-\delta_5, \delta_5 )$. Therefore, any interval of length $2d^{-2}$ must contain at least one of these points. Proposition~\ref{prop:adiabatic_limit}(a) follows. 

\noindent\textbf{Step $3$:} This step shows that $\cJ$ satisfies Proposition~\ref{prop:adiabatic_limit}(b). By Property (ii) of the sequence $\bfk$ and Lemma~\ref{lem:pigeonhole_blowup}, the set $\widehat\bfs_\omega(d; \bfk) \cup \widehat\bfs_\lambda(d; \bfk)$ has Lebesgue measure zero. It follows that $\cJ$ intersects the complement of this set. Then Proposition~\ref{prop:adiabatic_limit}(b) follows from Proposition~\ref{prop:closed_orbits_adiabatic}. 

\noindent\textbf{Step $4$:} This step shows that $\cJ$ satisfies Proposition~\ref{prop:adiabatic_limit}(c) when $d$ is sufficiently large. Any $s' \in \cJ$ satisfies $|s' - s| \leq d^{-2}$. As $s' \to s$, we have convergence $(\eta^{s'}, J^{s'}) \to (\eta^s, J^s)$. It follows that if $d$ is sufficiently large, for any $s' \in \cJ$ and any $\Xi$ that is a $1/2n$-almost cylinder with respect to $(\eta^{s'}, J^{s'})$, we have that $\Xi$ is a $1/n$-almost cylinder with respect to $(\eta^s, J^s)$. By construction, $\cJ$ does not intersect $\widehat{\bfs}_\omega(d, \epsilon_{d,2n}; \bfk)$ or $\widehat{\bfs}_\chi(d, b_{2n}; \bfk)$. By Proposition~\ref{prop:almost_cylinders_adiabatic}, if we take $d$ sufficiently large, then for $s' \in \cJ$ and any $(\Xi, s') \in \widehat\cX_d$, we have that $\Xi$ is a $1/2n$-almost cylinder with respect to $(\eta^{s'}, J^{s'})$. Therefore, $\Xi$ is a $1/n$-almost cylinder with respect to $(\eta^s, J^s)$. This proves Proposition~\ref{prop:adiabatic_limit}(c). 

\noindent\textbf{Step $5$:} This step shows that $\cJ$ satisfies Proposition~\ref{prop:adiabatic_limit}(d). This follows from Property (iii) of $\bfk$ because $\cJ$ has positive length. 
\end{proof}

\subsection{Other $4$-manifolds}\label{subsec:other_manifolds_adiabatic} Let $\bM$ be any symplectic $4$-manifold that embeds into a closed symplectic $4$-manifold $\bW$ such that $b^+ = 1$ and the symplectic form $\Omega$ has rational cohomology class. Theorem~\ref{thm:main2} holds for any smooth function $H: \bM \to \bR$. To generalize the proof to this case, we replace Proposition~\ref{prop:closed_curves} with the more general Proposition~\ref{prop:closed_curves_general}.

%% file: periodic_orbits.tex
In this section, we prove Theorems~\ref{thm:main3} and \ref{thm:main4}. The proof of Theorem~\ref{thm:main3} relies heavily on the ideas, language, and results in \S\ref{sec:lcy}. The proof of Theorem~\ref{thm:main4} requires no additional background of a reader who is willing to accept Theorem~\ref{thm:main2} as true. 

\subsection{Proof of Theorem~\ref{thm:main3}}\label{subsec:main3_proof}

We prove Theorem~\ref{thm:main3} by proving the following equivalent result. 

\begin{thm}\label{thm:main3_local}
Let $H: \bR^4 \to \bR$ be a smooth function. Then for any $s_0 \in \cR_c(H)$, there exists some $\delta_5 > 0$ and a full measure subset $\cQ \subseteq \cR_c(H) \cap (s_0 - \delta_5, s_0 + \delta_5)$ such that, for any $s \in \cQ$, the level set $H^{-1}(s)$ contains at least two closed $X_H$-orbits. 
\end{thm}

For the remainder of the section, we fix a smooth function $H: \bR^4 \to \bR$ and $s_0 \in \cR_c(H)$. Assume without loss of generality that $s_0 = 0$ by replacing $H$ with $H - s_0$. Set $Y := H^{-1}(0)$. Assume for the sake of contradiction that Theorem~\ref{thm:main3_local} is false. We will begin with some geometric setup. 

\subsubsection{Recollections from \S\ref{sec:lcy}} We repeat the geometric setup from \S\ref{subsec:geom_setup_adiabatic}. Notation will often be used without any reminders. 

\subsubsection{Fixing a primitive} The symplectic form $\Omega$ is exact on $\bW\,\setminus\,\bD$. Fix a smooth $1$-form $\nu$ such that $d\nu$ coincides with $\Omega$ to $\bW\,\setminus\,\bD$. Define $\widehat\nu := \iota^*\nu$. Observe that $d\widehat\nu = \widehat\Omega$. 

For any $s \in [-\delta_1, \delta_1]$, define a $1$-form $\nu^s$ on $Y$ to be the pullback of $\widehat\nu$ by the map $y \mapsto (s, y)$.  Define a $1$-form $\bar{\nu}^s$ on $[-8, 8] \times Y$ to be the unique $1$-form such that $\bar{\nu}^s(\partial_a) \equiv 0$ and such that $\bar{\nu}^s$ restricts to $\nu^s$ on each level set $\{a\} \times Y$.  
For any $k$, define $\widehat\nu_k := (\widehat{f}_k)_*\nu$. We have $d\widehat\nu_k = \widehat\Omega_k$ for every $k$. A convergence lemma similar to Corollary~\ref{cor:geometric_convergence_adiabatic} holds for $\nu$. Fix any $k$ and any $a$ such that $[a - 8, a + 8] \subseteq [-\widehat{L}_k, \widehat{L}_k]$. Let $\widehat{\nu}_k^a$ be the restriction of $\tau_{a_k}^*\widehat{\nu}_k$ to $[-8, 8] \times Y$. 

\begin{lem}\label{lem:geometric_convergence_nu}
Fix any sequence $\{a_k\}$ such that $a_k \in (-k, k)$ for each $k$ and such that $\delta_5 k^{-1}a_k \to s$. Then, we have $\widehat{\nu}_k^{a_k} \to \bar\nu^s$.
\end{lem}

\begin{proof}
The proof is similar to Steps $2$ and $3$ of the proof of Lemma~\ref{lem:geometric_convergence}. Write $\bar\nu_k := \widehat{\nu}_k^{a_k}$ for each $k$. Define a smooth, form-valued function $\bar\nu^*$ on $[-8, 8]$ by defining $\bar\nu^*(a)$ to be the pullback of $\bar\nu_k$ by $y \mapsto (a, y)$. To show $\bar\nu_k \to \widehat{\nu}^s$, it suffices to show
\begin{equation}\label{eq:geometric_convergence_nu}
\bar\nu_k^* \to \nu^s,\quad \bar\nu_k(\partial_a) \to 0,
\end{equation}
in the $C^\infty$ topology. In the first assertion of \eqref{eq:geometric_convergence_nu}, we regard $\nu^s$ as a constant form-valued function on $[-8, 8]$. The convergence $\bar\nu_k^* \to \nu^s$ follows from an explicit computation. Consider the smooth functions $\widehat{s}_k: [-8, 8] \to [-\delta_3, \delta_3]$ defined by $a \mapsto \widehat{\Phi}_k^{-1}(a_k + a)$. Since $\delta_5^{-1}k a_k \to s$, the sequence $\{\widehat{s}_k\}$ $C^\infty$-converges to the constant function $\bar{s}(a) = s$. We compute $\bar\nu_k^*(a) = \nu^{s_k(a)}$ for each $a$. Since $s_k \to \bar{s}$, we have $\bar\nu_k^* = \nu^* \circ s_k \to \nu^* \circ \bar{s} = \nu^s$, proving the first item in \eqref{eq:geometric_convergence_nu}. 

Next, we verify the second assertion. The map $\widehat{f}_k^{-1}: \bW_k \to \bW$ restricts on $(-k, k) \times Y$ to the map $(a, y) \mapsto (\delta_5 k^{-1}a, y)$. Therefore, on $(-k, k) \times Y$, we have $(\widehat{f}_k^{-1})_*(\partial_a) = \delta_5 k^{-1} \partial_s$. Thus, we have
\begin{equation}\label{eq:geometric_convergence_nu3}
\begin{split}
\bar\nu_k(\partial_a) &= \widehat{\nu}((\tau_{a_k} \circ \widehat{f}_k^{-1})_*(\partial_a)) = \delta_5 k^{-1}\widehat{\nu}(\partial_s).
\end{split}
\end{equation}

The second assertion of \eqref{eq:geometric_convergence_nu} follows from \eqref{eq:geometric_convergence_nu3}. 
\end{proof}

\subsubsection{The bad subset} The assumption that Theorem~\ref{thm:main3_local} is false implies that a positive measure subset of levels contain at most one closed orbit. After some modifications, this subset satisfies several other properties. Lemma~\ref{lem:bad_subset} below contains all of this. To state the lemma, we introduce some additional notation. 

Recall the constant $\delta_5$ defined in \S\ref{subsubsec:neck_stretching_adiabatic}. Recall that every $s \in (-\delta_5, \delta_5)$ is a regular value of $H$. For any integer $i$, let $\cT_i \subseteq (-\delta_5, \delta_5)$ denote the set of $s$ such that $H^{-1}(s)$ contains exactly $i$ closed orbits. For any $s \in \cT_1$, let $\widecheck{\gamma}_s \subset H^{-1}(s)$ denote the unique closed orbit. Define functions $e_\lambda: \cT_1 \to \bR$ sending $s \mapsto \int_{\widecheck\gamma_s} \widecheck\lambda$ and $e_\nu: \cT_1 \to \bR$ sending $s \mapsto \int_{\widecheck\gamma_s} \nu$. Now, we state and prove the promised lemma. 

\begin{lem}\label{lem:bad_subset}
Then there exists a compact subset $\cB_3 \subseteq (-\delta_5, \delta_5)$ with positive Lebesgue measure such that:
\begin{enumerate}[(a)]
\item For each $s \in \cB_3$, the hypersurface $H^{-1}(s)$ contains exactly one closed orbit.
\item The function $e_\lambda$ is continuous on $\cB_3$. 
\item The function $e_\nu$ is continuous on $\cB_3$. 
\end{enumerate}
\end{lem}

\begin{proof}[Proof of Lemma~\ref{lem:bad_subset}]
The proof of the lemma will take $4$ steps. 

\noindent\textbf{Step $1$:} Let $\cB_0 := \cT_0 \cap \cT_1$ denote the set of levels $s$ for which $H^{-1}(s)$ contains at most one closed orbit. 

This step proves that $\cB_0$ is measurable. Observe that $\cB_0$ is measurable if and only if the complement $\bigcup_{i \geq 2} \cT_i$ is measurable. For any integer $\ell \geq 1$, let $\cS_\ell$ denote the set of levels $s \in [-\delta_5, \delta_5]$ such that $H^{-1}(s)$ contains two closed orbits that (i) have $X_H$-period at most $\ell$ and (ii) have Hausdorff distance at least $\ell^{-1}$ with respect to the Euclidean metric. The set $\cS_\ell$ is compact. We have 
$$\bigcup_{i \geq 2} \cT_i = \bigcup_{\ell \geq 1} \cS_\ell \cap (-\delta_5, \delta_5).$$
Therefore, $\bigcup_{i \geq 2} \cT_i$ is measurable. 

\noindent\textbf{Step $2$:} By assumption, $\cB_0$ has positive outer Lebesgue measure. This step refines $\cB_0$ to satisfy Property (a) of the lemma. By the almost-existence theorem \cite{HZ87, Rabinowitz87, Struwe90}, there exists a full measure subset $\cQ_0 \subseteq (-\delta_5, \delta_5)$ such that $s \not\in \cT_0$\footnote{This can be deduced using adiabatic neck stretching of holomorphic spheres, following \cite{FH22} or \S\ref{sec:lcy}}. The intersection $\cB_1 := \cB_0 \cap \cQ_0$ has positive Lebesgue measure and is contained in $\cT_1$. 

\noindent\textbf{Step $3$:} This step refines $\cB_1$ to satisfy Property (b) of the lemma. We claim that $e_\lambda$ is lower semicontinuous on $\cB_1$. Fix any convergent sequence $s_k \to s$. Any subsequence of the closed orbits $\widecheck\gamma_{s_k}$ either has no limit or converges to a cover of $\widecheck\gamma_s$. Therefore,
$$e_\lambda(s) = \int_{\widecheck\gamma_s} \widecheck\lambda \leq \liminf_{k \to \infty} \int_{\widecheck{\gamma}_{s_k}} \widecheck\lambda = \liminf_{k \to \infty} e(s_k).$$ 

Since $e_\lambda$ is lower semicontinuous, it is measurable. By Lusin's theorem, $e_\lambda$ is continuous on a compact subset $\cB_2 \subseteq \cB_1$ of positive Lebesgue measure. 

\noindent\textbf{Step $4$:} This step refines $\cB_2$ to satisfy Property (c) of the lemma. Define functions $e_\nu^+ := \max(0, e_\nu)$ and $e_\nu^- := \min(0, e_\nu)$ on $\cB_2$. We claim that $e_\nu^+$ is lower semicontinuous and that $e_\nu^-$ is upper semicontinuous. Choose a convergent sequence $s_k \to s$. Since $e_\lambda$ is continuous on $\cB_2$, the orbits $\widecheck\gamma_{s_k}$ have uniformly bounded $X_H$-period. After passing to a subsequence, they converge to a cover of $\widecheck\gamma_s$. Our claim follows. Then, by Lusin's theorem, there exists a compact subset $\cB_3 \subseteq \cB_2$ of positive Lebesgue measure on which both $e_\nu^+$ and $e_\nu^-$ are continuous. Therefore, $e_\nu = e_\nu^+ + e_\nu^-$ is continuous on $\cB_3$. The subset $\cB_3$ satisfies Properties (a--c). 
\end{proof}

\subsubsection{Closed curves with point constraints} Repeat the construction from \S\ref{subsubsec:closed_curves_adiabatic}. Thus, we have a family $\{\widehat{u}_{d,k}\}$ of holomorphic curves for $d \geq 1$, $k \geq 1$. We make only one modification. We replace the point constraints $\widehat{\bfw}_{d,k}$ with finite subsets $\widecheck{\bfw}_{d,k} \subset \bW_k$. We will specify $\widecheck{\bfw}_{d,k}$ when we complete the proof of Theorem~\ref{thm:main3_local}. None of the intermediate results below depend on the choice of $\widecheck{\bfw}_{d,k}$. 

\subsubsection{Computation of $e_\nu$ using holomorphic curves} 

We prove a technical lemma, Lemma~\ref{lem:nu_bound} below, computing $e_\nu$ using the curves $\{\widehat{u}_{d,k}\}$. 

We need some additional notation regarding closed orbits. Fix $s \in \cB_3$. Recall that $\widecheck{\gamma}_s$ deontes the unique closed orbit of $X_H$ in $H^{-1}(s)$. Define $\widehat{\gamma}_s \subset Y$ to be the projection to $Y$ of the closed loop $\iota^{-1}(\widecheck\gamma_s) \subset \{s\} \times Y$. Note that $\widehat{\gamma}_s$ is the unique closed orbit of the vector field $\widehat{R}_s$. For any integer $m \geq 1$, we let $m\widehat{\gamma}_s$ denote the closed $R^s$-orbit set $\{(\widehat{\gamma}_s, m)\}$. 

We also recall some notation from \S\ref{subsubsec:capped_acc_blowup}. For each $d$, $s$, $k$, write $a_{s,k} = \delta_5^{-1}k s$, $C_{s,d,k} = (a \circ \widehat{u}_{d,k})^{-1}( (a_{s,k} - 2, a_{s,k} + 2) )$, and $\cC_{s,d,k}$ for the $2$-current defined by 
$$\cC_{s,d,k}(\beta) = \int_{C_{s,d,k}} (\tau_{a_{s,k}} \circ \widehat{u}_{d,k})^*\beta.$$

We also recall the function $r: [-2,2] \to [0, 1]$ and the function
$$E_{d,k,\lambda}(s) = \cC_{s,d,k}(r'(a)da \wedge \widehat\lambda_k).$$

The following result is a consequence of Proposition~\ref{prop:closed_orbits_quantitative}. 

\begin{cor}\label{cor:countable_periodic_orbits}
Fix any $d \geq 1$, any sequence of point constraints $\{\widecheck{\bfw}_{d,k}\}$, and any sequence $\bfk$. Then, for any countable subset 
$$\bfs \subset \cB_3\,\setminus\, (\widehat\bfs_\lambda(d; \bfk) \cup \bfs_\omega(d; \bfk))$$
there exists a subsequence $\bfk' = \{k_j\}$ such that the following holds. For each $s \in \bfs$, there exists some integer $m \geq 1$ such that the sequence $\{\cC_{s,d,k_j}\}$ geometrically converges to $(-2, 2) \cdot m \cdot \widehat{\gamma}_s$. 
\end{cor}

\begin{proof}
Label the elements of $\bfs$ as $\{s_\ell\}_{\ell \geq 1}$. Applying Proposition~\ref{prop:closed_orbits_quantitative} successively for each $\ell$ gives the following result. There exists a sequence $\{\bfk^\ell\}_{\ell \geq 1}$ of nested subsequences
$$\ldots \subseteq \bfk^2 \subseteq \bfk^1 \subseteq \bfk$$
such that the following holds for each $\ell \geq 1$. Expanding $\bfk^\ell = \{k_j^\ell\}$, there exists a closed $\widehat{R}^{s_\ell}$-orbit set $\cO_\ell$ such that $\{\cC_{s_\ell, d, k_j^\ell}\}$ converges geometrically to $(-2, 2) \times \cO_\ell$ as $\ell \to \infty$. The set $\cO_\ell$ must equal $\{(\widehat{\gamma}_{s_\ell}, m_\ell)\}$ for some $m_\ell \geq 1$ because $\widehat{\gamma}_{s_\ell}$ is the only closed orbit of $\widehat{R}^{s_\ell}$. A diagonal subsequence $\bfk'$ then satisfies the conditions of the corollary. 
\end{proof}

The next lemma computes $e_\nu(s)$ using the holomorphic curves $\{\widehat{u}_{d,k}\}$.

\begin{lem}\label{lem:nu_bound}
Fix any collection of point constraints $\{\widecheck\bfw_{d,k}\}$. Fix any $s \in \cB_3$. Suppose that there exists $d \geq 1$, a sequence $s_j \to s$, and a sequence $\bfk = \{k_j\}$ such that the currents $\{\cC_{s_j,d,k_j}\}$ geometrically converge to $(-2, 2) \times m\cdot \widehat{\gamma}_s$ as $j \to \infty$. Then, there exists a sequence of real numbers $\{\bar{a}_j\}$ with the following properties.
\begin{enumerate}[(a)]
\item $\bar{a}_j$ is a regular value of $a \circ \widehat{u}_{d,k_j}$ for every $j$.
\item $\bar{a}_j \in (a_{s_j,k_j} - 2, a_{s_j,k_j} + 2)$ for every $j$. 
\item We have
$$\lim_{j \to \infty} \int_{(a \circ \widehat{u}_{d,k_j})^{-1}(\bar{a}_j)} \widehat{u}_{d,k_j}^*\widehat{\nu}_{k_j} = m \cdot e_\nu(s).$$
\end{enumerate}
\end{lem}

\begin{proof}[Proof of Lemma~\ref{lem:nu_bound}]
The proof will take $4$ steps. 

\noindent\textbf{Step $1$:} This step fixes several items of notation and then reduces the lemma to another claim. For each $j$, write $a_j := a_{s_j,k_j}$, $C_j := C_{s_j,d,k_j}$, and $\cC_j := \cC_{s_j,d,k_j}$. Write $\bar\nu_j := \widehat{\nu}_{k_j}^{a_j}$, $\bar\lambda_j := \widehat{\lambda}_{k_j}^{a_j}$, $\bar\omega_j := \widehat{\omega}_{k_j}^{a_j}$, and $\bar\Omega_j := d\bar{\nu}_j$. Finally, define $u_j := \tau_{a_j} \circ \widehat{u}_{d,k_j}$. The lemma is equivalent to the claim that there exists a sequence of real numbers $\{\widetilde{a}_j\}$ such that
\begin{enumerate}[(a)]
\item $\widetilde{a}_j$ is a regular value of $a \circ u_j$ for every $j$.
\item $\widetilde{a}_j \in (-2, 2)$ for every $j$. 
\item We have
$$\lim_{j \to \infty} \int_{(a \circ u_j)^{-1}(\widetilde{a}_j)} u_j^*\bar{\nu}_j = m \cdot e_\nu(s).$$
\end{enumerate}
To see why, observe that $\{\widetilde{a}_j\}$ satisfies (a,b,c) above if and only if the sequence $\{\bar{a}_j := \widetilde{a}_j + a_{s,k_j}\}$ satisfies Lemma~\ref{lem:nu_bound}(a,b,c). 

\noindent\textbf{Step $2$:} This step recovers $m \cdot e_\nu(s)$ as a limit of $\{\cC_j(\beta^s)\}$ for some fixed test form $\beta^s$. Define a compactly supported $2$-form $\beta^s := r'(a)da \wedge \bar\nu^s$. By the geometric convergence $\cC_j \to (-2, 2) \times m\widecheck{\gamma}_s$, we have 
\begin{equation} \label{eq:nu_bound1}
\begin{split}
\lim_{j \to \infty} \cC_j(\beta^s) &= m\int_{(-2, 2) \times \widecheck\gamma_s} r'(a)da \wedge \bar\nu^s = m\Big(\int_{-2}^2 r'(a) da\Big)\Big(\int_{\widecheck\gamma_s} \nu^s\Big) \\
&= m (r(2) - r(-2))e_\nu(s) = m\cdot e_\nu(s).
\end{split}
\end{equation}

The last identity follows because $r(2) = 1$ and $r(-2) = 0$. 

\noindent\textbf{Step $3$:} For each $j$, write $\bar\eta_j := (\bar\lambda_j, \bar\omega_j)$, $\bar{J}_j := \widehat{J}_{k_j}^{a_j}$, and $\bar{g}_j$ for the metric associated to $(\bar\eta_j, \bar{J}_j)$. Then, define $\beta_j := r'(a)da \wedge \bar\nu_j$. This step shows that
\begin{equation} \label{eq:nu_bound2} 
\lim_{j \to \infty} \cC_j(\beta^s) = \lim_{j \to \infty} \cC_j(\beta_j).
\end{equation}

We prove \eqref{eq:nu_bound2} using the geometric convergence of $\{\cC_j\}$, Corollary~\ref{cor:geometric_convergence_adiabatic}, and Lemma~\ref{lem:geometric_convergence_nu}.  We have
\begin{equation}\label{eq:nu_bound3}
\begin{split}
&\lim_{j \to \infty} \cC_j(\beta^s - \beta_j) \\
&\quad = \int_{C_j} u_j^*(\beta^s - \beta_j) \leq \limsup_{j \to \infty} \|\bar\nu^s - \bar\nu_j\|_{\bar{g}_j} \cC_j(r'(a)(da \wedge \bar\lambda_j + \bar\omega_j)).
\end{split}
\end{equation}

The inequality uses the identity $u_j^*(da \wedge \bar\lambda_j + \bar\omega_j) = \operatorname{dvol}_{u_j^*\bar{g}_j}$. By Corollary~\ref{cor:geometric_convergence_adiabatic}, the Riemannian metrics $\bar{g}_j$ converge in the $C^\infty$ topology. By Lemma~\ref{lem:geometric_convergence_nu}, we have
\begin{equation}\label{eq:nu_bound4}
\lim_{j \to \infty} \|\bar\nu^s - \bar\nu_j\|_{\bar{g}_j} = 0.
\end{equation}

We have
\begin{equation}\label{eq:nu_bound5}
\begin{split}
&\limsup_{j \to \infty} \cC_j(r'(a)(da \wedge \bar\lambda_j + \bar\omega_j)) \\
&\quad \leq \limsup_{j \to \infty} \cC_j(r'(a)(da \wedge \lambda^s + \omega^s)) \\
&\quad\quad + \limsup_{j \to \infty} (\|\lambda^s - \bar\lambda_j\|_{\bar{g}_j} + \|\omega^s - \bar\omega_j\|_{\bar{g}_j})\cC_j(r'(a)(da \wedge \bar\lambda_j + \bar\omega_j)).
\end{split}
\end{equation}

By Corollary~\ref{cor:geometric_convergence_adiabatic}, we have $\|\lambda^s - \bar\lambda_j\|_{\bar{g}_j} \to 0$ and $\|\omega^s - \bar\omega_j\|_{\bar{g}_j} \to 0$. It follows from this and the geometric convergence of $\{\cC_j\}$ that the left-hand side of \eqref{eq:nu_bound5} is finite. Combined with \eqref{eq:nu_bound4}, the right-hand side of \eqref{eq:nu_bound3} must equal $0$. This proves \eqref{eq:nu_bound2}. 

\noindent\textbf{Step $4$:} This step completes the proof of the lemma. Choose some $\widetilde{a} \in (0, 2)$ very close to $2$ such that:
\begin{enumerate}[(\roman*)]
\item For each $j$, $\widetilde{a}$ and $-\widetilde{a}$ are regular values of $a \circ u_j$. 
\item $r(\widetilde{a}) = 1$ and $r(-\widetilde{a}) = 0$. 
\item $r'(a) = 0$ for any $a$ such that $|a| \geq \widetilde{a}$. 
\end{enumerate} 

It follows from (iii) that 
\begin{equation}\label{eq:nu_bound8}\int_{(a \circ u_j)^{-1}([-\widetilde{a}, \widetilde{a}])} u_j^*\beta_j = \cC_j(\beta_j)\end{equation}
for each $j$. Using Stokes' theorem and (ii), we compute
\begin{equation}\label{eq:nu_bound10}
\begin{split}
&\int_{(a \circ u_j)^{-1}([-\widetilde{a}, \widetilde{a}])} u_j^*\beta_j \\
&\quad = \int_{(a \circ u_j)^{-1}([-\widetilde{a}, \widetilde{a}])} u_j^*(r'(a)da \wedge \bar\nu_j) = \int_{(a \circ u_j)^{-1}([-\widetilde{a}, \widetilde{a}])} d(u_j^*(r(a)\bar\nu_j)) - u_j^*(r(a)\bar\Omega_j) \\
&\quad = \int_{(a \circ u_j)^{-1}(\widetilde{a})} u_j^*\bar\nu_j - \int_{(a \circ u_j)^{-1}([-\widetilde{a}, \widetilde{a}])} u_j^*(r(a)\bar\Omega_j).
\end{split}
\end{equation}

We now control the second term on the right-hand side of \eqref{eq:nu_bound10}. Let $\chi: (-2, 2) \to [0, 1]$ be a compactly supported smooth function such that $\chi(a) = 1$ whenever $|a| \leq \widetilde{a}$. Then, we have
\begin{equation}\label{eq:nu_bound11}
\begin{split}
&\limsup_{j \to \infty} \int_{(a \circ u_j)^{-1}([-\widetilde{a}, \widetilde{a}])} u_j^*(r(a)\bar\Omega_j)
\\
&\quad \leq  \limsup_{j \to \infty} \cC_j(\chi(a)r(a)\bar\Omega_j)\\
&\quad \leq \limsup_{j \to \infty} \cC_j(\chi(a)r(a)\omega^s) + \limsup_{j \to \infty} \|\omega^s - \bar\Omega_j\|_{\bar{g}_j}\limsup_{j \to \infty} \cC_j(\chi(a)(da \wedge \bar\lambda_j + \bar\omega_j)) = 0.
\end{split}
\end{equation}

The last line uses three facts. First, it uses the geometric convergence of $\{\cC_j\}$. Second, it uses the convergence $\|\omega^s - \bar\Omega_j\|_{\bar{g}_j} \to 0$, which follows from Lemma~\ref{lem:geometric_convergence_nu}. Third, it uses the bound $\limsup_{j \to \infty} \cC_j(\chi(a)(da \wedge \bar\lambda_j + \bar\omega_j)) < \infty$, which follows from an identical argument to the one proving that the left-hand side of \eqref{eq:nu_bound5} is finite. Now, putting \eqref{eq:nu_bound1}, \eqref{eq:nu_bound2}, \eqref{eq:nu_bound8}, \eqref{eq:nu_bound10}, and \eqref{eq:nu_bound11} together, we have
$$\lim_{j \to \infty} \int_{(a \circ u_j)^{-1}(\widetilde{a})} u_j^*\bar\nu_j = m \cdot e_\nu(s).$$

Therefore, the sequence $\{\widetilde{a}_j\}$, where $\widetilde{a}_j := \widetilde{a}$ for each $j$, satisfies each of the conditions (a--c) from Step $1$. 
\end{proof}

\subsubsection{Positive lower bound on $e_\nu$}

We refine $\cB_3$ so that $e_\nu$ is positive. The proof uses Lemma~\ref{lem:nu_bound} and Stokes' theorem. 

\begin{lem}\label{lem:nu_lower_bound}
There exists a compact subset $\cB_4 \subseteq \cB_3$ of positive Lebesgue measure and a constant $c_7 \geq 1$ such that $e_\nu(s) \geq c_7^{-1}$ for every $s \in \cB_4$. 
\end{lem}

\begin{proof}
Fix arbitrary point constraints $\{\widecheck{\bfw}_{1,k}\}$. The proof will take $2$ steps. 

\noindent\textbf{Step $1$:} This step shows that, for any sequence $\bfk$ and any $s \in \cB_3\,\setminus\,(\widehat\bfs_\lambda(1; \bfk) \cup \widehat\bfs_\omega(1; \bfk)$, we have $e_\nu(s) > 0$. 

By Proposition~\ref{prop:closed_orbits_quantitative}, there exists a subsequence $\bfk' = \{k_j\}$ and a sequence $s_j \to s$ such that the currents $\{\cC_{s_j,1,k_j}\}$ geometrically converge to $(-2, 2) \times m\widehat\gamma_s$ for some integer $m \geq 1$. By Lemma~\ref{lem:nu_bound}, there exists a sequence of regular values $\{\bar{a}_j\}$ of $a \circ \widehat{u}_{1,k_j}$ such that
\begin{equation}\label{eq:nu_lower_bound} \lim_{j \to \infty} \int_{(a \circ \widehat{u}_{1,k_j})^{-1}(\bar{a}_j)} \widehat{u}_{1,k_j}^*\widehat\nu_{k_j} = m \cdot e_\nu(s).\end{equation}

For each $j$, set $\Gamma_j := (a \circ \widehat{u}_{1,k_j})^{-1}(\bar{a}_j) \subset \widehat{C}_{1,k_j}$. The submanifold $\Gamma_j$ separates $\widehat{C}_{d,k_j}$ into two halves $\Sigma_{j,\pm}$ intersecting $\bW_+$ and $\bW_-$, respectively. It follows from Stokes' theorem that
\begin{equation}\label{eq:nu_lower_bound1}\int_{\Gamma_j} \widehat{u}_{1,k_j}^*\widehat\nu_{k_j} = \int_{\Sigma_{j,-}} \widehat{u}_{1,k_j}^*\widehat\Omega_{k_j} > 0.\end{equation}

By \eqref{eq:nu_lower_bound} and \eqref{eq:nu_lower_bound1}, we have $e_\nu(s) > 0$. 

\noindent\textbf{Step $2$:} This step finishes the proof. By Lemma~\ref{lem:pigeonhole_adiabatic2} and Lemma~\ref{lem:pigeonhole_blowup}, there exists a sequence $\bfk$ such that the set $\widehat\bfs_\lambda(1; \bfk) \cup \widehat\bfs_\omega(1; \bfk)$ has Lebesgue measure zero. Choose a compact subset $\cB_4 \subseteq \cB_3\,\setminus\,(\widehat\bfs_\lambda(1; \bfk) \cup \widehat\bfs_\omega(1; \bfk))$ of positive Lebesgue measure. By Step $1$, $e_\nu$ is positive on $\cB_4$. The lemma follows because $\cB_4$ is compact and $e_\nu$ is continuous (Lemma~\ref{lem:bad_subset}(c)). 
\end{proof}

\subsubsection{Restrictions on accumulation and blowup} Recall the adiabatic limit sets $\{\widehat\cX_d(\bfk)\}$ from \S\ref{subsubsec:adiabatic_limit_set_construction}. Recall the sets $\widehat\bfs_\omega(d; \bfk)$ and $\widehat\bfs_\lambda(d; \bfk)$ from \S\ref{subsubsec:capped_acc_blowup}. Let $\cB_5 \subseteq \cB_4$ be the full measure subset of points with Lebesgue density $1$. The following lemma proves that, after passing to a subsequence, the set $\widehat\bfs_\omega(d; \bfk) \cap \cB_5$ has $O(d)$ points. This is quite strong since the set $\widehat\bfs_\omega(d; \bfk)$ could itself be infinite. 

\begin{lem}\label{lem:finite_accumulation}
There exists a constant $c_8 \geq 1$ such that the following holds for any $d \geq 1$, any point constraints $\{\widecheck{\bfw}_{d,k}\}$, and any sequence $\bfk$. There exists a subsequence $\bfk' \subseteq \bfk$ such that 
$$\#(\widehat\bfs_\omega(d; \bfk') \cap \cB_5) \leq c_8 d.$$
\end{lem}

The next lemma asserts that $\lambda$-integrals do not blow up if there is no $\omega$-accumulation. 

\begin{lem}\label{lem:blowup_implies_accumulation}
Fix any $d \geq 1$, any point constraints $\{\widecheck{\bfw}_{d,k}\}$, and any sequence $\bfk = \{k_j\}$. Then, for any $s \in \cB_5$ that does not lie in $\widehat\bfs_\omega(d; \bfk)$, we have 
$$\liminf_{j \to \infty} E_{d,k_j,\lambda}(s) < \infty.$$
\end{lem}

\subsubsection{Deriving a contradiction} We defer the proofs of Lemmas~\ref{lem:finite_accumulation} and \ref{lem:blowup_implies_accumulation} and first give the proof of Theorem~\ref{thm:main3}. Let $c_8 \geq 1$ be the constant from Lemma~\ref{lem:finite_accumulation}. Fix $d > c_8$. Fix a set $\bfw = \{w_{d,i}\}_{i=1}^{d^2} \subset \bW$ as follows. Choose a subset $\{s_{d,i}\}_{i=1}^{d^2} \subseteq \cB_5$. Then, for each $i$, let $w_{d,i}$ be any point in $H^{-1}(s_{d,i})$ that does not lie on $\widecheck\gamma_{s_{d,i}}$. For each $k$, define point constraints $\widecheck{\bfw}_{d,k} := \widehat{\Phi}_k(\bfw)$. 

We use Lemmas~\ref{lem:finite_accumulation} and \ref{lem:blowup_implies_accumulation} to prove the following claim. We claim that there exists $d \geq 1$ and $i \in \{1, \ldots, d^2\}$ such that the hypersurface $H^{-1}(s_{d,i})$ contains a closed orbit passing through the point $w_{d,i}$. By construction, there are no closed orbits passing through $w_{d,i}$. Therefore, once the claim is established, we arrive at a contradiction. 

\begin{proof}[Proof of claim]
By Lemma~\ref{lem:finite_accumulation} and the inclusion $\{s_{d,i}\} \subset \cB_5$, there exists a sequence $\bfk = \{k_j\}$ such that $\widehat\bfs_\omega(d; \bfk) \cap \{s_d^i\}_{i=1}^{d^2}$ contains $c_8 d< d^2$ points. Thus, there exists $i \in \{1, \ldots, d^2\}$ such that $s_{d,i} \not\in \widehat\bfs_\omega(d; \bfk)$. Let $y \in Y$ be such that $\iota(s_{d,i}, y) = w_{d,i}$. Let $R := R^{s_{d,i}}$ denote the Hamiltonian vector field of $\eta^{s_{d,i}}$. Recall that $R$ is conjugated by the map $\iota(s_{d,i}, -)$ to a reparameterization of $X_H$ on the hypersurface $H^{-1}(s_{d,i})$. By Lemma~\ref{lem:blowup_implies_accumulation}, we have $\liminf_{j \to \infty} E_{d,k_j,\lambda}(s_{d,i}) < \infty.$. By Proposition~\ref{prop:closed_orbits_quantitative}, setting $a_j := \delta_5^{-1}k_j s_{d,i}$, a subsequence of the slices 
$$\tau_{a_j} \cdot \Big(\widehat{u}_{d,k_j}(\widehat{C}_{d,k_j}) \cap [a_j - 1/2, a_j + 1/2] \times Y\Big)$$
converges in the Hausdorff topology to $[-1/2, 1/2] \times \gamma$, where $\gamma \subset Y$ is a closed orbit of $R$. Each slice contains $(0, y)$, so $\gamma$ contains $y$. Thus, $\iota(s_{d,i}, \gamma)$ is a closed orbit of $X_H$ passing through $w_{d,i}$. 
\end{proof}

\subsubsection{Proof of Lemma~\ref{lem:finite_accumulation}} The argument proving Lemma~\ref{lem:finite_accumulation} is technical and has no analogue in prior works, so we provide a sketch. Suppose there is positive accumulation of action at a level $s \in \cB_5$. By Stokes' theorem, for large $k$, the $\nu$-integrals on the level sets of $\widehat{u}_{d,k}$ will make a positive jump at $s$. As $k \to \infty$, the $\nu$-integrals cluster near a discrete set with gap at least $c_7/2$, where $c_7$ is the constant from Lemma~\ref{lem:nu_bound}. We use this to prove that the positive jump has size at least $c_7/2$, so at least $c_7/2$ action must accumulate at $s$. By Lemma~\ref{lem:pigeonhole_adiabatic}, it follows that there are only $O(d)$ points where action accumulates. We now provide a detailed proof. 

\begin{proof}
Let $c_7$ be the constant from Lemma~\ref{lem:nu_lower_bound}. Define $c_8 := 8c_7$. Assume for the sake of contradiction that for any subsequence $\bfk' \subseteq \bfk$, we have $\#(\widehat\bfs_\omega(d; \bfk') \cap \cB_5) > c_8 d$. The proof will take $6$ steps. 

\noindent\textbf{Step $1$:} This step finds a well-behaved subsequence of $\bfk$. By Lemma~\ref{lem:pigeonhole_adiabatic2} and Lemma~\ref{lem:pigeonhole_blowup}, there exists a subsequence $\bfk^1$ such that $\widehat\bfs_\lambda(d; \bfk^1) \cup \widehat\bfs_\omega(d; \bfk^1)$ has Lebesgue measure $0$. Choose a countable dense subset $\bfs \subseteq \cB_5\,\setminus\,(\widehat{\bfs}_\lambda(d, \bfk^1) \cup \widehat\bfs_\omega(d; \bfk^1))$ as follows. Let $\{U_i\}$ denote the collection of open intervals in $\bR$ that have rational endpoints and intersect $\cB_5$. Since each point in $\cB_5$ has Lebesgue density $1$, $U_i$ intersects $\cB_5$ if and only if $U_i$ intersects $\cB_5\,\setminus\,(\widehat{\bfs}_\lambda(d, \bfk^1) \cup \widehat\bfs_\omega(d; \bfk^1))$. Choose a point 
$$s_i \in U_i \cap (\cB_5\,\setminus\,(\widehat{\bfs}_\lambda(d, \bfk^1) \cup \widehat\bfs_\omega(d; \bfk^1)))$$
for each $i$ and let $\bfs$ be the collection of all $s_i$. Note that the sets $U_i \cap \cB_5$ form a countable basis of relatively open sets for $\cB_5$, so $\bfs$ is dense in $\cB_5$.  

By Corollary~\ref{cor:countable_periodic_orbits}, there exists a further subsequence $\bfk^2 \subseteq \bfk^1$ such that for each $s \in \bfs$, the currents $\{\cC_{s,d,k}\}_{k \in \bfk^2}$ geometrically converge to $(-2, 2) \times m_s\widehat{\gamma}_s$ for some $m_s \geq 1$. Next, set $N := \lfloor c_8 d \rfloor + 1$. After passing to at most $N$ successive subsequences of $\bfk^2$, there exists a subsequence $\bfk' = \{k_j\}$ of $\bfk^2$ and a set of $N$ distinct levels $s_1, \ldots, s_N$ in $\cB_5$ with the following property. There exists a positive constant $\epsilon \in (0, 1)$ and, for each $i$, a sequence of intervals $\{\cL_{j,i} = [s_{j,i,-}, s_{j,i,+}]\}$ converging to $s_i$, such that
\begin{equation}\label{eq:finite_accumulation1}\limsup_{j \to \infty} \int_{\widehat{C}_{d,k_j}^{\cL_{j,i}}} \widehat{u}_{d,k_j}^*\widehat\omega_{k_j} > 0.\end{equation}

We simplify the notation. For each $j$, write $u_j := \widehat{u}_{d,k_j}$, $\widehat{C}_{j,i} : = \widehat{C}_{d,k_j}^{\cL_{j,i}}$, $\nu_j := \widehat{\nu}_{k_j}$, and $\Omega_j := \widehat{\Omega}_{k_j}$. By Lemma~\ref{lem:j_tame} and \eqref{eq:finite_accumulation1}, there exists $\epsilon \in (0,1)$ such that, for each $i$, we have
\begin{equation}\label{eq:finite_accumulation2}\limsup_{j \to \infty} \int_{\widehat{C}_{j,i}} u_j^*\Omega_j \geq \epsilon.\end{equation}

\noindent\textbf{Step $2$:} This step proves that, for each $i$, there exists a pair $s_{i,\pm} \in \bfs$ such that:
\begin{enumerate}[(\roman*)]
\item $s_{i,-} < s_i < s_{i,+}$; 
\item $|e_\nu(s_{i,-}) - e_\nu(s_{i,+})| < \epsilon(4c_7 d)^{-2}$.
\end{enumerate}

We also require 
\begin{enumerate}[(\roman*)]
\setcounter{enumi}{2}
\item the intervals $[s_{i,-}, s_{i,+}]$ are pairwise disjoint. 
\end{enumerate}

By Lemma~\ref{lem:bad_subset}, $e_\nu$ is continuous on $\cB_5$. Therefore, there exists $\delta > 0$ such that the intervals $(s_i - \delta, s_i + \delta)$ are pairwise disjoint and, for any $i$ and any pair $s_\pm \in (s_i - \delta, s_i + \delta) \cap \cB_5$, we have $|e_\nu(s_+) - e_\nu(s_-)| < \epsilon(4c_7 d)^{-2}$. Since $s$ has Lebesgue density $1$ in $\cB_5$ and $\cB_5$ has positive Lebesgue measure, there exist open intervals $U_{+} \subseteq (s_i, s_i + \delta)$ and $U_{-} \subseteq (s_i - \delta, s_i)$ with rational endpoints such that $U_\pm$ both intersect $\cB_5$. Choose $s_{i,+} \in \bfs \cap U_{+}$ and $s_{i,-} \in \bfs \cap U_{-}$. 

\noindent\textbf{Step $3$:} This step recalls how to compute $e_\nu(s_{i,\pm})$ using Lemma~\ref{lem:nu_bound}. Apply Lemma~\ref{lem:nu_bound} to $s_{i,+}$ and $s_{i,-}$. There exists two sequences $\{\bar{a}_{j,i,+}\}$ and $\{\bar{a}_{j,i,-}\}$ such that:
\begin{enumerate}[(a)]
\item $\bar{a}_{j,i,+}$ and $\bar{a}_{j,i,-}$ are regular values of $a \circ u_j$ for every $j$.
\item $\bar{a}_{j,i,+} \in (a_{s_{i,+}, k_j} - 2, a_{s_{i,+}, k_j} + 2)$ and $\bar{a}_{j,i,-} \in (a_{s_{i,-}, k_j} - 2, a_{s_{i,-}, k_j} + 2)$ for every $j$. 
\item There exists positive integers $m_{i,+}$ and $m_{i,-}$ such that
$$\lim_{j \to \infty} \int_{(a \circ u_j)^{-1}(\bar{a}_{j,i,+})} u_j^*\nu_j = m_{i,+}e_\nu(s_{i,+}),\quad \lim_{j \to \infty} \int_{(a \circ u_j)^{-1}(\bar{a}_{j,i,-})} u_j^*\nu = m_{i,-}e_\nu(s_{i,-}).$$
\end{enumerate}

We simplify our notation as follows. Recall that $s_{j,i,\pm}$ denote the endpoints of $\cL_{j,i}$. For each $i$ and each $j$, define $a_{j,i,-} := \delta_5^{-1}k_js_{j,i,-}$ and $a_{j,i,+} := \delta_5^{-1}k_js_{j,i,+}$. We have $(a \circ u_j)(\widehat{C}_{j,i}) \subseteq [a_{j,i,-} - 2, a_{j,i,+} + 2]$ for each $j$ and $i$. Since $s_{j,i,+} \to s_i$ and $s_{j,i, -} \to s_i$, it follows that $[a_{j,i,-} - 2, a_{j,i,+} + 2]$ lies inside $[a_{s_{i,-}, k_j} + 2, a_{s_{i,+}, k_j} - 2]$ for sufficiently large $j$. It follows from Property (b) above that $[a_{j,i,-} - 2, a_{j,i, +} + 2]$ lies inside $[\bar{a}_{j,i, -}, \bar{a}_{j,i,+}]$ for sufficiently large $j$. Apply this to obtain the following symplectic area bound:.
\begin{equation} \label{eq:finite_accumulation3}
\begin{split}
\limsup_{j \to \infty} \int_{\widehat{C}_{j,i}} u_j^*\Omega_j &\leq \limsup_{j \to \infty} \int_{(a \circ u_j)^{-1}([\bar{a}_{j,i, -}, \bar{a}_{j,i,+}])} u_j^*\Omega_j \\
&= \limsup_{j \to \infty} \Big( \int_{(a \circ u_j)^{-1}(\bar{a}_{j,i,+})} u_j^*\nu_j - \int_{(a \circ u_j)^{-1}(\bar{a}_{j,i,-})} u_j^*\nu_j\Big) \\
&= m_{i,+}e_\nu(s_{i,+}) - m_{i,-}e_\nu(s_{i,-}) \\
\end{split}
\end{equation}

The second line uses Stokes' theorem. The third line uses Property (c) above. 

\noindent\textbf{Step $4$:} This step proves that $m_{i,+} \geq m_{i,-} + 1$ for each $i$. First, we prove
\begin{equation}\label{eq:finite_accumulation4} m_{i,+} \leq 4c_7 d,\qquad m_{i,-} \leq 4c_7 d. \end{equation}

Observe that
\begin{equation}\label{eq:finite_accumulation5}
\begin{split}
m_{i,+}e_\nu(s_{i,+}) &= \lim_{j \to \infty} \int_{(a \circ u_j)^{-1}(\bar{a}_{j,+})} u_j^*\nu_j \leq \limsup_{j \to \infty} \int_{\widehat{C}_{d,k_j}} u_j^*\Omega_j \leq 4d.
\end{split}
\end{equation}

The first inequality follows from Stokes' theorem. The second follows from Corollary~\ref{cor:closed_curves_adiabatic}. Plug the bound $e_\nu(s_{i,+}) \geq c_7^{-1}$ into \eqref{eq:finite_accumulation5}. Re-arrange to obtain the first bound in \eqref{eq:finite_accumulation4}. The second bound in \eqref{eq:finite_accumulation4} follows from an identical argument. By \eqref{eq:finite_accumulation2} and \eqref{eq:finite_accumulation3}, we have
$$(m_{i,+} - m_{i,-})e_\nu(s_{i,+}) + m_{i,-}(e_\nu(s_{i,+}) - e_\nu(s_{i,-})) = m_{i,+}e_\nu(s_{i,+}) - m_{i,-}e_\nu(s_{i,-}) \geq \epsilon.$$

Apply \eqref{eq:finite_accumulation4} to bound the left-hand side from above. We have
$$(m_{i,+} - m_{i,-})e_\nu(s_{i,+}) + 4c_7 d |e_\nu(s_{i, +}) - e_\nu(s_{i,-})| \geq \epsilon.$$

Since $|e_\nu(s_{i,+}) - e_\nu(s_{i,-})| < (4c_7 d)^{-1}\epsilon$ and $e_\nu(s_{i,+}) > 0$ (Lemma~\ref{lem:nu_lower_bound}), the bound $m_{i,+} - m_{i,-} > 0$ follows. Since both of $m_{i,\pm}$ are integers, we have $m_{i,+} \geq m_{i,-} + 1$. 

\noindent\textbf{Step $5$:} This step shows that, for any $i$, we have the bound
\begin{equation}\label{eq:finite_accumulation6}\liminf_{j \to \infty} \int_{(a \circ u_j)^{-1}([\bar{a}_{j,i, -}, \bar{a}_{j,i,+}])} u_j^*\Omega_j \geq c_7^{-1}/2. \end{equation}

We have
\begin{equation}\label{eq:finite_accumulation7}
\begin{split}
m_{i,+}e_\nu(s_{i,+}) - m_{i,-}e_\nu(s_{i,-}) &\geq e_\nu(s_{i,-}) - m_{i,+}|e_\nu(s_{i,+}) - e_\nu(s_{i,-})| \\
&\geq c_7^{-1} - m_{i,+}|e_\nu(s_{i,+}) - e_\nu(s_{i,-})| \geq c_7^{-1}/2. 
\end{split}
\end{equation}

The first inequality uses Lemma~\ref{lem:nu_lower_bound} and the bound $m_{i,+} \geq m_{i,-} + 1$ from Step $4$. The second inequality uses Lemma~\ref{lem:nu_lower_bound}. The third inequality uses \eqref{eq:finite_accumulation4} and the bound $|e_\nu(s_{i,+}) - e_\nu(s_{i,-})| < (4c_7 d)^{-2}$. Now \eqref{eq:finite_accumulation6} follows from \eqref{eq:finite_accumulation3} and \eqref{eq:finite_accumulation7}. 

\noindent\textbf{Step $6$:} This step completes the proof. By Property (iii) in Step $2$, the intervals $[\bar{a}_{j,i,-}, \bar{a}_{j,i,+}]$ are pairwise disjoint for sufficiently large fixed $j$. Sum \eqref{eq:finite_accumulation6} over all $i$ and apply Corollary~\ref{cor:closed_curves_adiabatic}. We have
\begin{equation}\label{eq:finite_accumulation8}
\begin{split}
Nc_7^{-1}/2 &\leq\sum_{i=1}^N  \liminf_{j \to \infty} \int_{(a \circ u_j)^{-1}([\bar{a}_{j,i, -}, \bar{a}_{j,i,+}])} u_j^*\Omega_j \leq \liminf_{j \to \infty} \sum_{i=1}^N \int_{(a \circ u_j)^{-1}([\bar{a}_{j,i, -}, \bar{a}_{j,i,+}])} u_j^*\Omega_j \\
&\leq \liminf_{j \to \infty} \int_{\widehat{C}_{d,k_j}} u_j^*\Omega_j \leq 4d. 
\end{split}
\end{equation}

The second inequality follows from Fatou's lemma. Recall from Step $1$ that $N > 8c_7 d$. Therefore, \eqref{eq:finite_accumulation8} implies $4d < 4d$. This is the desired contradiction. 
\end{proof}

\subsubsection{Proof of Lemma~\ref{lem:blowup_implies_accumulation}} Here is a sketch of the proof. For any $s \in \cB_5$ not in $\widehat\bfs_\omega(d; \bfk)$, we consider limit points of sequences $\{E_{d,k_j,\lambda}(s_j)\}$ where $s_j \to s$. After passing to a subsequence, at least one limit point exists. The set of limit points is discrete, since they are $\lambda$-integrals over iterates of $\widecheck\gamma_s$. After passing to a subsequence, the set of limit points is connected. It follows from these observations that no sequence $\{E_{d,k_j,\lambda}(s_j)\}$ can diverge. We now give a detailed proof. 

\begin{proof}
Assume for the sake of contradiction that for any subsequence $\bfk' \subseteq \bfk$, there exists some $s \in \cB_5$, such that $s \not\in \widehat\bfs_\omega(d; \bfk')$ and such that $\{E_{\lambda, d, k}(s)\}_{k \in \bfk'}$ diverges. The proof will take $3$ steps. 

\noindent\textbf{Step $1$:} This step is very similar to Step $1$ of the proof of Lemma~\ref{lem:finite_accumulation}. We find a well-behaved subsequence of $\bfk$. By Lemma~\ref{lem:pigeonhole_adiabatic2} and Lemma~\ref{lem:pigeonhole_blowup}, there exists a subsequence $\bfk^1$ such that $ \widehat\bfs_\lambda(d; \bfk^1) \cup \widehat\bfs_\omega(d; \bfk^1)$ has Lebesgue measure $0$. Choose a countable dense subset $\bfs \subseteq \cB_5\,\setminus\,(\widehat{\bfs}_\lambda(d, \bfk^1) \cup \widehat\bfs_\omega(d; \bfk^1))$. 

By Corollary~\ref{cor:countable_periodic_orbits}, there exists a subsequence $\bfk' = \{k_j\}$ of $\bfk^1$ such that for each $s \in \bfs$, the currents $\{\cC_{s,d,k_j}\}$ geometrically converge to $(-2, 2) \times m \cdot \widecheck{\gamma}_s$ for some $m \geq 1$. 

\noindent\textbf{Step $2$:} By assumption, there exists some $s \in \widehat\bfs_\lambda(d; \bfk') \cap \cB_5$ that does not lie in $\widehat\bfs_\omega(d; \bfk')$. For each $j$, write $E_j := E_{d,k_j,\lambda}$. This step shows that there exists a sequence $s_j \to s$ such that
$$\liminf_{j \to \infty} E_j(s_j) < \infty.$$

For each $s' \in \bfs$, the currents $\cC_{s',d,k_j}$ converge geometrically as $j \to \infty$ to a current $(-2, 2) \times m\widehat\gamma_{s'}$ for some integer $m \geq 1$. It follows that $\lim_{j \to \infty} E_j(s') = me_\lambda(s')$. 

The same argument proving \eqref{eq:finite_accumulation4} shows that $m \leq 4c_7d$. By Lemma~\ref{lem:bad_subset}(b), there exists a constant $c_9 \geq 1$ such that $e_\lambda \leq c_9$. It follows that
$$\lim_{j \to \infty} E_j(s') \leq 4c_7c_9 d$$
for any $s' \in \bfs$. Since $\bfs$ is dense in $(-\delta_5, \delta_5)$, there exists a sequence $s_j \to s$ such that $\liminf_{j \to \infty} E_j(s_j) \leq 4c_7c_9 d$. 

\noindent\textbf{Step $2$:} This step proves that, for any $s \in \cB_5$ and any sequence $s_j \to s$, any limit point of the sequence $\{E_j(s_j)\}$ lies in the set $\bZ \cdot e_\lambda(s) \subset \bR$. Let $E$ be any such limit point. Fix a subsequence $\{j_\ell\}$ such that $E_{j_\ell}(s_{j_\ell}) \to E$. To simplify notation, write $s_\ell' := s_{j_\ell}$, $k_\ell' := k_{j_\ell}$, $a_\ell' := a_{s_\ell', k_\ell'}$, $C_\ell' := C_{s_\ell', d, k_\ell'}$, and $\cC_\ell' := \cC_{s_\ell', d, k_\ell'}$. Write $u_\ell' := \widehat{u}_{d,k_\ell'}$, $\lambda_\ell' := \widehat\lambda_{k_\ell'}^{a_\ell'}$, and $\omega_\ell' := \widehat\omega_{k_\ell'}$. Since $s \not\in \widehat\bfs_\omega(d; \bfk')$, Proposition~\ref{prop:closed_orbits_quantitative} shows that after passing to a subsequence, $\cC_\ell' \to (-2, 2) \times m \cdot \widecheck{\gamma}_s$ for some integer $m \geq 1$. The claim then follows from a computation:
\begin{equation*}
\begin{split}
E &= \lim_{\ell \to \infty} E_{j_\ell}(s_\ell') = \lim_{\ell \to \infty} \cC_\ell'(r'(a)da \wedge \lambda_\ell') = \lim_{\ell \to \infty} \cC_\ell'(r'(a)da \wedge \lambda^s) \\
&= m\Big(\int_{-2}^2 r'(a)\Big)e_\lambda(s) = m \cdot e_\lambda(s).
\end{split}
\end{equation*}

The third equality requires two facts to prove. First, since $s \not\in \widehat{\bfs}_\omega(d; \bfk')$, Lemma~\ref{lem:lambda_currents} implies that $\operatorname{Area}(\cC_\ell')$ is uniformly bounded. Second, we have $\lambda_\ell' \to \lambda^s$ by Corollary~\ref{cor:geometric_convergence_adiabatic}. The fourth equality uses the convergence $\cC_\ell' \to (-2,2) \times m\widecheck\gamma_s$. 

\noindent\textbf{Step $3$:} This step completes the proof. By Step $1$, there exists some $E \in \bR$ and a sequence $s_j \to s$ such that $\liminf_{j \to \infty} E_j(s_j) \leq E$. After passing to a subsequence in $j$, we may assume that $\limsup_{j \to \infty} E_j(s_j) < E + 1$. Choose some $E' > E + 1$ such that $E' \not\in \bZ \cdot e_\lambda(s)$. By assumption, $\lim_{j \to \infty} E_j(s) = \infty$. We have $E_j(s) > E' > E_j(s_j)$ for each sufficiently large $j$. By the intermediate value theorem, for each sufficiently large $j$, there exists $s_j'$ such that $|s_j' - s| \leq |s_j - s|$ and $E_j(s_j') = E'$. Therefore, $s_j' \to s$ and $\lim_{j \to \infty} E_j(s_j') = E'$. This contradicts Step $2$. 
\end{proof}

\subsection{Proof of Theorem~\ref{thm:main4}}\label{subsec:main4} To prove Theorem~\ref{thm:main4}, we recall some known dynamical properties of $C^\infty$-generic Hamiltonians $H: \bR^4 \to \bR$. Let $\gamma \subset H^{-1}(s)$ be a closed orbit of $X_H$. Choosing a point $x \in \gamma$. The linearized flow of $X_H$ defines a linear operator $P_\gamma: T_{x} \bR^4 \to T_{x} \bR^4$. The orbit $\gamma$ is \emph{hyperbolic} if $P_\gamma$ has a real eigenvalue with norm not equal to $1$. It is \emph{elliptic} if $P_\gamma$ has an eigenvalue of norm $1$ that is not a root of unity. These properties do not depend on the choice of $x \in \gamma$. An elliptic orbit $\gamma \subset H^{-1}(s)$ is \emph{Moser stable} if it has nondegenerate Birkhoff normal form; see Appendix~\ref{sec:generic} for details. A deep theorem of Moser \cite{Moser62} implies that a Moser stable elliptic orbit is accumulated by an infinite sequence of closed orbits in $H^{-1}(s)$. Now, we prove Theorem~\ref{thm:main4} using Theorem~\ref{thm:main2}.

\begin{proof}[Proof of Theorem~\ref{thm:main4}]
Let $\cQ_* \subseteq \cR_c(H)$ denote set of all levels $s$ such that any closed orbit $\gamma \subset H^{-1}(s)$ is either (i) hyperbolic or (ii) elliptic and Moser stable. By the assumptions of the theorem, $\cQ_*$ has full Lebesgue measure in $\cR_c(H)$. Let $\cH \subseteq \cR_c(H)$ denote the set of all $s$ such that every closed orbit in $H^{-1}(s)$ is hyperbolic. Let $\cE \subseteq \cR_c(H)$ denote the set of $s$ such that $H^{-1}(s)$ contains a Moser stable elliptic closed orbit. Let $\cA \subseteq \cR_c(H)$ denote the set of $s$ such that $H^{-1}(s)$ contains infinitely many closed orbits. Our goal is to show that $\cA$ has full Lebesgue measure in $\cR_c(H)$. A Moser stable orbit is accumulated by an infinite sequence of closed orbits, so 
\begin{equation}\label{eq:main4_proof1} \cE \subseteq \cA \end{equation}

By Theorem~\ref{thm:main2}, there exists a full measure subset $\cQ_1 \subseteq \cR_c(H)$ such that for any $s \in \cQ_1$, the set $\overline{\cP(s)}$ is either equal to $H^{-1}(s)$ or not locally maximal. In particular, if $s \in \cQ_1$ and $H^{-1}(s)$ contains finitely many closed orbits, then at least one has to be non-hyperbolic. It follows that
\begin{equation}\label{eq:main4_proof2} \cQ_1 \cap \cH \subseteq \cA. \end{equation}

Recall that $\cQ_* \subseteq \cE \cup \cH$. By \eqref{eq:main4_proof1} and \eqref{eq:main4_proof2}, we have
\begin{equation}\label{eq:main4_proof3} \cQ_1 \cap \cQ_* \subseteq (\cQ_1 \cap \cE) \cup (\cQ_1 \cap \cH) \subseteq \cA. \end{equation}

Thus, the set $\cQ := \cQ_1 \cap \cQ_*$ is a subset of $\cR_c(H)$ with full Lebesgue measure, in which every corresponding level set has infinitely many closed orbits. 
\end{proof}

\subsection{Other $4$-manifolds}\label{subsec:other_manifolds_periodic}

Consider a smooth function $H: \bM \to \bR$ where $\bM$ is a symplectic $4$-manifold. Assume that $\bM$ symplectically embeds into a closed symplectic $4$-manifold $\bW$ with $b^+ = 1$ and rational symplectic form $\Omega$. 

\subsubsection{Almost-existence of two closed orbits} Theorem~\ref{thm:main3} extends to this situation provided that (i) the symplectic form on $\bM$ is exact and (ii) for any $s \in \cR_c(H)$, at least one component of $\mathbb{M}\,\setminus\,H^{-1}(s)$ has compact closure. It follows, for example, that Theorem~\ref{thm:main3} extends to the cases $\bM \in \{T^*\mathbb{S}^2, T^*\mathbb{T}^2\}$. As explained in \S\ref{subsec:other_manifolds_adiabatic}, the results of \S\ref{sec:lcy} go through for $H$. However, the proofs of Lemmas~\ref{lem:finite_accumulation} and \ref{lem:blowup_implies_accumulation} rely on the facts that (i) $\bR^4$ is an exact symplectic manifold and (ii) each compact regular level set of $H$ bounds a compact domain. 

\subsubsection{Generic almost-existence of infinitely many closed orbits} Theorem~\ref{thm:main4} generalizes to this situation without any exactness assumptions. This is because the theorem follows directly from Theorem~\ref{thm:main2}. Lemma~\ref{lem:generic} holds for any symplectic $4$-manifold, so Corollary~\ref{cor:main4} generalizes as well.

%% file: closed_curves.tex
This appendix contains a proof of Proposition~\ref{prop:closed_curves}. It also discusses a version that holds for many other closed symplectic $4$-manifolds with $b^+ = 1$, such as $\mathbb{S}^2 \times \mathbb{S}^2$. The results follow from combining some known results about Taubes' Gromov invariant. 

\subsection{Taubes' Gromov invariant} Let $\bW$ be a closed symplectic $4$-manifold with symplectic form $\Omega$.  Let $J$ be any $\Omega$-tame almost-complex structure. A \emph{$J$-holomorphic cycle} in $\bW$ is a finite set of pairs $\cC = \{(u_i, n_i)\}$. Each $u_i$ is a somewhere injective $J$-holomorphic map $C_i \to \bW$ from a closed, irreducible Riemann surface $C_i$ without any nodal points. Each $n_i$ is a positive integer. The \emph{support} of $\cC$ is the set 
$$\operatorname{supp}(\cC) := \bigcup_i u_i(C_i) \subset \bW.$$ 

The \emph{homology class} of $\cC$ is the class $[\cC] := \sum_i m_i (u_i)_*[C_i] \in H^2(\bW; \bZ)$. Taubes' Gromov invariant, constructed in \cite{Taubes96}, is an integer-valued function 
$$\operatorname{Gr}(\bW, -): H_2(\bW; \bZ) \otimes \mathbb{A}(\bW) \to \bZ$$
defined by counting $J$-holomorphic cycles satisfying incidence conditions. Fix $A \in H^2(X; \bZ)$. Define the index 
$$I(A) := \langle c_1(\bW), A \rangle + A \cdot A \in \bZ.$$ 

The index $I(A)$ is an even integer. If $I(A) < 0$, we define $\operatorname{Gr}(\bW, A) = 0$. If $I(A) \geq 0$, then $\operatorname{Gr}(\bW, A)$ is defined as follows. We denote by $\cD(A)$ the set of triples $(J, \bfw, \Gamma)$ where $J$ is a smooth $\Omega$-tame almost-complex structure, $\bfw \in \bW^{I(A)/2}$ is a set of $I(A)/2 - k$ points in $\bW$, and $\Gamma$ is a collection of . Then, $\operatorname{Gr}(\bW, A)$ is a count of $J$-holomorphic cycles $\cC$ such that $[\cC] = A$ and $\bfw \subset \operatorname{supp}(C)$, where $(J, \bfw)$ is chosen from a Baire-generic subset of $\cD(A)$. The definition of the count and the proof that it is independent of choices are quite subtle. We will only use the following consequence. 

\begin{prop}\label{prop:gr_existence}
Fix $A \in H_2(\bW; \bZ)$ and assume that $\operatorname{Gr}(\bW, A) \neq 0$. Then there exists a Baire-generic subset $\cD_0 \subseteq \cD(A)$ such that the following holds for any $(J, \bfw) \in \cD_0$. There exists a $J$-holomorphic cycle $\cC = \{(u_i, n_i)\}$ such that:
\begin{enumerate}[(a)]
\item $[\cC] = A$.
\item $\bfw \subset \operatorname{supp}(\cC)$. 
\item For each $i$, let $A_i$ be the class $(u_i)_*[C_i]$. Then we have $0 \leq I(A_i) \leq I(A)$. 
\item For each $i$, the image $u_i(C_i)$ contains exactly $I(A_i)/2$ points from $\bfw$. 
\end{enumerate}
\end{prop}

\begin{proof}
The proposition follows from the definition of $\operatorname{Gr}(\bW, A)$; see the list \cite[($1.4$)]{Taubes96}. 
\end{proof}

\subsection{Improved existence for positive classes} Let $\bW$ be a closed symplectic $4$-manifold with symplectic form $\Omega$ as above. We will deduce an improved version of Propositions~\ref{prop:gr_existence} when $A$ has certain strong positivity properties, that we now discuss. Let $\cE \subset H_2(\bW; \bZ)$ denote the set of $B \in H_2(\bW; \bZ)$ such that $B \cdot B = -1$ and, for any $\Omega$-tame almost-complex structure $J$, there exists an embedded $J$-holomorphic sphere $\mathbb{S} \subset \bW$ representing $B$. For any $A \in H_2(\bW; \bZ)$, define $g(A) := \frac{1}{2}(A \cdot A - \langle c_1(\bW), A \rangle) + 1$. Let $\cP \subseteq H_2(\bW; \bZ)$ denote the set of all classes $A \in H_2(\bW; \bZ)$ satisfying the following conditions:
\begin{itemize}
\item $A \cdot A > 0$.
\item $I(A) \geq 0$ and $g(A) \geq 0$.
\item $\langle [\Omega], A \rangle > 0$.
\item $A \cdot B \geq 0$ for all $B \in \cE$. 
\end{itemize}

Now, we state the improved existence result.  

\begin{prop}\label{prop:gr_existence_improved}
Fix $A \in \cP \subseteq H_2(\bW; \bZ)$ and assume that $\operatorname{Gr}(\bW, A) \neq 0$. There exists a Baire-generic subset $\cD_1 \subseteq \cD(A)$ such that the following holds for any $(J, \bfw) \in \cD_1$. There exists a closed, embedded $J$-holomorphic surface $C \subset \bW$ such that (i) $G(C) = g(A)$, (ii) $[C] = A$ and (iii) $\bfw \subset C$. 
\end{prop}

The proof relies on the following lemma.

\begin{lem}\label{lem:current_restrictions}
Fix $A \in \cP$. There exists a Baire-generic subset $\cD_2 \subseteq \cD(A)$ such that the following holds for any $(J, \bfw) \in \cD_2$. Let $\cC = \{(u_i, n_i)\}$ be a $J$-holomorphic current such that $[\cC] = A$ and $\bfw \subset \operatorname{supp}(\cC)$. Then, $\cC$ satisfies the following additional restrictions:
\begin{enumerate}[(a)]
\item For each $i$, let $A_i$ denote the class represented by $u_i$. Then $I(A_i) \geq 0$. 
\item We have $\sum_i n_i I(A_i) \leq I(A)$. 
\item We have $\sum_i n_i I(A_i) = I(A)$ if and only if $\cC = \{(C, 1)\}$, where $C \subset \bW$ is a closed, embedded $J$-holomorphic surface of genus $g(A)$. 
\end{enumerate}
\end{lem}

\begin{proof}
This is proved by Taubes \cite[Proposition $3.4$]{Taubes11}. 
\end{proof}

We prove Proposition~\ref{prop:gr_existence_improved} by combining Proposition~\ref{prop:gr_existence} and Lemma~\ref{lem:current_restrictions}. 

\begin{proof}[Proof of Proposition~\ref{prop:gr_existence_improved}]
Define $\cD_1 \subseteq \cD(A)$ to be the intersection of the subsets $\cD_0$ and $\cD_2$ from Proposition~\ref{prop:gr_existence} and Lemma~\ref{lem:current_restrictions}. By Proposition~\ref{prop:gr_existence}, there exists a $J$-holomorphic cycle $\cC = \{(u_i, n_i)\}$ such that:
\begin{enumerate}[(a)]
\item $[\cC] = A$.
\item $\bfw \subset \operatorname{supp}(\cC)$. 
\item Write $A_i = u_i(C_i)$ for each $i$. Then for each $i$, we have $0 \leq I(A_i) \leq I(A)$. 
\item For each $i$, the image $u_i(C_i)$ contains exactly $I(A_i)/2$ points from $\bfw$. 
\end{enumerate}

Combining (d) with Lemma~\ref{lem:current_restrictions}(a,b), we have
$$I(A) \leq \sum_i I(A_i) \leq \sum_i n_i I(A_i) \leq I(A).$$

Therefore, we have $I(A) = \sum_i n_i I(A_i)$. The proposition now follows from Lemma~\ref{lem:current_restrictions}(c). 
\end{proof}

Combining Proposition~\ref{prop:gr_existence_improved} with the Gromov compactness theorem, we obtain an existence result for all $(J, \bfw)$. 

\begin{cor}\label{cor:gr_existence_improved}
Fix any $A \in \cP$ such that $\operatorname{Gr}(\bW, A) \neq 0$ and any finite subset $\bfw \subset \bW$ of size at most $I(A)/2$. Then, for any $\Omega$-tame almost-complex structure $J$, there exists a closed, connected Riemann surface $C$ and a $J$-holomorphic curve $u: C \to \bW$ such that (i) $G_a(C) = g(A)$, (ii) $u_*[C] = A$ and (iii) $\bfw \subset u(C)$. 
\end{cor}

\subsection{Closed curves in $\mathbb{CP}^2$}\label{subsec:closed_curves_proof} We prove Proposition~\ref{prop:closed_curves} using Corollary~\ref{cor:gr_existence_improved}. 

\begin{proof}[Proof of Proposition~\ref{prop:closed_curves}]
Let $A \in H_2(\mathbb{CP}^2; \bZ)$ denote the Poincar\'e dual of $\Omega$. Then $A$ is represented by an embedded holomorphic sphere $\bD \subset \mathbb{CP}^2$ of self intersection $1$. Given Corollary~\ref{cor:gr_existence_improved}, it suffices to show that, for any integer $e \geq 1$, we have $eA \in \cP$ and $\operatorname{Gr}(\mathbb{CP}^2, eA) \neq 0$. The former claim follows from the computatations $eA \cdot eA = e^2$, $I(eA) = e^2 + 3e$, $g(eA) = (e-1)(e-2)/2$, $\langle \Omega, eA \rangle = e$, and the fact that $\mathbb{CP}^2$ is positive-definite and contains no surfaces of negative self-intersection. The computation $\operatorname{Gr}(\mathbb{CP}^2, eA) \neq 0$ can be done by explicitly counting $J$-curves for an integrable $J$ (see for example \cite[Theorem B.$3$]{Edtmair22}). 
\end{proof}

\subsection{Closed curves in closed symplectic $4$-manifolds with $b^+ = 1$}\label{subsec:closed_curves_gen_proof} Corollary~\ref{cor:gr_existence_improved} yields existence results for symplectic $4$-manifolds besides $\mathbb{CP}^2$. 

\begin{prop}\label{prop:closed_curves_general}
Let $\bW$ be any closed symplectic $4$-manifold such that $b^+ = 1$ and the symplectic form $\Omega$ has rational cohomology class. Let $A \in H_2(\bW; \mathbb{Q})$ denote the Poincar\'e dual of $[\Omega]$. Fix any positive integer $e > 1 + |\langle c_1(\bW), A \rangle| / (A \cdot A)$ such that $eA \in H_2(\bW; \mathbb{Z})$. Fix any set $\bfw \subset \bW$ such that 
$$\#\bfw \leq e^2(A \cdot A) + e\langle c_1(\bW),A\rangle - 1.$$ 

Then, for any $\Omega$-tame almost-complex structure $\bar{J}$, there exists a closed, connected Riemann surface $C$ and a $\bar{J}$-holomorphic curve $u: C \to \bW$ such that 
(i) $G_a(C) = \frac{1}{2}(e^2(A \cdot A) - e\langle c_1(\bW), A \rangle) + 1$, (ii) $u_*[C] = eA$ and (iii) $\bfw \subset u(C)$.
\end{prop}

\begin{proof}
Let $A \in H_2(\bW; \mathbb{Q})$ denote the Poincar\'e dual of $\Omega$ and let $e$ be any positive integer larger than $1 + |\langle c_1(\bW), A \rangle|/(A \cdot A)$ such that $eA \in H_2(\bW; \mathbb{Z})$. The proof of the proposition will take $3$ steps. 

\noindent\textbf{Step $1$:} This step shows that, for any integer $e \geq 1 + |\langle c_1(\bW), A \rangle|/(A \prod A)$, we have $A \in \cP$. Note that $A \cdot A > 0$. Using the lower bound on $e$, we have
$$I(eA) = e\langle c_1(\mathbb{W}), A \rangle + e^2(A \cdot A) > e(\langle c_1(\mathbb{W}), A \rangle + |\langle c_1(\mathbb{W}), A \rangle| > 0$$
and
$$g(eA) = \frac{1}{2}(-e\langle c_1(\mathbb{W}), A \rangle + e^2(A \cdot A)) + 1 > \frac{e}{2}(-\langle c_1(\bW), A \rangle + |\langle c_1(\bW), A \rangle|) + 1 \geq 0.$$

Since $A$ is dual to $[\Omega]$, we have $\langle \Omega, eA \rangle > 0$. Since $A$ is dual to $[\Omega]$, it must pair positively with any $J$-holomorphic sphere. Thus, we have verified all the necessary conditions and conclude that $eA \in \cP$. 

\noindent\textbf{Step $2$:} This step proves the proposition with the additional assumption that either (i) $b_1(\bW) \neq 2$ or (ii) the cup product on $H^1(\bW; \bZ)$ is nonzero. We claim that $\operatorname{Gr}(\bW, eA) \neq 0$. The claim is proved as follows. In \cite{LL99}, Li--Liu showed that Taubes' ``$\operatorname{SW} = \operatorname{Gr}$'' theorem extends to the $b^+ = 1$ case, so $\operatorname{Gr}(\bW, eA)$ is equal to a corresponding Seiberg--Witten invariant. It follows from a result of the same authors, namely \cite[Lemma $3.3$]{LL01}, that this Seiberg--Witten invariant does not vanish. Now the proposition follows from Corollary~\ref{cor:gr_existence_improved}. 

\noindent\textbf{Step $3$:} This step explains how to prove the proposition in the exceptional case where $b_1 = 2$ and the cup product on $H^1(\bW; \bZ)$ is zero. Choose a pair of disjoint curves $\gamma_1$ and $\gamma_2$ whose homology classes generate $H_1(\bW; \bZ)/\operatorname{Torsion}$. In this case, as explained in \cite[Proposition $1.1$]{Taubes17}, a slightly weaker version of Proposition~\ref{prop:gr_existence_improved} holds. For a generic choice of $J'$, a generic choice of $I(eA)/2 - 1$ points $\bfw'$, and small perturbations $\gamma_1'$ and $\gamma_2'$ of the chosen loops, there exists an embedded $J'$-holomorphic surface $C \subset \bW$ such that (i) $G(C) = g(A)$, (ii) $[C] = A$, (iii) $\bfw' \subset C$ and (iv) $C$ intersects $\gamma_1'$ and $\gamma_2'$. The proposition follows from taking $J' \to J$ and $\bfw' \to \bfw$ and applying the Gromov compactness theorem. 
\end{proof}

%% file: generic.tex
The goal of this appendix is to sketch a proof of the following lemma. 

\begin{lem}\label{lem:generic}
There exists a Baire-generic subset $\cG \subseteq C^\infty(\bR^4)$ such that the following holds. For any $H \in \cG$, there exists a full measure subset $\cQ_* \subseteq \cR_c(H)$ such that for each $s \in \cQ_*$, every closed orbit $\gamma \subset H^{-1}(s)$ is either (i) hyperbolic or (ii) elliptic and Moser stable.
\end{lem}

Lemma~\ref{lem:generic} follows from Takens' \cite{Takens70} perturbation theorem for Hamiltonians. We explain Moser stability, then explain Takens' result, then give the proof. 

\subsection{Moser stability}

We discuss Moser's work in \cite{Moser62}. Fix any area-preserving diffeomorphism $\phi$ of the $2$-disk such that, near $0$, we have a Birkhoff normal form
$$\phi(r, \theta) = (r, \theta + \alpha_0 + \alpha_1 r) + O(r^2)$$
where $(r, \theta)$ denote polar coordinates. Assume that the Birkhoff normal form is nondegenerate, i.e. $\alpha_0$ is irrational and $\alpha_1 \neq 0$. Thus, $\phi$ is, up to higher order terms, an integrable monotone twist map of the disk. The higher order terms, however, significantly affect the dynamics. Nevertheless, Moser proved that some integrability survives; the origin is accumulated by a positive measure family of smooth $\phi$-invariant circles. Since $\alpha_1 \neq 0$, there must exist an infinite sequence of invariant circles accumulating at $0$, with pairwise distinct rotation numbers. Thus, the Poincar\'e--Birkhoff theorem implies that $0$ is accumulated by periodic points of $\phi$.

Fix any $H \in C^\infty(\bR^4)$ and any $s \in \bR$. For any elliptic closed orbit $\gamma \subset H^{-1}(s)$ and any transverse $2$-disk $D \subset H^{-1}(s)$, the Poincar\'e return map is conjugate near $0$ to a Birkhoff normal form as above. We say $\gamma$ is \emph{Moser stable} if the Birkhoff normal form is nondegenerate. 

\subsection{Takens' theorem} Takens \cite{Takens70} proved a powerful jet perturbation theorem for Poincar\'e return maps of closed orbits of Hamiltonians, generalizing a result for $1$-jets by Robinson \cite{Robinson70}. For any $n \geq 1$, let $J^r(n)$ denote the set of $r$-jets at $0$ of smooth symplectic diffeomorphisms of $\bR^{2n}$ that fix $0 \in \bR^{2n}$. We say a subset $Q \subset J^r(n)$ is \emph{invariant} if $\sigma Q \sigma^{-1} = Q$ for all $\sigma \in J^r(n)$. The following lemma is a consequence of Takens' results. 

\begin{lem}[{\cite[Theorem A]{Takens70}}] \label{lem:takens}
Fix $n \geq 1$, $r \geq 1$, $T > 0$, and any finite collection $Q_1, \ldots, Q_N$ of invariant real analytic subvarieties of $J^r(n)$. Then, there exists a Baire-generic subset $\cG_* \subseteq C^\infty(\bR^{2n})$ such that each $H \in \cG_*$ has the following property. For any closed orbit $\gamma$ of $X_H$ with minimal period $< T$, there exists a neighborhood $U$ of $\gamma$ in which all but finitely many closed orbits of minimal period $< T$ have Poincar\'e return maps whose $r$-jet at $0$ does not lie in $\bigcup_{i=1}^N Q_i$. 
\end{lem}

Let us give an informal explanation of how Lemma~\ref{lem:takens} follows from Taken's results. For generic $H$, the closed orbits live in smooth $1$-parameter families. Therefore, we cannot expect every $r$-jet of every closed orbit in a $1$-parameter family to avoid the subvarieties $Q_i$. However, Takens proved that, for generic $H$, each $1$-parameter family of $r$-jets will intersect each $Q_i$ transversely. Therefore, all but finitely many $r$-jets in each family will avoid $\bigcup_{i=1}^N Q_i$. 

\subsection{Moser stability is generic} We sketch a proof of Lemma~\ref{lem:generic}. 

\begin{proof}
There exists a countable collection of invariant real analytic subvarieties $\{Q_i\} \subset J^2(2)$ such that the following holds. If $\gamma$ is a closed orbit of $X_H$ such that the $2$-jet of the Poincar\'e map avoids each $Q_i$, then $\gamma$ is either (i) hyperbolic or (ii) elliptic and Moser stable. 

Fix any $N \geq 1$ and let $\cQ_{N} \subseteq \cR_c(H)$ denote the set of all levels $s$ such that, for each closed orbit $\gamma \subset H^{-1}(s)$ of minimal period $< N$, the $2$-jet of its Poincar\'e return map avoids $Q_1, \ldots, Q_N$. By Lemma~\ref{lem:takens}, there exists a Baire-generic subset $\cG_N$ such that if $H \in \cG_N$, then only a discrete set of closed orbits have $2$-jets intersecting $\bigcup_{i=1}^N Q_i$. It follows that if $H \in \cG_N$, then the complement of $\cQ_N$ in $\cR_c(H)$ is discrete. Therefore, if $H \in \cG_N$, the set $\cQ_N$ has full Lebesgue measure in $\cR_c(H)$. 

Define $\cG_* := \bigcap_{N \geq 1} \cG_N$. If $H \in \cG_*$, then $\cQ_* := \bigcap_{N \geq 1} \cQ_N$ has full Lebesgue measure in $\cR_c(H)$. For any $s \in \cQ_*$, every closed orbit $\gamma \subset H^{-1}(s)$ is either (i) hyperbolic or (ii) elliptic and Moser stable. 
\end{proof}

%% file: r4.bbl
\newcommand{\etalchar}[1]{$^{#1}$}
\begin{thebibliography}{CGHHL23}

\bibitem[AB16]{AB16}
Luca Asselle and Gabriele Benedetti.
\newblock The {L}usternik-{F}et theorem for autonomous {T}onelli {H}amiltonian
  systems on twisted cotangent bundles.
\newblock {\em J. Topol. Anal.}, 8(3):545--570, 2016.

\bibitem[AMMP17]{AMMP17}
Alberto Abbondandolo, Leonardo Macarini, Marco Mazzucchelli, and Gabriel~P.
  Paternain.
\newblock Infinitely many periodic orbits of exact magnetic flows on surfaces
  for almost every subcritical energy level.
\newblock {\em J. Eur. Math. Soc. (JEMS)}, 19(2):551--579, 2017.

\bibitem[AMP15]{AMP15}
Alberto Abbondandolo, Leonardo Macarini, and Gabriel~P. Paternain.
\newblock On the existence of three closed magnetic geodesics for subcritical
  energies.
\newblock {\em Comment. Math. Helv.}, 90(1):155--193, 2015.

\bibitem[Arn88]{Arnold88}
V.~I. Arnold.
\newblock On some problems in symplectic topology.
\newblock In {\em Topology and geometry---{R}ohlin {S}eminar}, volume 1346 of
  {\em Lecture Notes in Math.}, pages 1--5. Springer, Berlin, 1988.

\bibitem[BEH{\etalchar{+}}03]{BEHWZ03}
F.~Bourgeois, Y.~Eliashberg, H.~Hofer, K.~Wysocki, and E.~Zehnder.
\newblock Compactness results in symplectic field theory.
\newblock {\em Geom. Topol.}, 7:799--888, 2003.

\bibitem[BPS03]{BPS03}
Paul Biran, Leonid Polterovich, and Dietmar Salamon.
\newblock Propagation in {H}amiltonian dynamics and relative symplectic
  homology.
\newblock {\em Duke Math. J.}, 119(1):65--118, 2003.

\bibitem[CDR23]{CDR23}
Vincent Colin, Pierre Dehornoy, and Ana Rechtman.
\newblock On the existence of supporting broken book decompositions for contact
  forms in dimension 3.
\newblock {\em Invent. Math.}, 231(3):1489--1539, 2023.

\bibitem[CGGM23]{CGGM23}
Erman Cineli, Viktor Ginzburg, Ba\c{s}ak G\"{u}rel, and Marco Mazzucchelli.
\newblock Invariant sets and hyperbolic closed {R}eeb orbits.
\newblock {\em arXiv preprint arXiv:2309.04576v1}, 2023.

\bibitem[CGH16]{CGH16}
Daniel Cristofaro-Gardiner and Michael Hutchings.
\newblock From one {R}eeb orbit to two.
\newblock {\em J. Differential Geom.}, 102(1):25--36, 2016.

\bibitem[CGHHL23]{CGHHL23}
Dan Cristofaro-Gardiner, Umberto Hryniewicz, Michael Hutchings, and Hui Liu.
\newblock Proof of {H}ofer--{W}ysocki--{Z}ehnder's two or infinity conjecture.
\newblock {\em arXiv preprint arXiv:2310.07636v2}, 2023.

\bibitem[CGHP19]{CGHP19}
Dan Cristofaro-Gardiner, Michael Hutchings, and Daniel Pomerleano.
\newblock Torsion contact forms in three dimensions have two or infinitely many
  {R}eeb orbits.
\newblock {\em Geom. Topol.}, 23(7):3601--3645, 2019.

\bibitem[CGK04]{CGK04}
Kai Cieliebak, Viktor~L. Ginzburg, and Ely Kerman.
\newblock Symplectic homology and periodic orbits near symplectic submanifolds.
\newblock {\em Comment. Math. Helv.}, 79(3):554--581, 2004.

\bibitem[CGP24]{CGP24}
Dan Cristofaro-Gardiner and Rohil Prasad.
\newblock A global {L}e {C}alvez--{Y}occoz property for area-preserving surface
  diffeomorphisms and three-dimensional {R}eeb flows.
\newblock {\em arXiv preprint arXiv:2401.14445v1}, 2024.

\bibitem[Con06]{Contreras06}
Gonzalo Contreras.
\newblock The {P}alais-{S}male condition on contact type energy levels for
  convex {L}agrangian systems.
\newblock {\em Calc. Var. Partial Differential Equations}, 27(3):321--395,
  2006.

\bibitem[DLL{\etalchar{+}}24]{DLLQW24}
Huagui Duan, Hui Liu, Yiming Long, Zihao Qi, and Wei Wang.
\newblock Three closed characteristics on non-degenerate star-shaped
  hypersurfaces in $\mathbb{R}^6$.
\newblock {\em arXiv preprint 2401.13259v1}, 2024.

\bibitem[DW21]{DW21}
Aleksander Doan and Thomas Walpuski.
\newblock Castelnuovo's bound and rigidity in almost complex geometry.
\newblock {\em Adv. Math.}, 379:Paper No. 107550, 23, 2021.

\bibitem[Edt22]{Edtmair22}
Oliver Edtmair.
\newblock An elementary alternative to {PFH} spectral invariants.
\newblock {\em arXiv preprint arXiv:2207.12553v1}, 2022.

\bibitem[FH94]{FH94}
A.~Floer and H.~Hofer.
\newblock Symplectic homology. {I}. {O}pen sets in {${\bf C}^n$}.
\newblock {\em Math. Z.}, 215(1):37--88, 1994.

\bibitem[FH22]{FH22}
Joel~W. Fish and Helmut H.~W. Hofer.
\newblock Almost existence from the feral perspective and some questions.
\newblock {\em Ergodic Theory Dynam. Systems}, 42(2):792--834, 2022.

\bibitem[FH23]{FH23}
Joel~W. Fish and Helmut Hofer.
\newblock Feral curves and minimal sets.
\newblock {\em Ann. of Math. (2)}, 197(2):533--738, 2023.

\bibitem[Fis11]{Fish11}
Joel~W. Fish.
\newblock Target-local {G}romov compactness.
\newblock {\em Geom. Topol.}, 15(2):765--826, 2011.

\bibitem[Fra99]{Franks99}
John Franks.
\newblock The {C}onley index and non-existence of minimal homeomorphisms.
\newblock In {\em Proceedings of the {C}onference on {P}robability, {E}rgodic
  {T}heory, and {A}nalysis ({E}vanston, {IL}, 1997)}, volume~43, pages
  457--464, 1999.

\bibitem[FS07]{FS07}
Urs Frauenfelder and Felix Schlenk.
\newblock Hamiltonian dynamics on convex symplectic manifolds.
\newblock {\em Israel J. Math.}, 159:1--56, 2007.

\bibitem[GG03]{GG03}
Viktor~L. Ginzburg and Ba\c{s}ak~Z. G\"{u}rel.
\newblock A {$C^2$}-smooth counterexample to the {H}amiltonian {S}eifert
  conjecture in {$\Bbb R^4$}.
\newblock {\em Ann. of Math. (2)}, 158(3):953--976, 2003.

\bibitem[GG04]{GG04}
Viktor~L. Ginzburg and Ba\c{s}ak~Z. G\"{u}rel.
\newblock Relative {H}ofer-{Z}ehnder capacity and periodic orbits in twisted
  cotangent bundles.
\newblock {\em Duke Math. J.}, 123(1):1--47, 2004.

\bibitem[GG09]{GG09}
Viktor~L. Ginzburg and Ba\c{s}ak~Z. G\"{u}rel.
\newblock Periodic orbits of twisted geodesic flows and the {W}einstein-{M}oser
  theorem.
\newblock {\em Comment. Math. Helv.}, 84(4):865--907, 2009.

\bibitem[GG18]{GG18}
Viktor~L. Ginzburg and Ba\c{s}ak~Z. G\"{u}rel.
\newblock Hamiltonian pseudo-rotations of projective spaces.
\newblock {\em Invent. Math.}, 214(3):1081--1130, 2018.

\bibitem[GHHM13]{GHHM13}
Viktor~L. Ginzburg, Doris Hein, Umberto~L. Hryniewicz, and Leonardo Macarini.
\newblock Closed {R}eeb orbits on the sphere and symplectically degenerate
  maxima.
\newblock {\em Acta Math. Vietnam.}, 38(1):55--78, 2013.

\bibitem[Gin87]{Ginzburg87}
Viktor~L. Ginzburg.
\newblock New generalizations of poincar{\'e}'s geometric theorem.
\newblock {\em Functional Analysis and Its Applications}, 21:100--106, 1987.

\bibitem[Gin95]{Ginzburg95}
Viktor~L. Ginzburg.
\newblock An embedding {$S^{2n-1}\to {\bf R}^{2n}$}, {$2n-1\geq 7$}, whose
  {H}amiltonian flow has no periodic trajectories.
\newblock {\em Internat. Math. Res. Notices}, (2):83--97, 1995.

\bibitem[Gin96]{Ginzburg96}
Viktor~L. Ginzburg.
\newblock On the existence and non-existence of closed trajectories for some
  {H}amiltonian flows.
\newblock {\em Math. Z.}, 223(3):397--409, 1996.

\bibitem[Gin05]{Ginzburg05}
Viktor~L. Ginzburg.
\newblock The {W}einstein conjecture and theorems of nearby and almost
  existence.
\newblock In {\em The breadth of symplectic and {P}oisson geometry}, volume 232
  of {\em Progr. Math.}, pages 139--172. Birkh\"{a}user Boston, Boston, MA,
  2005.

\bibitem[GN15]{GN15}
Viktor~L. Ginzburg and C\'{e}sar~J. Niche.
\newblock A remark on unique ergodicity and the contact type condition.
\newblock {\em Arch. Math. (Basel)}, 105(6):585--592, 2015.

\bibitem[Her98]{Herman98}
Michael Herman.
\newblock Some open problems in dynamical systems.
\newblock In {\em Proceedings of the {I}nternational {C}ongress of
  {M}athematicians, {V}ol. {II} ({B}erlin, 1998)}, pages 797--808, 1998.

\bibitem[Her99]{Herman99}
Michel~R. Herman.
\newblock Examples of compact hypersurfaces in {${\bf R}^{2p},\ 2p\ge6$}, with
  no periodic orbits.
\newblock In {\em Hamiltonian systems with three or more degrees of freedom
  ({S}'{A}gar\'{o}, 1995)}, volume 533 of {\em NATO Adv. Sci. Inst. Ser. C:
  Math. Phys. Sci.}, page 126. Kluwer Acad. Publ., Dordrecht, 1999.

\bibitem[HLS97]{HLS97}
Helmut Hofer, V\'{e}ronique Lizan, and Jean-Claude Sikorav.
\newblock On genericity for holomorphic curves in four-dimensional
  almost-complex manifolds.
\newblock {\em J. Geom. Anal.}, 7(1):149--159, 1997.

\bibitem[Hut14]{ech_notes}
Michael Hutchings.
\newblock Lecture notes on embedded contact homology.
\newblock In {\em Contact and symplectic topology}, volume~26 of {\em Bolyai
  Soc. Math. Stud.}, pages 389--484. J\'{a}nos Bolyai Math. Soc., Budapest,
  2014.

\bibitem[HV92]{HV92}
H.~Hofer and C.~Viterbo.
\newblock The {W}einstein conjecture in the presence of holomorphic spheres.
\newblock {\em Comm. Pure Appl. Math.}, 45(5):583--622, 1992.

\bibitem[HWZ98]{HWZ98}
H.~Hofer, K.~Wysocki, and E.~Zehnder.
\newblock The dynamics on three-dimensional strictly convex energy surfaces.
\newblock {\em Ann. of Math. (2)}, 148(1):197--289, 1998.

\bibitem[HWZ03]{HWZ03}
H.~Hofer, K.~Wysocki, and E.~Zehnder.
\newblock Finite energy foliations of tight three-spheres and {H}amiltonian
  dynamics.
\newblock {\em Ann. of Math. (2)}, 157(1):125--255, 2003.

\bibitem[HZ87]{HZ87}
H.~Hofer and E.~Zehnder.
\newblock Periodic solutions on hypersurfaces and a result by {C}. {V}iterbo.
\newblock {\em Invent. Math.}, 90(1):1--9, 1987.

\bibitem[Iri15]{Irie15}
Kei Irie.
\newblock Dense existence of periodic {R}eeb orbits and {ECH} spectral
  invariants.
\newblock {\em J. Mod. Dyn.}, 9:357--363, 2015.

\bibitem[Kat73]{Katok73}
A.~B. Katok.
\newblock Ergodic perturbations of degenerate integrable {H}amiltonian systems.
\newblock {\em Izv. Akad. Nauk SSSR Ser. Mat.}, 37:539--576, 1973.

\bibitem[Ker99]{Kerman99}
Ely Kerman.
\newblock Periodic orbits of {H}amiltonian flows near symplectic critical
  submanifolds.
\newblock {\em Internat. Math. Res. Notices}, (17):953--969, 1999.

\bibitem[LCY97]{LCY97}
Patrice Le~Calvez and Jean-Christophe Yoccoz.
\newblock Un th\'{e}or\`eme d'indice pour les hom\'{e}omorphismes du plan au
  voisinage d'un point fixe.
\newblock {\em Ann. of Math. (2)}, 146(2):241--293, 1997.

\bibitem[LL99]{LL99}
Tian-Jun Li and Ai-Ko Liu.
\newblock The equivalence between {${\rm SW}$} and {${\rm Gr}$} in the case
  where {$b^+=1$}.
\newblock {\em Internat. Math. Res. Notices}, (7):335--345, 1999.

\bibitem[LL01]{LL01}
Tian-Jun Li and Ai-Ko Liu.
\newblock Uniqueness of symplectic canonical class, surface cone and symplectic
  cone of 4-manifolds with {$B^+=1$}.
\newblock {\em J. Differential Geom.}, 58(2):331--370, 2001.

\bibitem[Lu98]{Lu98}
Guangcun Lu.
\newblock The {W}einstein conjecture on some symplectic manifolds containing
  the holomorphic spheres.
\newblock {\em Kyushu J. Math.}, 52(2):331--351, 1998.

\bibitem[Lu00]{Lu98correction}
Guangcun Lu.
\newblock Correction to: ``{T}he {W}einstein conjecture on some symplectic
  manifolds containing the holomorphic spheres'' [{K}yushu {J}. {M}ath. {\bf
  52} (1998), no. 2, 331--351; {MR}1645455 (99g:58054)].
\newblock {\em Kyushu J. Math.}, 54(1):181--182, 2000.

\bibitem[Mac04]{Macarini04}
Leonardo Macarini.
\newblock Hofer-{Z}ehnder capacity and {H}amiltonian circle actions.
\newblock {\em Commun. Contemp. Math.}, 6(6):913--945, 2004.

\bibitem[McD87]{McDuff87}
Dusa McDuff.
\newblock Applications of convex integration to symplectic and contact
  geometry.
\newblock {\em Ann. Inst. Fourier (Grenoble)}, 37(1):107--133, 1987.

\bibitem[McD91]{McDuff91}
Dusa McDuff.
\newblock The local behaviour of holomorphic curves in almost complex
  {$4$}-manifolds.
\newblock {\em J. Differential Geom.}, 34(1):143--164, 1991.

\bibitem[McM96]{McMullen96}
Curtis~T. McMullen.
\newblock {\em Renormalization and 3-manifolds which fiber over the circle},
  volume 142 of {\em Annals of Mathematics Studies}.
\newblock Princeton University Press, Princeton, NJ, 1996.

\bibitem[Mos62]{Moser62}
J.~Moser.
\newblock On invariant curves of area-preserving mappings of an annulus.
\newblock {\em Nachr. Akad. Wiss. G\"{o}ttingen Math.-Phys. Kl. II},
  1962:1--20, 1962.

\bibitem[Mos76]{Moser76}
J.~Moser.
\newblock Periodic orbits near an equilibrium and a theorem by {A}lan
  {W}einstein.
\newblock {\em Comm. Pure Appl. Math.}, 29(6):724--747, 1976.

\bibitem[MS01]{MS01}
Dusa McDuff and Jennifer Slimowitz.
\newblock Hofer-{Z}ehnder capacity and length minimizing {H}amiltonian paths.
\newblock {\em Geom. Topol.}, 5:799--830, 2001.

\bibitem[MS05]{MS05}
Leonardo Macarini and Felix Schlenk.
\newblock A refinement of the {H}ofer-{Z}ehnder theorem on the existence of
  closed characteristics near a hypersurface.
\newblock {\em Bull. London Math. Soc.}, 37(2):297--300, 2005.

\bibitem[Pra23a]{Prasad23a}
Rohil Prasad.
\newblock Invariant probability measures from pseudoholomorphic curves {I}.
\newblock {\em J. Mod. Dyn.}, 19:31--74, 2023.

\bibitem[Pra23b]{Prasad23b}
Rohil Prasad.
\newblock Invariant probability measures from pseudoholomorphic curves {II}:
  {P}seudoholomorphic curve constructions.
\newblock {\em J. Mod. Dyn.}, 19:75--160, 2023.

\bibitem[Rab78]{Rabinowitz78}
Paul~H. Rabinowitz.
\newblock Periodic solutions of {H}amiltonian systems.
\newblock {\em Comm. Pure Appl. Math.}, 31(2):157--184, 1978.

\bibitem[Rab87]{Rabinowitz87}
Paul~H. Rabinowitz.
\newblock On a theorem of {H}ofer and {Z}ehnder.
\newblock In {\em Periodic solutions of {H}amiltonian systems and related
  topics ({I}l {C}iocco, 1986)}, volume 209 of {\em NATO Adv. Sci. Inst. Ser.
  C: Math. Phys. Sci.}, pages 245--253. Reidel, Dordrecht, 1987.

\bibitem[Rob70]{Robinson70}
R.~Clark Robinson.
\newblock Generic properties of conservative systems.
\newblock {\em Amer. J. Math.}, 92:562--603, 1970.

\bibitem[Sal06]{Salazar06}
Jos\'{e}~M. Salazar.
\newblock Instability property of homeomorphisms on surfaces.
\newblock {\em Ergodic Theory Dynam. Systems}, 26(2):539--549, 2006.

\bibitem[Sch06]{Schlenk06}
Felix Schlenk.
\newblock Applications of {H}ofer's geometry to {H}amiltonian dynamics.
\newblock {\em Comment. Math. Helv.}, 81(1):105--121, 2006.

\bibitem[Str90]{Struwe90}
Michael Struwe.
\newblock Existence of periodic solutions of {H}amiltonian systems on almost
  every energy surface.
\newblock {\em Bol. Soc. Brasil. Mat. (N.S.)}, 20(2):49--58, 1990.

\bibitem[Tak70]{Takens70}
Floris Takens.
\newblock Hamiltonian systems: {G}eneric propeties of closed orbits and local
  perturbations.
\newblock {\em Math. Ann.}, 188:304--312, 1970.

\bibitem[Tau96]{Taubes96}
Clifford~Henry Taubes.
\newblock Counting pseudo-holomorphic submanifolds in dimension {$4$}.
\newblock {\em J. Differential Geom.}, 44(4):818--893, 1996.

\bibitem[Tau98]{Taubes98}
Clifford~Henry Taubes.
\newblock The structure of pseudo-holomorphic subvarieties for a degenerate
  almost complex structure and symplectic form on {$S^1\times B^3$}.
\newblock {\em Geom. Topol.}, 2:221--332, 1998.

\bibitem[Tau09]{Taubes09}
Clifford~Henry Taubes.
\newblock An observation concerning uniquely ergodic vector fields on
  3-manifolds.
\newblock {\em J. G\"{o}kova Geom. Topol. GGT}, 3:9--21, 2009.

\bibitem[Tau11]{Taubes11}
Clifford~Henry Taubes.
\newblock Tamed to compatible: symplectic forms via moduli space integration.
\newblock {\em J. Symplectic Geom.}, 9(2):161--250, 2011.

\bibitem[Tau17]{Taubes17}
Clifford~Henrry Taubes.
\newblock Tamed to compatible when $b^{2+} = 1$ and $b^1 = 2$.
\newblock {\em arXiv preprint arXiv:1708.04279}, 2017.

\bibitem[Ush09]{Usher09}
Michael Usher.
\newblock Floer homology in disk bundles and symplectically twisted geodesic
  flows.
\newblock {\em J. Mod. Dyn.}, 3(1):61--101, 2009.

\bibitem[Vit87]{Viterbo87}
Claude Viterbo.
\newblock A proof of {W}einstein's conjecture in {${\bf R}^{2n}$}.
\newblock {\em Ann. Inst. H. Poincar\'{e} Anal. Non Lin\'{e}aire},
  4(4):337--356, 1987.

\bibitem[Wei73]{Weinstein73}
Alan Weinstein.
\newblock Normal modes for nonlinear {H}amiltonian systems.
\newblock {\em Invent. Math.}, 20:47--57, 1973.

\bibitem[Wei78]{Weinstein78}
Alan Weinstein.
\newblock Periodic orbits for convex {H}amiltonian systems.
\newblock {\em Ann. of Math. (2)}, 108(3):507--518, 1978.

\end{thebibliography}
